\documentclass[reqno,11pt,centertags]{amsart}

\usepackage{amsmath,amsthm,amscd,amssymb,latexsym,upref,enumerate}
\input epsf
\usepackage{epsfig}
\usepackage[T2A,OT1]{fontenc}
\usepackage[ot2enc]{inputenc}
\usepackage[russian,english]{babel}
\usepackage{epic,eepicemu}
\usepackage{tikz}
\usepackage{mathrsfs}  
\newcommand{\bbD}{{\mathbb{D}}}
\newcommand{\bbN}{{\mathbb{N}}}
\newcommand{\bbR}{{\mathbb{R}}}

\newcommand{\bbZ}{{\mathbb{Z}}}
\newcommand{\bbC}{{\mathbb{C}}}
\newcommand{\bbQ}{{\mathbb{Q}}}
\newcommand{\bbH}{{\mathbb{H}}}

\newcommand{\cF}{{\mathcal{F}}}
\newcommand{\cL}{{\mathcal{L}}}
\newcommand{\cP}{{\mathcal{P}}}

\newcommand{\cJ}{{\mathcal{J}}}

\newcommand{\cC}{{\mathcal{C}}}

\newcommand{\cU}{{\mathcal{U}}}
\newcommand{\cK}{{\mathcal{K}}}
\newcommand{\cN}{{\mathcal{N}}}
\newcommand{\cT}{{\mathcal{T}}}
\newcommand{\cUnew}{{\mathcal{V}}}

\newcommand{\sa}{{\mathsf{a}}}
\renewcommand{\sb}{{\mathsf{b}}}
\newcommand{\sE}{{\mathsf{E}}}
\newcommand{\ol}{\overline}

\newcommand{\cD}{{\mathcal{D}}}
\newcommand{\fB}{{\mathfrak{B}}}
\newcommand{\fD}{{\mathfrak{D}}}

\newcommand{\fa}{{\mathfrak{a}}}

\newcommand{\e}{{\epsilon}}
\newcommand{\vk}{{\varkappa}}

\newcommand{\fA}{{\mathfrak{A}}}

\newcommand{\ve}{{\varepsilon}}

\newcommand{\fj}{{\mathfrak{j}}}

\newcommand{\z}{\zeta}
\newcommand{\g}{\gamma}
\renewcommand{\Re}{\text{\rm Re\ }}
\renewcommand{\Im}{\text{\rm Im\ }}

\newcommand{\tr}{\text{\rm tr}}

\newcommand{\loc}{\text{\rm loc}}
\newcommand{\supp}{\operatorname{supp}}
\allowdisplaybreaks \numberwithin{equation}{section}
\newtheorem{theorem}{Theorem}[section]
\newtheorem*{theorem-non}{Theorem}

\newtheorem{lemma}[theorem]{Lemma}
\newtheorem{proposition}[theorem]{Proposition}
\newtheorem{corollary}[theorem]{Corollary}

\theoremstyle{definition}
\newtheorem{definition}[theorem]{Definition}

\newtheorem{remark}[theorem]{Remark}

\newtheorem{example}[theorem]{Example}

\date{\today}

\title{New universality classes associated to fractals}

\author[B.\ Eichinger]{Benjamin Eichinger}

\address{B. Eichinger: School of Mathematical Sciences, Lancaster University, Lancaster LA1 4YF, United Kingdom and Institute of Analysis and Scientific Computing, TU Wien, 1040 Wien,
Austria}

\email{b.eichinger@lancaster.ac.uk}

\thanks{B.\ E.\ was supported by the Austrian Science Fund FWF, project no: P33885}

\author[A.\ Kheifets]{Alexander Kheifets}

\address{A. Kheifets: Department of Mathematics and Statistics, University of Massachusetts Lowell, Lowell, MA~01854, USA}

\email{{alexander\_kheifets@uml.edu}}

\author[M.\ Luki\'c]{Milivoje Luki\'c}

\address{M. Luki\'c: Department of Mathematics, Emory University, Atlanta, GA~30322, USA}

\email{milivoje.lukic@emory.edu}

\thanks{M. L.\ was supported in part by NSF grants DMS--2154563/2626207 and DMS--2453758/2626193.}

\author[P.\ Yuditskii]{Peter Yuditskii}

\address{P. Yuditskii: Institut f{\"u}r Analysis,
	Johannes Kepler Universit{\"a}t, 
	Altenberger Strasse 69,
	4040 Linz, Austria}

\email{peter.yuditskii@gmail.com}

\thanks{P.\ Y.\ was supported by the Austrian Science Fund FWF, project no: P34414}



\begin{document}
	\maketitle

\begin{abstract}
The local behavior of zeros of orthogonal polynomials is determined by the local scaling behavior of Christoffel--Darboux (CD) kernels. All previously studied behaviors are described by scaling limits, and different values of the limit kernel correspond to different universality classes. In this paper, we describe new universality classes in which instead of a single limit kernel, there is a limit cycle. These are naturally suited to Cantor spectra and to fractal behaviors of the measure. We show that these new phenomena occur for two canonical models with singular measures: the middle third Cantor measure and the balanced/equilibrium measure on a real Julia set of an expanding polynomial.  In particular, this is the first result on the local behavior of CD kernels for an almost periodic operator with singular spectrum.

As a complementary result, we describe the asymptotics of the scaling function, and the Christoffel function, at a fixed point of a quadratic iteration. This is the first such analysis for an almost periodic model with singular spectrum. It allows us to conclude that, at the fixed point, the local scaling of zeros of the polynomial of degree $n$ is precisely of order $n^{-1/\alpha}$, where $\alpha$ is the local dimension of the measure. We also study the limit chain in this case, and prove that it can be parametrized by its asymptotics with respect to a Martin function (so-called $M$-type); this is the first result of this kind for a chain with a singular measure.
\end{abstract}

\section{Introduction}

	Let $\mu$ be a probability measure on $\bbR$ which is not supported on finitely many points, and has finite moments: 
	\begin{align}\label{intro:eq9}
		\int \lvert x\rvert^n d\mu(x)<\infty, \qquad \forall n\in\bbN_0.
	\end{align}
By applying the Gram--Schmidt process to the sequence of monomials $(z^k)_{k=0}^\infty$ in $L^2(\bbR,d\mu)$, one obtains the sequence of orthogonal polynomials $(p_k(z))_{k=0}^\infty$. The polynomial $p_k$ has $k$ real simple zeros, and the global and local asymptotics of the zeros are questions of classical interest \cite{Freud69,SimonEquilibrium,SimonSzego,StahlTotik,Szego}. The local behavior is studied using the Christoffel--Darboux (CD) kernels
\[
K(n,z,w) = \sum_{k=0}^{n-1}p_{k}(z)\overline{p_{k}(w)}.
\]
There is a long history of studying their scaling behavior around a point $\xi \in \bbR$ in regimes where a scaling limit exists, i.e.\
\begin{equation}\label{intro:eq18}
\lim\limits_{n\to\infty}\frac{K\left(n,\xi+\frac{z}{r(n)},\xi+\frac{w}{r(n)}\right)}{K\left(n,\xi,\xi\right)}= K_\infty (z,w)
\end{equation}
for some scaling sequence $r(n) \to \infty$ and nonconstant limit kernel $K_\infty$. The most studied phenomenon is so-called bulk universality, which corresponds to the sine kernel limit $K_\infty(z,w) = \sin(\pi(z - \ol w))/ (\pi (z - \ol w))$; cf. \cite{AvilaLastSimon,BSing,Breuer2021ConstrApprox,BreuerLastSimon,DKMVZ1,DKMVZ2,DKMVZ3,DeiftOPandRM,EichLukSim,EichLukWor,Findley08,LubinskyAnnals,LubinskyJFA09,LubSurvey,MaltsevCMP,Mitkovski,SimonTwoExt08,TotikUniv09,Totik16} for a very partial list of references. In particular, for measures corresponding to a determinate moment problem, bulk universality at the scale $r(n) = c K(n,\xi,\xi)$ for some $c \in (0,\infty)$, together with the condition
\begin{equation}\label{eqnNevai}
\lim_{n\to\infty} \frac{K(n+1,\xi,\xi)}{K(n,\xi,\xi)} = 1,
\end{equation}
is equivalent  to $\mu([\xi, \xi \pm \epsilon)) /  \epsilon \to c$ as $\epsilon \downarrow 0$ \cite{EichLukWor}. Other limit kernels are also characterized in terms of local behaviors of the measure $\mu$ at $\xi$, and these are viewed as different universality classes, see e.g. \cite{Danka17,EichLukWor,MorenoFinkelSousaConstr,KMVV04,KvL02}. For instance, Bessel kernels correspond to hard edge universality. In the non-Hermitian setting, other kernels have been studied for instance in \cite{AmeurHedenmalmMakarov2011,CronvallWennman,HedenmalmWennman2021}.

Known universality classes arose from the study of measures on intervals (although they are now known to hold under local conditions which don't require intervals in the support \cite{EichLukSim,EichLukWor}), and they are not well suited to phenomena common in spectral theory, such as self-similarity and Cantor spectra \cite{DamanikFillman}.

The first goal of this paper is to describe new scaling behaviors associated with self-similar structures. Instead of a scaling limit, there will be a family of kernels
$K_s(z,w)$ continuous in $s \in (0,\infty)$ with respect to locally uniform convergence, with $K_{s\rho} = K_s$ for some $\rho > 1$, an increasing sequence of $r(n) \in (0,\infty)$, and $\beta \in (-1,1)$ such that
\begin{equation}\label{eqnLimitCycleCDkernel}
\frac 1{ r(n)^{1+\beta}} K\left(n,\xi+\frac{z}{r(n)},\xi+\frac{w}{r(n)}\right) - K_{r(n)}(z,w) \to 0,\qquad n \to\infty.
\end{equation}
This is a common way to describe the asymptotics of an object which doesn't converge but asymptotically approaches periodic or almost periodic behavior, compare \cite[Theorem 12.3]{Widom69}. We will provide a sufficient condition for \eqref{eqnLimitCycleCDkernel} in terms of Weyl functions
\begin{equation}\label{intro:eq20}
m_\mu(z) = \int \frac 1{x- z}\,d\mu(x), \quad z \in\bbC_+
\end{equation}
with the notation
\[
\bbC_\pm = \{ z \in \bbC \mid \pm \Im z > 0 \}.
\]
We describe \eqref{eqnLimitCycleCDkernel} as limit cycle behavior, due to the multiplicatively periodic condition $K_{s\rho} = K_s$, and view it as a new type of universality class. Even individually, the kernels $K_s$ have not previously appeared in the literature. We also note that the kernels $K_s$ for different $s \in [1,\rho)$ are not merely rescalings of each other. Whereas previously studied limit kernels are expressed in terms of Bessel or hypergeometric functions, the limit cycle corresponds to a self-similar canonical system. The solutions of this canonical system are the special functions associated with the universality class.

To formulate corollaries about local zero spacing, we will label the zeros of $p_n$ relative to $\xi$ by $\xi_k^{(n)}$, so that
\[
\dots < \xi_{-1}^{(n)} < \xi_0^{(n)} \le \xi < \xi_1^{(n)} < \xi_2^{(n)} < \dots
\]
Of course, for each $n$, only finitely many of these are well-defined. When we write a limit statement involving some $\xi_k^{(n)}$, it is implied as part of that statement that $\xi_k^{(n)}$ exists for all large enough $n$.
We will use the notation $f \lesssim g$ if $\limsup f/g \le C$ for some $C < \infty$ and $f \asymp g$ if $f \lesssim g$ and $g\lesssim f$. Typical of our results will be that for $k < l$,
\begin{equation}\label{eqn:distancebetweenzeros2}
 \xi_l^{(n)} - \xi_k^{(n)} \asymp K(n,\xi,\xi)^{-1/(1+\beta)}, \qquad n \to\infty.
\end{equation}
The fine structure of zeros for singular measures is a well-known open problem, popularized by Simon in \cite[Problem 1.4]{vanAssche10}. The only general result is a lower bound in terms of transfer matrices \cite[Theorems 2.1, 2.2]{LastSimon}; applied to our setting, this would give
\[
 \xi_l^{(n)} - \xi_k^{(n)} \gtrsim K(n,\xi,\xi)^{-(1+\lvert \beta\rvert)/(1+\beta)}, \qquad n \to\infty,
\]
so \eqref{eqn:distancebetweenzeros2} improves the exponent and provides a matching upper bound.

The second goal is to show that this behavior occurs in two prominent examples.  The first is the middle third Cantor measure, whose orthogonal polynomials have been studied as a natural model for the zero spacing for singular measures \cite{KrugerSimon15}. The second are balanced/equilibrium measures of real Julia sets of expanding polynomials; these measures are the canonical model in inverse spectral theory of almost periodic Jacobi operators with singular continuous spectrum \cite{BarnGerHarr82,BaGerHarr85,BellBeMou82,BelGerVolYu05,PeVolYu06,NPVY}.

In spectral theory, absolutely continuous spectrum is often encountered on intervals, whereas singular spectrum is often encountered on Cantor sets, but there are many counterexamples to this. Accordingly, there is no formal connection between the spectral type and the universality class, and we will provide an example of limit cycle behavior in an a.c.\ measure. Still, whereas bulk universality holds a.e.\ on the a.c.\ part of the measure \cite{AvilaLastSimon,EichLukSim}, singular measures are harder to study. Our application to balanced measures of Julia sets is the first result on scaling behavior of the CD kernels for almost periodic operators with singular spectrum.

In our examples, the behavior \eqref{eqnLimitCycleCDkernel} holds with a sequence such that $r(n) \to \infty$ and  $r(n+1) / r(n) \to 1$ as $n\to\infty$.  This ensures that every kernel in the limit cycle is indeed an accumulation point of the rescaled CD kernels.

The third goal of our paper is to study the growth rate of the scaling function $r(n)$, which is tightly related to the sequence $K(n,\xi,\xi)$: it follows already from \eqref{eqnLimitCycleCDkernel} evaluated at $z=w=0$ that
\[
K(n,\xi,\xi) \asymp r(n)^{1+\beta}.
\]
The reciprocal of $K(n,\xi,\xi)$ is known as the Christoffel function. Its asymptotics are a question of classical interest, traditionally studied on the absolutely continuous spectrum \cite{BessEntr,EichJAnalyse,GubkinMNT,MaNevTot,TotikAdv14}. 
For Stahl--Totik regular measures, under a Lebesgue point assumption for the density and a local Szeg\H o condition, $K(n,\xi,\xi) \sim n$ as $n\to\infty$ \cite{MaNevTot,TotikAdv14}. Whereas the scaling behavior of CD kernels depends on the local behavior of the measure, the asymptotic behavior of $K(n,\xi,\xi)$ depends also on the global properties of the measure $\mu$; see e.g. the discussion preceding Theorem 1.2 in \cite{TotikAdv14} for the role of Stahl--Totik regularity. The asymptotics are also known to be different for measures of unbounded support \cite{DeiftOPandRM}.

Singular measures have much more varied behavior, and an analysis of Christoffel function is expected to be highly model-dependent. We develop this analysis precisely for balanced measures on Julia sets of quadratic polynomials, at a fixed point of the quadratic iteration. We prove that
\[
K(n,\xi,\xi) \asymp n, \qquad n\to\infty.
\]
In fact, we will describe the oscillations in $K(n,\xi,\xi) / n$ and show that 
\begin{align}
K(n,\xi,\xi) \sim n b(n), \qquad n\to\infty \label{eqnKnoscillations} \\
r(n) \sim n^\kappa c(n), \qquad n \to\infty \label{eqnrnoscillations} 
\end{align}
for some explicit continuous functions $b,c:(0,\infty) \to (0,\infty)$ which are multiplicatively periodic in the sense that $b(2\ell) = b(\ell)$ and $c(2\ell) = c(\ell)$.

In fact, the main aspect of our analysis is the study of the limit cycle for this particular universality class. The kernels in the limit cycle are obtained from a chain of $J$-inner functions, and we will prove that this chain can be parametrized by its growth rate relative to the Martin function (known as $M$-type); this is the first positive result about parametrization by $M$-type for a chain with a singular measure.  From the chain parametrized in this way, the functions $b(n)$, $c(n)$ and the asymptotics \eqref{eqnKnoscillations}, \eqref{eqnrnoscillations} are obtained.

Quadratic iteration has a central place in dynamics, and it is our intention that this analysis helps in understanding more general dynamically defined settings.

In the remainder of this introduction, we will formulate the results precisely and provide some further context.

\subsection{Limit cycle behavior and local zero spacing}
We will describe asymptotic behavior for $2\times 2$ matrix kernels, from which \eqref{eqnLimitCycleCDkernel} will follow. To formulate this, we need the following definitions. Orthonormal polynomials of the second kind $q_j$ are defined by
\[
q_j(z)=\int\frac{p_j(x)-p_j(z)}{x-z}d\mu(x), \qquad j \ge 0.
\]
We will study the $2\times 2$ matrix kernels
\begin{align}\label{intro:eq3}
\cK_\mu(n,z,w)=\begin{pmatrix}
		\sum_{k=0}^{n-1}q_{k}(z)\overline{q_{k}(w)}& \sum_{k=0}^{n-1}q_{k}(z)\overline{p_{k}(w)}\\\sum_{k=0}^{n-1}p_{k}(z)\overline{q_{k}(w)}& \sum_{k=0}^{n-1}p_{k}(z)\overline{p_{k}(w)}
	\end{pmatrix},
\end{align}
which encode information about all eigensolutions of the associated Jacobi matrix, and contain the CD kernel as a diagonal entry.  Scaling limits of these kernels are not kernels generated by orthogonal polynomials. They typically correspond, in a more abstract way, to a measure without finite moments. To formulate our results properly and describe the limit cycle, we must introduce a larger space of $2\times 2$ matrix kernels and recall elements of $J$-multiplicative function theory \cite{ArovDym,deBrangesHilbertSpace,Potapov,RemlingBookCanSys}.

\begin{remark}\label{rem:intro}
In this paper we use
	\begin{align}\label{intro:eq16}
		J=\begin{pmatrix}
			0&-1\\1&0
		\end{pmatrix}.
	\end{align}
\begin{enumerate}[(i)]
\item 		An entire $2\times 2$ matrix function, $\fA(z)$, is said to be $J$-inner if it obeys $i(\fA(z)J\fA(z)^*-J)\leq 0$ for $z\in\bbC_+$ and $\fA(z)J\fA(z)^*-J=0$ for $z\in\bbR$. We denote the set of all $2\times 2$ entire $J$-inner matrix functions with $\det \fA = 1$ by $\mathcal I(2)$, and we denote by $\bbC\bbD(2)$ the set of kernels $\cK:\bbC\times\bbC\to\bbC^{2\times 2}$ generated by some $\fA\in\mathcal I(2)$ by the formula
		\begin{align}\label{intro:eq7}
			\cK(z,w)=\cK_\fA(z,w)=\frac{\fA(z)J\fA(w)^*-J}{z-\overline{w}}.
		\end{align}
Any $\cK \in \bbC\bbD(2)$ is  analytic in $z$ and antianalytic in $w$, i.e., the denominator in \eqref{intro:eq7} produces a removable singularity at $z=\ol w$. We equip $\bbC\bbD(2)$ with the topology of local uniform convergence in $\bbC\times\bbC$ in every entry.
\item
	A family $(\fA(t,\cdot))_{t\in[0,\infty)}$ of entire $J$-inner matrix functions is said to be $J$-monotonic if for every $t_1<t_2$ and $z\in\bbC\setminus\bbR$
	\[
	\cK_{\fA(t_1,\cdot)}(z,z)\leq  \cK_{\fA(t_2,\cdot)}(z,z), \quad \cK_{\fA(0,\cdot)}(z,z)=0.
	\] 
\item The kernels \eqref{intro:eq3} obtained from orthogonal polynomials on the real line are in the space $\bbC\bbD(2)$. They are naturally interpolated to a continuous $J$-monotonic family by piecewise linear interpolation,
		\begin{equation}\label{linearinterpolation}
			\cK_\mu(L,z,w) = \cK_\mu(n,z,w) + (L-n)( \cK_\mu(n+1,z,w) - \cK_\mu(n,z,w) )
\end{equation}
for $n \le L < n+1$. Kernels corresponding to one-dimensional Schr\"o\-dinger and Dirac operators are also in the space $\bbC\bbD(2)$. So are kernels for orthogonal polynomials on the unit circle, after a simple transformation.
\item Denote $\fA=\left(\begin{smallmatrix}\fA_{11}&\fA_{12}\\\fA_{21}&\fA_{22}\end{smallmatrix}\right)$. Henceforth, we will assume in addition that the $J$-monotonic family is in the limit point case, that is, the limit
\begin{align}\label{intro:eq13}
	m(z)=\lim_{t\to\infty}\frac{\fA_{11}(t,z)\tau+\fA_{12}(t,z)}{\fA_{21}(t,z)\tau+\fA_{22}(t,z)}
\end{align}
is independent of $\tau \in \bbC_+ \cup \bbR \cup \{\infty\}$.  In this case, \eqref{intro:eq13} defines the Weyl function, which is a Herglotz function: an analytic map $\bbC_+ \to \bbC_+ \cup \bbR \cup\{\infty\}$.
\item If $(\fA(t,\cdot))_{t\in[0,\infty)}$ is a continuous $J$-monotonic family, we call 
\[
\mathscr C=\{\fA(t,\cdot)\mid t\in [0,\infty)\}
\]
the associated \textit{chain}. Every monotone bijection from $[0,\infty)$ to $\mathscr C$ will be called a \textit{parametrization} of the chain. In the limit point case, the Weyl function \eqref{intro:eq13} does not depend on the  parametrization, so  we may  denote it by $m_{\mathscr C}$. 
\item If $U(t)$ are $J$-unitary, i.e., $U(t) J U(t)^* = J$, a gauge transformation $\fA(t,z) \mapsto \fA(t,z) U(t)$ doesn't change the Weyl function. Any chain can be uniquely transformed into the Potapov--de Branges gauge $\fA(t,0) = I$.
\item We recall the following result of de Branges \cite[Chapter 2]{deBrangesHilbertSpace}: the map ${\mathscr C}\mapsto m_{\mathscr C}$ is a bijection from the set of all chains coming from continuous, $J$-monotonic families in the limit point case in the Potapov--de Branges gauge, and the set of all Herglotz functions. For a Herglotz function $m$, we will denote by ${\mathscr C}(m)$ the corresponding chain. 
\item Every chain can be  parametrized so that $t\mapsto \fA(t,\cdot)$ is absolutely continuous and so that for a.e. $t\in[0,\infty)$, 
\[
\tr\left(\partial_z\partial_t\fA(t,z)J|_{z=0}\right)=1.
\]
This trace parametrization was used by de Branges \cite{deBrangesHilbertSpace}, and the chain is then denoted by
\begin{equation}\label{eq:1jun-1}
\mathscr C(m)=\{\fD(t,\cdot)\mid t\in [0,\infty)\}, \ \ \fD(0,z)=I,\ \tr\left(\partial_z\partial_t\fD(t,z)J|_{z=0}\right)=1.
\end{equation}
Completeness of this chain is the main contribution of de Branges in the proof of the theorem as stated above. Existence of the chain was found earlier by V.P. Potapov. 
The chain ${\mathscr C}(m)$ is complete in the following sense: let $\fB(z)$   be an entire $J$-inner matrix function such that $\det\fB(z)=1$, $\fB(0)=I$, and
$$
\cK_\fB(z,z)\le \cK_{\fD(t_0,\cdot)}(z,z)
$$
for some $t_0\in\bbR$. Then there exists a unique $t_\fB\le t_0$ such that $\fB(z)=\fD(t_\fB,z)$.
\item Certain transformations obviously stay in the space $\bbC\bbD(2)$; for instance, we will use the conjugation $\cK_s \mapsto \cUnew(r) \cK_s \cUnew(r)^*$ where $r > 0$,
			\begin{align}\label{def:VR}
			\cUnew(r)=\begin{pmatrix}
				r^{1/2}&0\\0&r^{-1/2}
			\end{pmatrix}.
			\end{align}	
\end{enumerate}
\end{remark}

All prior studies of the scaling behavior of CD kernels described scaling limits. In our result, we derive limit cycle behavior of the CD kernels from oscillatory boundary behavior of the Weyl function.

	\begin{theorem}\label{thm:4}
		Let $\mu$ be a measure on $\bbR$ corresponding to a determinate moment problem, $m_\mu$ the associated Weyl function and $\cK_\mu(L,\cdot,\cdot)$ the associated kernels \eqref{intro:eq3}, \eqref{linearinterpolation}. Fix $\xi \in \bbR$ and assume that there exist $\beta \in (-1,1)$, $\rho > 1$, and a function $\omega:\bbC_+ \to \bbC\setminus\{0\}$ which satisfies
\begin{equation}\label{rhoperiodic}
\omega(\rho z)=\omega(z), \qquad \forall z\in\bbC_+,
\end{equation}
such that the Weyl function has the normal boundary behavior
\begin{align}\label{intro:eq14}
\lim_{y\to 0}| e^{-\beta i \pi/2} y^{-\beta}m_\mu(\xi+iy)-\omega(iy)|=0.
\end{align}
		For $L > 0$ define
		\begin{align}\label{intro:eq17}
			r(L)=r(L,\xi):=\sqrt{(\cK_\mu)_{11}(L,\xi,\xi)(\cK_\mu)_{22}(L,\xi,\xi)}.
		\end{align}		
Then, there exists a family of matrix kernels $\cL_s \in \bbC\bbD(2)$, $s\in (0,\infty)$, and a constant $\rho > 1$, such that $\cL_{s\rho} = \cL_s$ for all $s$ and
			\begin{align}\label{intro:eq11}
				\lim\limits_{L\to\infty}\left(\frac{1}{r(L)}\cUnew(r(L)^\beta)\cK_\mu\left(L,\xi+\frac{z}{r(L)},\xi+\frac{w}{r(L)}\right)\cUnew(r(L)^\beta)^*-\cL_{r(L)}(z,w)\right)=0
			\end{align}
			uniformly for $(z,w)$ in compact subsets of $\bbC\times\bbC$. In particular,  \eqref{eqnLimitCycleCDkernel} follows from \eqref{intro:eq11} by taking a diagonal entry.

With the branch of $z^\beta$ on $\bbC_+$ with $0 < \arg(z^\beta) < \beta \pi$, the function $m_{\beta,\omega}(z) = z^\beta \omega(z)$ is a Herglotz function. This function uniquely determines the limit cycle of kernels: 
 if we denote by $\cK(t,\cdot,\cdot)$ the trace parametrized kernels associated with $m_{\beta,\omega}$, and for $s > 0$ we denote by $t(s)$ the unique point in $(0,\infty)$ so that 
			\begin{align}\label{eq:increasingTS}
			\cK_{11}(t(s),0,0)\cK_{22}(t(s),0,0)=s^2,
			\end{align}
then the limit cycle consists of the kernels
\begin{equation}\label{eqnLimitKernelsMatrix1}
\cL_s(z,w)=\frac{1}{s}\cUnew(s^\beta)\cK\left(t(s),\frac{z}{s},\frac{w}{s}\right)\cUnew(s^\beta)^*.
\end{equation}
\end{theorem}

Theorem \ref{thm:4} is only a special case of our more general Theorem \ref{thm:3}. Theorem \ref{thm:3} is formulated for canonical systems, and Theorem \ref{thm:4} is its application to orthogonal polynomials. Theorem~\ref{thm:3} has analogous applications to continuum Schr\"odinger operators, Dirac systems, Krein strings, and orthogonal polynomials on the unit circle.
	
We note that the geometric mean of diagonal entries \eqref{intro:eq17}, used here as a scaling function, was previously used in \cite{JitomirskayaLast} in the setting of Jacobi matrices and more generally for canonical systems in \cite{LanPruckWor} in order to estimate the function $y\mapsto |m_\mu(\xi+ iy)|$.

By Theorem~\ref{thm:4}, the function $m_{\beta,\omega}$ determines the universality class. Moreover, transformations of the form
\begin{equation}\label{eqnmrescaling6}
\tilde m_{\beta,\omega}(z) = \begin{cases} c_1 c_2 m_{\beta,\omega}(c_2 z) & c_1 > 0, c_2 > 0 \\
 c_1 c_2  \ol{ m_{\beta,\omega}(c_2 \ol z)} & c_1 > 0,  c_2  < 0
 \end{cases}
\end{equation}
correspond to simple transformations of the spectral measure: scalar multiplication by $c_1$ and pushforward by the linear map $\lambda \mapsto c_2 \lambda$. These correspond to equally simple transformations of the kernels in the limit cycle. Thus, if $m_{\beta,\omega}$ and $\tilde m_{\beta,\omega}$ are related by \eqref{eqnmrescaling6}, we will say that they correspond to the same universality class.

The limit cycle corresponds to a canonical system, which we think of as the limiting object representing the local behavior at $\xi$. This canonical system has its own measure and spectral type, not to be confused with the orthogonality measure and its spectral type.  Previously studied universality classes, where there is convergence to a single limit kernel, correspond to $\omega(z)\equiv \omega_0$ for some constant $\omega_0$, so the limiting object had an a.c.\ measure. The multiplicatively periodic function $\omega$ allows for much more freedom. We will provide examples where the measure of $m_{\beta,\omega}$ is purely singular continuous. This is, up to our knowledge, the first example of universality classes where the limiting kernels correspond to a singular measure. 

We will now formulate corollaries for zero spacing. 
We will say that the universality class is an edge universality class if $m_{\beta,\omega}$ has an analytic extension to $\bbC \setminus (-\infty,0]$ or to $\bbC \setminus [0, \infty)$ which satisfies $\ol{ m_{\beta,\omega}(\ol z)} = m_{\beta,\omega}(z)$ (for constant $\omega$, this corresponds to the well known hard edge universality classes). For zero spacing, edge universality classes are qualitatively different: the nearby zeros of orthogonal polynomials will be on one side of the point $\xi$, with at most one exception. We will also state a result about the distance from the smallest  zero to the minimum of the spectrum.

\begin{corollary}\label{corollaryZerosGeneral}
Under the assumptions of Theorem~\ref{thm:4}, let $\xi$ be such that the Weyl function has the normal boundary behavior \eqref{intro:eq14}. Then: 
\begin{enumerate}[(a)]
\item If the universality class of $m_{\beta,\omega}$ is not an edge universality class, then for any $k,l\in \bbZ$ with $k<l$, \eqref{eqn:distancebetweenzeros2} holds.
\item If the universality class of $m_{\beta,\omega}$ is an edge universality class, without loss of generality say $\ol{ m_{\beta,\omega}(\ol z)} = m_{\beta,\omega}(z)$ on $\bbC \setminus [0,\infty)$, then for any $k, l \in \bbN$, with $k < l$, \eqref{eqn:distancebetweenzeros2} holds. Moreover, for any $R > 0$, for all large enough $n$, $p_n$ can have at most one zero in $[\xi - R K(n,\xi,\xi)^{-1/(1+\beta)}, \xi]$, and along a subsequence where such a zero $\xi_0^{(n_j)}$ exists, \eqref{eqn:distancebetweenzeros2} holds also for $0 = k < l$.
\item If $\supp\mu$ is bounded below, and at the point $\xi = \min \supp \mu$, the Weyl function satisfies the normal boundary behavior \eqref{intro:eq14}, then  for $1 \le l$,
\begin{equation}\label{eqn:distancefromendpoint2}
\xi_l^{(n)} - \xi \asymp K(n,\xi,\xi)^{-1/(1+\beta)} 
\end{equation}
\end{enumerate}
The constants implicit in \eqref{eqn:distancebetweenzeros2} and \eqref{eqn:distancefromendpoint2} depend only on $\beta,\omega, k,l$.
\end{corollary}

\subsection{The Cantor measure}
We now recall the middle third Cantor set $C \subset [0,1]$, which can be written in terms of $T(x) = 3 \lvert x - \frac 12\rvert - \frac 12$ as $C = \cap_{n=0}^\infty (T^{-1})^n([0,1])$, and the middle third Cantor measure.

\begin{theorem}\label{thmCantorMeasure}
Let $\mu$ be the middle third Cantor measure. For any $\xi \in \bbQ \cap \supp \mu$, the conclusions of Theorem~\ref{thm:4} hold at $\xi$, with
\begin{equation}\label{eqnbetaCantorMeasure}
\beta = \frac{\log 2}{\log 3} - 1.
\end{equation}
Moreover,  if $\xi,  \eta \in  \bbQ \cap \supp \mu$ are such that for some $m, n \in \bbN$, $3^m \xi  - \lfloor 3^m \xi \rfloor = 3^n \eta - \lfloor 3^n \eta \rfloor$, then they are in the same universality class.
\end{theorem}

In particular, all gap edges are in the same edge universality class as the endpoint $0$. The relevance of rationality in this result is that rational points are eventually periodic in the dynamics on the Cantor set, and the relevance of the condition on $\xi,\eta$ is that  they are eventually on the same periodic orbit. For the middle third Cantor measure, a lower bound on the zero spacing was proved in  \cite{KrugerSimon15}.

In the above example, $\beta+1$ is the Hausdorff dimension, but we wish to reiterate that spectral type is not directly related to the local scaling behavior of the kernel. Bulk universality was originally studied for a.c.\ measures but can also occur for singular measures \cite{BSing} and even for pure point measures \cite{EichLukWor}.  Likewise, we now point out an example of limit cycle behavior with $\beta < 0$ in an absolutely continuous measure.

\begin{example}\label{xmplACLimitCycle}
For the absolutely continuous measure
\[
d\mu(x) = \sum_{n=0}^\infty  \left( \frac 32 \right)^{n+1} \chi_{(2/3^{n+1}, 1/3^n)}(x) \,dx,
\]
the CD kernels have limit cycle behavior at $\xi = 0$,  as in the conclusions of Theorem~\ref{thm:4}, with $\beta$ again given by \eqref{eqnbetaCantorMeasure}. At every nonzero $\xi \in \supp \mu$, the CD kernels obey bulk universality or hard edge universality.
\end{example}

\subsection{Balanced/equilibrium measures on Julia sets} 
Let $T$ be a polynomial with a real Julia set $\sE_0$. In particular, $T$ has real coefficients. 
Its $n$-th iterate is denoted $T^{\circ n}$. The polynomial is said to be expanding on its Julia set if there exist $a> 0$, $\lambda > 1$ such that $\lvert (T^{\circ n})'(z) \rvert \ge a \lambda^n$ for all $z \in \sE_0$ and $n\in\bbN$ \cite[Section 3]{ErLyu}. The Julia set is then of the form
\begin{align}\label{JuliaSet}
\sE_0:=\bbC\setminus\{z\in\bbC\mid T^{\circ n}(z)\to\infty,\ n\to\infty\}.
\end{align}
Let $\mu_{\sE_0}$ denote the balanced measure of $\sE_0$, i.e., the unique probability measure such that for every $f\in C(\sE_0)$
\[
\int \frac{1}{d}\sum_{T (y)=x}f(y)d\mu_{\sE_0}(x)=\int f(x)d\mu_{\sE_0}(x).
\]
From the perspective of potential theory, $\mu_{\sE_0}$ is the equilibrium measure of $\sE_0$ \cite{Brolin65}, and the harmonic measure for the point $\infty$. There are no prior results about rescaling limits of CD-kernels for $\mu_{\sE_0}$. We show that these balanced measures provide a natural application of Theorem~\ref{thm:4}; in fact, they were the motivating example for this work.

\begin{theorem}\label{intro:thm2}
Let $T$ be an expanding polynomial, $d=\deg T > 1$, with a real Julia set $\sE_0 $. Let $\mu_{\sE_0}$ denote the balanced measure of $\sE_0$. Let $\xi \in \sE_0$ be an eventually fixed or periodic point, i.e., there exists a fixed or periodic point $\eta$ such that $T^{\circ n}(\xi) = \eta$ for some $n$. Then the matrix kernels associated to $\mu_{\sE_0}$ obey the limit cycle behavior \eqref{intro:eq11} at the point $\xi$. The constant $\beta$ is given by
\[
\beta = \frac{q \log d}{\log \lvert (T^{\circ q})'(\eta) \rvert} - 1
\]
where $q$ is the period of $\eta$. The limit cycle, up to rescaling, depends only on the eventual orbit of the point.
\end{theorem}

From the invariance of $\mu_{\sE_0}$ it follows that it is singular continuous \cite{BellBeMou82}. In many cases, $\mu_{\sE_0}$ is known to correspond to an almost periodic (in fact, limit periodic) Jacobi matrix \cite{BelGerVolYu05}; this is expected to hold in the general case, but it is still an open problem. In particular, Theorem~\ref{intro:thm2} is the first result to describe scaling behavior of CD kernels for some almost periodic Jacobi matrices with singular spectrum.

We regard it as an interesting open problem whether this limit cycle behavior holds in other almost periodic (and related) models of interest.

\subsection{Asymptotics of the scaling function and parametrization by $M$-type}
We now turn to two highly model-dependent questions: the asymptotic behavior of $K(n,\xi,\xi)$ as $n\to\infty$, and the nature of the  chain corresponding to the limit cycle. We will answer these questions in the canonical model of interest: that of a quadratic polynomial with a real Julia set, at a fixed point.

We write the polynomial in the form
\[
T(z) = z(z+\rho)
\]
with $-\rho > 2$ (by conjugating by the dynamics, this is equivalent to the often used form $T(w) = w^2 - \lambda$, with $\lambda > 2$).
 We will study the fixed point $0$, which is the internal fixed point. As before, $\sE_0$ denotes its Julia set and $\mu_{\sE_0}$ the balanced measure. Since $T$ is expanding and $0$ is a fixed point of $T$, by Theorem~\ref{intro:thm2}, $\cK_{\mu_{\sE_0}}$ has a limit cycle  at $\xi =0$.

To describe the chain corresponding to the limit cycle in this case, we have to recall exponential type and $M$-type and their importance in spectral theory.

\begin{remark}
\begin{enumerate}[(i)]
\item Recall that the order of an entire function $f$ is 
\begin{align*}
	\rho_f:=\limsup_{r\to\infty}\frac{\ln\ln\max_{|z|\leq r}|f(z)|}{\ln r}
\end{align*}   
and the type with respect to this order is 
\[
\tau_f:=\limsup_{r\to\infty}\frac{\ln\max_{|z|\leq r}|f(z)| }{r^{\rho_f}}.
\]
\item In direct spectral theory, transfer matrices obey some universal asymptotics at $\infty$. For instance, the $n$-step transfer matrix of the Jacobi recursion is a polynomial of degree $n$, Dirac transfer matrices over an interval of length $\ell$ have order $1$ and exponential type $\ell$, and Schr\"odinger transfer matrices over an interval of length $\ell$ have order $1/2$ and type $\ell$ \cite{PoschelTrubowitz,LukicFirstCourse}. These asymptotics underlie important results such as proofs of the Borg--Marchenko theorem or Stahl--Totik regularity theory (compare \cite{Bennewitz,SimonEquilibrium,EichLuk,EichLukGwa}).  Note that the length of the interval is encoded as the type; in other words, the natural parameter in direct spectral theory is encoded in the asymptotics of the transfer matrix.
\item In $J$-theory, there are no universal asymptotics, and asymptotics at $\infty$ are not preserved by locally uniform convergence. Our transfer matrices have exponential type given by the Krein--de Branges formula \cite{Krein51}, \cite[Section 39]{deBrangesHilbertSpace}, \cite[Theorem 1.11]{BLY1}, \cite{ReiffensteinWoracek}, but the exponential type is often $0$ and doesn't serve as a parameter for the family of transfer matrices. 
\item Now we provide a precise definition of the $M$-type. If $\sE\subset\bbR$ is closed and unbounded, $\infty$ is a boundary point of the Denjoy domain $\Omega=\bbC\setminus \sE$. Assume in addition that $\Omega$ is Greenian. Then a \textit{Martin function} at $\infty$ is a positive harmonic function in $\Omega$, which is bounded on bounded subsets of $\Omega$ and vanishes continuously at every Dirichlet boundary point of $E$.  It is called symmetric if in addition $M(\overline{z})=M(z)$. A symmetric Martin function is unique up to normalization \cite{Ancona79,Benedicks80}. We say that an entire function, $f$, has $M$-type $\ell>0$, if it is of bounded characteristic in $\bbC\setminus \sE$ \cite[p. 861]{VolYu16} and 
\[
\ell=\limsup\limits_{y\to\infty}\frac{\log|f(iy)|}{M(iy)}. 
\]
Using the matrix norm, the $M$-type of a matrix function is defined analogously, and this is a natural way to quantify the growth of transfer matrices.

\item  For a periodic canonical system, Floquet theory describes its spectrum $\sE$ in terms of the monodromy matrix. Typically, $\sE$ is unbounded, i.e., doesn't correspond to orthogonal polynomials, and we can consider a symmetric Martin function $M$. If the spectrum satisfies some known regularity properties (the Widom condition and DCT property), it was proved in \cite{Yuditskii2001} that the chain is parametrized in terms of $M$-type, i.e., each transfer matrix has finite $M$-type and for each $\ell > 0$ there is a unique transfer matrix in the chain which has $M$-type. In \cite{BessLukYu}, the same was proved for almost periodic canonical systems with the Widom and DCT properties. In these classes, spectra are of positive Lebesgue measure and spectral measures are absolutely continuous.

\item Parametrization by $M$-type is not always possible. For instance, the transfer matrix for the free Schr\"odinger operator is
\[
\fA_S(t,z)=\begin{pmatrix}
\cos(\sqrt{z}t)& \frac{\sin(\sqrt{z}t)}{\sqrt{z}}\\ -\sqrt{z}\sin(\sqrt{z}t)& \cos(\sqrt{z}t)
\end{pmatrix}
\]
and the transfer matrix for the free Dirac operator is
\[
\fA_D(t,z)=
\left(\begin{matrix}
\cos(zt)& \sin(zt)\\ -\sin(zt)& \cos(zt)
\end{matrix}\right).
\]
Building a chimera of these two families gives a $J$-monotonic family
\begin{equation}\label{eqnChimera}
	\fA(t,z)=\begin{cases}
		\fA_D(t,z),&\quad t\leq 1\\\medskip
		\fA_D(1,z)\fA_S(t-1,z),&\quad t\geq 1
	\end{cases}
\end{equation}
In this case for arbitrary $t>0$, the entries of $\fA(t,z)$ are of order $1$, but for $t\geq 1$, the type remains constant. Hence, a parametrization of $\fA$ in terms of $M$-type is not possible in such an example. 

\item Chimeras are common objects in $J$-theory. As an example consider a classical Krein problem on extensions of Hermitian-positive functions $S(t)$ from a finite interval $(-\ell,\ell)$ to the whole axis \cite[Chapt. 5, \S 3]{AkhiezerMomentProblem}. Each extension can be associated with a measure $\sigma$ on the real axis, which provides its Bochner representation, and the measure in turn corresponds to a Herglotz function. In the case of non-uniqueness, the collection of Herglotz functions is parametrised in terms of a fractional linear transform with entire $J$-inner function $\fA(\ell,z)$ as the matrix of this transform. This matrix is of exponential order one and its type depends on $\ell$. However, as soon as $S(t)$ possesses an additional smoothness in the origin (differentiable or just continuous) $\fA(\ell,z)$ contains with necessity a polynomial multiplier in its multiplicative representation, i.e., the factor corresponding to a moment problem. Indeed, each derivative of the Hermitian-positive function requires existence of an additional moment for the representative measure $\sigma$.
\item Chimeras have often led to counterexamples of conjectures in spectral theory. In \cite{Yuditskii2011}, they were used to prove that Widom DCT sets are not necessarily homogeneous, and that weakly homogeneous sets are not necessarily DCT. In \cite{VolbergYuditskii14}, they were used to disprove the Kotani--Last conjecture.
\end{enumerate}
\end{remark}

The main theorem of this part of the paper is that the chain of entire $J$-inner matrix functions associated to the limit cycle of $\cK_{\mu_{\sE_0}}$ can be parametrized by the $M$-type (i.e., in this sense this chain is not a chimera). In order to formulate the precise statement, we will now define the Herglotz function corresponding to the limit cycle and the $M$-type. 

Let $T(z)=z(z+\rho)$ and $\mu_{\sE_0}$ be the balanced measure of its Julia set. Define
\begin{equation}\label{intro:eq201}
m_{\sE_0}(z)=\int\frac{d\mu_{\sE_0}(x)}{x-z}.
\end{equation}
The Herglotz function corresponding to the limit cycle will be given in terms of a well-studied function in iteration theory. It is a classical result that the following limit exists, see e.g. \cite[Sect. 3]{Grabner15} 
\begin{equation*}
F(z)=\lim_{n\to\infty}F_n(z),\quad F_n(z)=T^{\circ n}(z/\rho^n).
\end{equation*}
The limit function $F(z)$ satisfies the Poincar\'e equation
\begin{equation}\label{eq1-2-25}
T(F(z))=F(\rho z),\quad F(0)=0, F'(0)=1,
\end{equation}
and we thus refer to it as the Poincar\'e  function. Set $\sE:=F^{-1}(\sE_0)$ and define the function 
\[
m_\sE(z)=F'(z)m_{{\sE_0}}(F(z)).
\]
The function $m_\sE(z)$ is again a Herglotz function and admits an integral representation  similar to \eqref{intro:eq201}
\begin{equation*}\label{intro:eq202}
m_{\sE}(z)=\int\frac{d\mu_{\sE}(x)}{x-z}.
\end{equation*}
 The measure $\mu_\sE$ is supported on $\sE$ and therefore in particular is again singular.   Note that in contrast to $\sE_0$, $\sE$ is an unbounded set and thus $\infty$ is a boundary point of $\Omega=\bbC\setminus \sE$.

\begin{theorem}\label{intro:thm3}
	Let $\sE_0$ be the Julia set of $T(z)=z(z+\rho)$, for $\rho<1-\sqrt{13}$. Let $\sE=F^{-1}(\sE_0)$ and $m_\sE(z)=F'(z)m_{\sE_0}(F(z))$. Let $\fD(t,z)$ be the de Branges' trace parametrization of ${\mathscr C}(m_\sE)$. Then the $M$-type of $\fD(t,z)$ exists for every $t>0$, i.e., 
	\[
	\ell(t)=\lim\limits_{y\to\infty}\frac{\log\|\fD(t,iy)\|}{M(iy)}. 
	\]
	Moreover, $\ell(t)$ is continuous and bijective. Hence, denoting by $t(\ell)$ its inverse,  $\fD(t(\ell),z)$ provides a parametrization of ${\mathscr C}(m_\sE)$ in terms of the $M$-type. 
\end{theorem}

\begin{remark}
	\begin{enumerate}[(i)]
		\item Theorem \ref{intro:thm3} is the first example where a parametrization in terms of the $M$-type was proven for a system with a singular spectral measure.
		\item Let us explain the condition on $\rho$. The proof of Theorem  \ref{intro:thm3} relies on limit periodicity of the Jacobi matrix associated with the balanced measure of the Julia set of the polynomial $T(z)=z(z+\rho)$. This fact is known for the Jacobi matrix associated with the balanced measure of the Julia set of the quadratic polynomial $\tilde T(z)=z^2-\lambda$, for $\lambda>3$, \cite{BellBeMou82,BeMeMo,NPVY} and was improved in \cite{BakerBessisMoussa} to $\lambda>2.192$. The corresponding claim for the Jacobi matrix associated with $T$ follows, since the dynamical systems generated by $T$ and $\tilde T$ are equivalent with $\lambda=\rho^2/4-\rho/2$. Thus, our condition on $\rho$ corresponds to $\lambda>3$. 
	\end{enumerate}
\end{remark}

 Let 
$$
\cK_\fD(t(\ell),z,w)=\frac{\fD(t(\ell),z)J\fD(t(\ell),w)^*-J}{z-\overline{w}}.
$$
 We now describe the asymptotics of the Christoffel function and the matrix scaling function $r(n)$:

	\begin{theorem}\label{intro:thm1}
		Let $\mu_{\sE_0}$ be the balanced measure for the polynomial $T(z)=z(z+\rho)$, where $-\rho>2$ and define $\kappa=\log |\rho|/\log 2$. 
		\begin{itemize}
		\item[(a)]
		\begin{align}
K(n,0,0) & \asymp n, \qquad n\to\infty \label{intro:eq23b} \\
r(n) & \asymp n^\kappa, \qquad n\to\infty \label{intro:eq23}
		\end{align}
		\item[(b)]
		Assume that $-\rho>\sqrt{13}-1$. Then the functions $b,c : (0,\infty) \to (0,\infty)$ given by
			\begin{equation}\label{eq:28may-1b}
b(\ell)=\frac{1}{\ell}(\cK_\fD)_{22}(t(\ell),0,0)
\end{equation}
		\begin{equation}\label{eq:28may-1}
c(\ell)=\frac{1}{\ell^{\kappa}}\sqrt{(\cK_\fD)_{11}(t(\ell),0,0)(\cK_\fD)_{22}(t(\ell),0,0)}
\end{equation}
are continuous, multiplicatively periodic in the sense that
\[
b(2\ell) = b(\ell), \qquad c(2\ell) = c(\ell),
\]
and \eqref{eqnKnoscillations}, \eqref{eqnrnoscillations} hold.
		\end{itemize}
	\end{theorem}

For instance, under the assumptions of (b), \eqref{intro:eq23b} can be written in the more explicit form
		\[
		\min_{\ell \in [1,2)} b \leq \liminf_{n\to\infty} \frac{ K(n,\xi,\xi)} n \leq\limsup_{n\to\infty} \frac{ K(n,\xi,\xi)} n  \leq \max_{\ell \in [1,2)} b.
		\]

For dyadic rationals $\ell = k / 2^s$, we can give an explicit formula for $\fD(t(\ell),z)$, see Theorem \ref{thnov22-1} and  \eqref{eq:bks}.

\medskip

In Section~\ref{sec:prel}, we will review some properties of canonical systems and their rescalings. In Section~\ref{secPeriodicModulation} we will study self-similar canonical systems and those with multiplicatively periodically modulated scaling behavior; in particular we will prove Theorem~\ref{thm:4}, Corollary~\ref{corollaryZerosGeneral}, and Theorem~\ref{thmCantorMeasure}. In Section~\ref{secExpandingJuliaSet} we study balanced measures on Julia sets and prove Theorem~\ref{intro:thm2}. In Section~\ref{sectionChristoffel}
we prove Theorem~\ref{intro:thm3} and Theorem~\ref{intro:thm1}.

\subsection*{Acknowledgements} We are thankful to M. Sodin, who pointed more than 40 years ago to a possible use of the Poincar\'e function in studying canonical systems with singular spectrum. 
	\section{Preliminaries}\label{sec:prel}

	In this section we recall some facts and basic definitions about canonical systems, some associated objects and continuity properties among those objects. 
	
	\subsection{Canonical systems, CD-kernels and Weyl functions}
	
	\begin{definition}
		We call a function $H:[0,\infty)\to\bbR^{2\times 2}$ a \emph{Hamiltonian} if 
		$H\in L^1_\loc([0,\infty))$, $H(t)\geq 0$ for a.e. $t\in(0,\infty)$, and $H(t)\neq 0$ for a.e. $t\in(0,\infty)$.	
		We assume in addition that $H$ is in the limit point case at $\infty$,  i.e., 
		\begin{align*}
			\int_0^\infty\tr H(t)dt=\infty. 
		\end{align*}
		We denote the set of all such Hamiltonians by $\bbH$. We call a Hamiltonian \emph{trace normalized}, if $\tr H(t)=1$ for a.e. $t\in(0,\infty)$. The set of all trace normalized canonical systems on $[0,\infty)$ will be denoted by $\hat\bbH$. 
		
		If $H\in\bbH$, the associated \textit{half-line}  \textit{canonical system} is the following differential expression
		\begin{align}\label{eq:10}
			\partial_t\fA_H(t,z) J=z\fA_H(t,z)H(t),\quad \fA_H(0,z)=
			\begin{pmatrix}
				1&0\\0&1
			\end{pmatrix},
		\end{align}
		where $t\in[0,\infty)$, $J$ is given by \eqref{intro:eq16} and $z\in\bbC$ is the spectral parameter. The matrix function $\fA_H$ is called the \textit{transfer matrix} of the system.
	\end{definition}
	
	Let $H_1,H_2\in\bbH$. We say that $H_2$ is a \emph{reparametrization} of $H_1$ and write $H_2\sim  H_1$, if there exists an increasing 
	bijection $\gamma:(0,\infty)\to(0,\infty)$ such that $\gamma$ and $\gamma^{-1}$ are absolutely continuous and
	\begin{equation}\label{eq:999}
		H_2(t)=H_1(\gamma(t))\gamma'(t).
	\end{equation}
	If $\fA_{1,2}$ denote the corresponding transfer matrices, then we have
	\[
	\fA_{2}(t,z)=\fA_1(\gamma(t),z).
	\]
	In particular using
	\[
	\gamma(t)=\int_0^t\tr H(s)ds
	\]
	shows:
	\begin{lemma}
		Let $H\in\bbH$. Then there always exists a trace-normalized $\hat H\in\hat\bbH$ with $H\sim \hat H$.
	\end{lemma}
	
	For $H\in\bbH$, it follows directly from \eqref{eq:10} that for $t\geq 0$
	\begin{align}\label{eq:22}
		\cK_H(t,z,w):=\frac{\fA_H(t,z)J\fA_H(t,w)^*-J}{z-\overline{w}}=\int_0^t\fA_H(s,z)H(s)\fA_H(s,w)^*ds.
	\end{align}
	In particular, $H\geq0$ implies that $\cK_H\in\bbC\bbD(2)$.
	
	A \emph{Herglotz function} is an analytic function from $\bbC_+$ to $\bbC_+$. The set of all analytic maps $w : \bbC_+ \to \overline{\bbC_+}=\bbC_+\cup\bbR\cup\{\infty\}$  is denoted by $\cN_0$. If $w\in \cN_0\setminus\{\infty\}$, then there exists $\alpha\in\bbR, \beta\geq 0$ and a positive Borel measure $\sigma$ with $\int\frac{d\sigma(x)}{1+x^2}<\infty$, 
	such that 
	\begin{align}\label{eq:100}
	w(z)=\alpha+\beta z+\int_\bbR\left(\frac{1}{x-z}-\frac{x}{1+x^2}\right)d\sigma(x).
	\end{align}

	Write $\fA_H=\left(\begin{smallmatrix}\fA_{11}&\fA_{12}\\\fA_{21}&\fA_{22}\end{smallmatrix}\right)$. For $z\in\bbC_+$, we define the \textit{Weyl function} of $H\in\bbH$ by 
	\begin{align}\label{eq:21}
		m_H(z):=\lim\limits_{t\to\infty}\frac{\fA_{11}(t,z)\tau+\fA_{12}(t,z)}{\fA_{21}(t,z)\tau+\fA_{22}(t,z)}	
	\end{align}
	where the limit does not depend on $\tau\in\overline{\bbC_+}$ since we are in the limit point case.  It follows from the positivity of $H$ that $m_H\in\cN_0$.

		\subsection{Jacobi recursion} \label{sec:Jacobi}
		Let us consider a measure $\mu$ on $\bbR$ with finite moments, not supported on finitely many points, and recall how it relates to canonical systems. 
		The orthonormal polynomials $p_n$ satisfy a three term recurrence relation, i.e., there exist Jacobi parameters $a_n>0,b_n\in\bbR$ such that 
		\begin{align}\label{intro:eq1}
			zp_n(z)&=a_{n+1}p_{n+1}(z)+b_np_n(z)+a_np_{n-1}(z),\quad n\geq 1\\
			zp_0(z)&=a_1p_1(z)+b_0p_0(z).\nonumber
		\end{align}
		For $n\geq 1$, the orthonormal polynomials of the second kind $q_n$ 
		satisfy the same recursion \eqref{intro:eq1}.
		Viewing this as a second order recursion and rewriting it in matrix form leads to 
		\begin{align}\label{intro:eq8}
			\begin{pmatrix}
				-a_nq_{n-1}(z)& -q_n(z)\\
				a_np_{n-1}(z)& p_n(z)
			\end{pmatrix}
			=
			\begin{pmatrix}
				-a_{n-1}q_{n-2}(z)& -q_{n-1}(z)\\
				a_{n-1}p_{n-2}(z)& p_{n-1}(z)
			\end{pmatrix}
			\fa_n(z),
		\end{align}
		where 
		\[
		\fa_n(z)=\begin{pmatrix}
			0& -\frac{1}{a_n}\\a_n& \frac{z-b_{n-1}}{a_n}
		\end{pmatrix}.
		\]
		Note that we are using row notation, where solutions of the recursion correspond to rows of the matrices in \eqref{intro:eq8}. Thus, iterating the recursion gives the $n$-step transfer matrix defined by 
		\begin{align}\label{intro:eq19}
			\fA_n(z)=\fa_1(z)\dots \fa_n(z).
		\end{align}
		From \eqref{intro:eq8} it follows that  
		\begin{align}\label{eq:50}
		\fA_n(z)=\begin{pmatrix}
			-a_nq_{n-1}(z)& -q_n(z)\\
			a_np_{n-1}(z)& p_n(z)
		\end{pmatrix}.
	\end{align}
		Then induction in $n$ shows that the kernel \eqref{intro:eq3} satisfies
		\begin{align}\label{intro:eq10}
			\cK_\mu(n,z,w)=\frac{\fA_n(z)J\fA_n(w)^*-J}{z-\overline{w}}.
		\end{align}
		Using $\cK_\mu(n,z,z) \ge 0$ shows that $\fA_n$ is an entire $J$-inner function and hence $\cK_\mu(n,\cdot,\cdot) \in \bbC\bbD(2)$.

		Let us relate this to canonical systems.
		For $\xi\in\bbR$, let
		\[
		W(n,z)=W(n,z,\xi)=\fA_n(z+\xi)\fA_n(\xi)^{-1}.
		\]
		Since 
		\[
		\fa_{n+1}(z+\xi)\fa_{n+1}(\xi)^{-1}=I-zN,\quad \text{where } N=\begin{pmatrix}
			0&0\\1&0
		\end{pmatrix},
		\]
		a direct computation shows that 
		\begin{align}\label{eq:20}
			(W(n+1,z)-W(n,z))J=zW(n,z)H_\xi(n),
		\end{align}
		where
		\[
		H_\xi(n)=\begin{pmatrix}
			q_n(\xi)^2&-p_n(\xi)q_n(\xi)\\-p_n(\xi)q_n(\xi)& p_n(\xi)^2
		\end{pmatrix}.
		\]
		This can be extended to a canonical system by linear interpolation: For $t\geq 0$ define
		\[
		W(t,z)=W(\lfloor t\rfloor,z)+(t-\lfloor t\rfloor)(W(\lfloor t\rfloor+1,z)-W(\lfloor t\rfloor,z))
		\]
		and
		\[
		H_\xi(t)=H_\xi(\lfloor t\rfloor).
		\]
		Then by direct verification, \eqref{eq:20} induces a  half-line canonical system
		\[
		\partial_t W(t,z)J=zW(t,z)H_\xi(t).
		\]
		Moreover, if $\mu$ is determinate, this canonical system is in the limit point case, and if  $m_{H_\xi}$ is the Herglotz function associated to $H_\xi$ by \eqref{eq:21} and $m_\mu$ as in \eqref{intro:eq20}, then it holds that 
		\begin{align*}
			m_{H_\xi}(z-\xi)=m_\mu(z).
		\end{align*} 
		
		\subsection{Continuity properties}\label{sec:cont}
		
		The topologies on $\cN_0$ and $\hat\bbH$ defined below are chosen so that the map $H\mapsto m_H$ is continuous. 
		
		Let $\Omega\subset\bbC^n$ be open and nonempty and $S_n$, $n\in\bbN_0$, be compact such that $S_n\subset \operatorname{int} S_{n+1}$, $\cup_n S_n=\Omega$. Moreover, let $Y$ be a metric space. Then the topology of locally uniform convergence on $C(\Omega,Y)$ is metrizable, with one choice of metric given by 
		\begin{align*}
			d(f,g):=\sum_{n=1}^\infty2^{-n}\min\{1,\sup_{z\in S_n}d_Y(f(z),g(z))\},\quad f,g\in C(\Omega,Y).
		\end{align*}
		
		Note that $\cN_0\subset C(\bbC_+,\overline{\bbC})$, where the Riemann sphere is endowed with the chordal metric. We equip $\cN_0$ with the subspace topology of locally uniform convergence in $C(\bbC_+,\overline{\bbC})$.  In this way, $\cN_0$ becomes a compact metric space. 
		
		Likewise, we equip $\hat\bbH$ with the metric
		\begin{align*}
			d(H_1,H_2):=\sum_{n=1}^\infty2^{-n}\min\left\{1,\sup_{T\in [0,n]}\left\|\int_0^TH_1(t)dt-\int_0^TH_2(t)dt\right\|\right\}
		\end{align*}
		for $H_1,H_2\in\hat\bbH$.
		
		It is a deep theorem of de Branges, that the map $H\mapsto m_H$ is a bijection.  The original proof of this result was published in the papers \cite[Theorem VI]{dB4}, \cite[Theorem XI, XII]{dB2}. For a book treatment see \cite[Theorem 5.1]{RemlingBookCanSys}. It is then an easy consequence of the definition of the topologies that the map is in fact a homeomorphism \cite[Corollary 5.8]{RemlingBookCanSys}.
		
		\begin{theorem}\label{thm:1}
			The map 
			\[
			\left\{
			\begin{array}{ccc}
				\hat\bbH&\to& \cN_0
				\\[2mm]
				H&\mapsto& m_H
			\end{array}\right.
			\]
			is a homeomorphism. In particular, $\hat\bbH$ is a compact metric space. 
		\end{theorem}
		
		Consider $\fA_H,\cK_H$ as elements of $C([0,\infty)\times\bbC, \bbC^{2\times 2})$ and  $C([0,\infty)\times\bbC\times\bbC, \bbC^{2\times 2})$, respectively. Then with the corresponding topologies, standard ODE arguments show that the maps
		\begin{align}\label{eq:11}
			H\mapsto\fA_H\quad\text{and}\quad H\mapsto \cK_H
		\end{align}
		are continuous.

		\subsection{Rescalings}
		For $H\in \bbH$ we write
		\[
		H(t)=\begin{pmatrix}
			h_1(t)& h_3(t)\\h_3(t)& h_2(t)
		\end{pmatrix}.
		\]
		
		We discuss the following transformation on $H$. 
		\begin{definition}\label{def:3}
			Fix $g(r):(0,\infty)\to(0,\infty)$ and let $\cUnew$ be as in \eqref{def:VR}. We define a map $\cF_r:\bbH\to\bbH$  by 
			\[
			(\cF_r H)(t):=\cUnew(g(r)) H(rt)\cUnew(g(r))^*=\begin{pmatrix}
				g(r)h_1(rt)& h_3(rt)\\h_3(rt)& \frac{1}{g(r)}h_2(rt)
			\end{pmatrix}.
			\]
		\end{definition}

		It can be explicitly computed how this transformation acts on all the above mentioned objects. 
		\begin{lemma}\label{lem:18}
			Let $r>0$ and $H$ be a Hamiltonian on $[0,\infty)$. Then we have 
			\begin{align*}
				\fA_{\cF_rH}(t,z)&=\cUnew(g(r)) \fA_H(tr,z/r)\cUnew(g(r))^{-1},\\
				\cK_{\cF_rH}(t,z,w)&=\frac{1}{r}\cUnew(g(r))\cK_H(tr,z/r,w/r)\cUnew(g(r))^*,\\
				m_{\cF_rH}(z)&=g(r)m_H(z/r).\\
			\end{align*}
\end{lemma}

\begin{proof}
The formula for	$\fA_{\cF_rH}(t,z)$ follows by verification of \eqref{eq:10}. All other identies can be checked directly. 
\end{proof}

		\section{Periodic modulation of power law canonical systems} \label{secPeriodicModulation}
		
		Let $\beta\in (-1,1)$ and $\rho>1$. In this section we discuss Herglotz functions $m_{\beta,\omega}$ of the form 
		\begin{align}\label{intro:eq21}
			m_{\beta,\omega}(z):=z^{\beta}\omega(z),
		\end{align}
		where $\beta\in(-1,1)$ and $\omega$ obeys $\omega(\rho z) = \omega(z)$. 
		In particular, we will provide an explicit construction leading to such Herglotz functions. 
		
		\subsection{Self-similar Herglotz functions and kernels}
		\begin{lemma}\label{lem:6}
			Fix $\rho > 1$ and $\beta \in (-1,1)$. A function $m:\bbC_+ \to \bbC$ obeys
			\begin{equation}\label{mselfsimilar}
				m(\rho z) = \rho^\beta m(z), \qquad  \forall z \in \bbC_+
			\end{equation}
			if and only if it is of the form $m(z) = z^\beta \omega(z)$ for some $\omega:\bbC_+ \to \bbC$ obeying \eqref{rhoperiodic}.
		\end{lemma}
		
		\begin{proof}
			If $m$ obeys \eqref{mselfsimilar}, taking $\omega(z) = z^{-\beta}m(z)$ gives
			\[
			\omega(\rho z)=\frac{m(\rho z)}{\rho^\beta z^\beta}=\frac{\rho^\beta m(z)}{\rho^\beta z^\beta}=\omega(z).
			\]
			The other implication is similar.
		\end{proof}

		Next, we show how to extract a cycle of kernels from a Herglotz function with the property \eqref{mselfsimilar}:
		
		\begin{lemma}\label{lem:8}
			Let $m \in \cN_0$ obey \eqref{mselfsimilar} for some $\beta \in (-1,1)$ and let $\cK$ denote the associated trace-normalized kernel. For $s\in [1,\infty)$ define the continuous family of kernels \eqref{eqnLimitKernelsMatrix1}, where $t(s)$ is the unique point in $(0,\infty)$ which satisfies \eqref{eq:increasingTS}.
			Then, for all $s \ge 1$,
			\[
			\cL_{s\rho}=\cL_s.
			\]
		\end{lemma}

		\begin{proof}
			Let $H$ denote the trace-normalized Hamiltonian corresponding to $m$. 
			In terms of Definition \ref{def:3}, we have
			\begin{align}\label{eq:24}
				\cL_s(z,w)=\cK_{\cF_s H}(t(s)/s,z,w),
			\end{align}
			for  $g(s)=s^\beta$. 
			
			Note that $t(s)$ is the unique value so that 
			\[
			(\cL_s)_{11}(0,0)(\cL_s)_{22}(0,0)=1.
			\]
			
			Since $g(s)g(t)=g(st)$, we have $\cF_s\cF_\rho=\cF_{s\rho}$. Denoting $m_s(z)=m_{\cF_sH}=g(s)m(z/s)$, then we have
			$g(\rho)m_s(z/\rho)=m_s(z)$ and hence by Lemma \ref{lem:18}, the kernels associated to $\cF_sH$ and to $\cF_\rho\cF_sH$ coincide up to reparametrization. Hence, using \eqref{eq:24} we find that there exists $\tilde t$ so that 
			\begin{align*}
				\cL_{s\rho}(z,w) = \cK_{\cF_\rho\cF_{s} H}(t(s\rho)/(s\rho),z,w)=\cK_{\cF_{s} H}(\tilde t,z,w)
			\end{align*}
			Since $t(s)$ is the unique $t>0$ so that $(\cL_{s})_{11}(0,0)(\cL_{s})_{22}(0,0)=1$, it follows that $\tilde t=t(s)s$ and hence $\cL_{s\rho}=\cL_s$. 
		\end{proof}

For the study of zeros, we need one more fact:

\begin{lemma}\label{lemmaInfinitelyManyZeros}
Assume that $m_{\beta,\omega}$ does not have an analytic continuation to $\bbC \setminus (-\infty,0]$ which is real-valued on $(0,\infty)$. Then $(\cK_1)_{22}(\cdot,0)$ has infinitely many zeros on $(0,\infty)$.
\end{lemma}

\begin{proof}
Assume that $(\cK_1)_{22}(\cdot,0)$ has finitely many zeros. We will use that $\cK_{\rho^n}=\cK_1$. Note that $\eqref{eq:increasingTS}$ implies that $t(\rho^n)$ is increasing and $t(\rho^n)\to\infty$ as $n\to\infty$. The relation \eqref{eqnLimitKernelsMatrix1} implies
\[
(\cK_1)_{22}(z,w)=\cK_{22}(t(1),z,w)=\frac{\cK_{22}(t(\rho^n),z/\rho^n,w/\rho^n)}{\rho^{n(1+\beta)}}
\]
Let $\xi_m$ denote the largest zero of $(\cK_1)_{22}(\cdot,0)$ on $(0,\infty)$. Then, for any $\epsilon>0$, for $n$ large enough $\cK_{22}(t(\rho^n),\cdot,0)$ has no zeros in $(\epsilon,\infty)$. Let $\fA(t,z)$ denote the transfer matrix generating $\cK$ by \eqref{eq:22}. An interlacing property as in the discussion preceding Theorem 10.1 in \cite{EichLukWor} shows  that $\fA_{22}(t(\rho^n),\cdot)$ has at most one zero in $(\epsilon,\infty)$. On the other hand choosing $\tau=\infty$ in \eqref{eq:21} shows that 
\[
m_{\beta,\omega}(z)=\lim\limits_{n\to\infty}\frac{\fA_{11}(t(\rho^n),z)}{\fA_{22}(t(\rho^n),z)}.
\]
Herglotz function arguments imply that $m_{\beta,\omega}$ has a meromorphic continuation to $\bbC \setminus (-\infty,\epsilon]$ with the symmetry $m_{\beta,\omega}(\bar z) = \overline{ m_{\beta,\omega}(z)}$ and with at most one pole on $(\epsilon,\infty)$. Since $\epsilon > 0$ was arbitrary it can be replaced by $0$. If $m_{\beta,\omega}$ has a pole at $\lambda > 0$, it would also have a pole at $\rho^n \lambda$, contradicting the uniqueness of the positive pole. Thus, this continuation is analytic on $\bbC \setminus (-\infty,0]$.
\end{proof}

\subsection{Rescaling limits}
The aim of this subsection is to prove rescaling limits for the matrix Christoffel-Darboux kernel associated to Herglotz functions which are asymptotically close to some $m_{\beta,\omega}$.  We begin by reinterpreting a statement about the normal behavior of a Herglotz function as a convergence statement in $\cN_0$.
		
		\begin{lemma}\label{lem:3}
			Let $w\in\cN_0$. Let $\beta\in (-1,1), \rho>1$ and $\omega:\bbC_+\to\bbC \setminus \{0\}$ obey \eqref{rhoperiodic}.  Let $m_{\beta,\omega}$ be given by \eqref{intro:eq21}. Then
			\begin{equation}\label{10augequiv}
				\lim\limits_{y\to 0} \lvert  e^{-i\beta\pi/2} y^{-\beta}w(iy)-\omega(iy) \rvert = 0
			\end{equation}
			if and only if uniformly on compact subsets of $\bbC_+$
			\begin{align}\label{eq:8}
				\lim\limits_{n\to\infty}w(z \rho^{-n}) \rho^{n\beta}=m_{\beta,\omega}(z).
			\end{align}
			In particular, if this holds, then $m_{\beta,\omega}$ is a Herglotz function.
		\end{lemma}
		
		\begin{proof} We write $r_n = \rho^n$.
			
			$\Rightarrow$:	By compactness and analyticity it suffices to show that for $y>0$ 
			\begin{align}\label{eq:7}
				\lim\limits_{n\to\infty}|w(iy/r_n)r_n^\beta-m_{\beta,\omega}(iy)|=0.
			\end{align}
			Set $\xi_n=y/r_n$. Then using that $\omega(i\xi_n)=\omega(iy)$, we get
			\begin{align*}
				|w(iy/r_n)r_n^\beta-m_{\beta,\omega}(iy)|&=|w(i\xi_n)(iy/i\xi_n)^\beta-(iy)^\beta\omega(i\xi_n)|\\
				&=y^\beta|w(i\xi_n)(i\xi_n)^{-\beta}-\omega(i\xi_n)|.
			\end{align*}
			Hence, by assumption we get \eqref{eq:7}.
			
			$\Leftarrow$: Taking $z = it$ with $t \in [1, \rho]$, we have uniform convergence
			\[
			\lim_{n\to\infty} w(it \rho^{-n}) \rho^{n\beta} = (it)^\beta \omega(it)
			\] 
			so dividing by $(it)^\beta$ and using $\omega(it) = \omega(it \rho^{-n})$ we rewrite this as
			\[
			\lim_{n\to\infty} \sup_{t \in [1,\rho]} \lvert w(it \rho^{-n}) (it \rho^{-n})^{-\beta} - \omega(it \rho^{-n}) \rvert = 0.
			\]
			Using $y = t \rho^{-n}$ this is logically equivalent to \eqref{10augequiv}.
			
			When \eqref{eq:8} holds, it gives $m_{\beta,\omega}$ as a locally uniform limit of Herglotz functions; since $m_{\beta,\omega}$ is nonzero, it must be a Herglotz function.
		\end{proof}

		\begin{lemma}\label{lem:4}
			Uniform on compact subsets of $\bbC_+$ convergence
			\[
			\rho^{n\beta} w (z / \rho^n) \to m(z), \qquad n\to\infty
			\]
			implies that:
			\begin{enumerate}[(a)]
				\item For any $a \in [1,\rho]$, with $r_n = a \rho^n$,
				\[
				\lim_{n\to\infty}  r_n^\beta w(z/r_n) = a^\beta m(z/a).
				\]
				\item Let $r_n$ be a sequence with $r_n\to\infty$. Let $k_n=\lfloor \log_\rho r_n\rfloor$ and $a_n=\rho^{\{log_\rho r_n\}}\in[1,\rho)$, i.e., $r_n=\rho^{k_n}a_n$. Choose a subsequence so that $a_{n_j}\to a\in[1,\rho]$.  Then
				\[
				\lim_{j\to\infty}  r_{n_j}^\beta w(z/r_{n_j}) = a^\beta m(z/a)
				\]
			\end{enumerate}
		\end{lemma}
		
		\begin{proof}
			(a) is immediate from the assumption.
			
			(b) We have
			\[
			\lim_{j\to\infty} r_{n_j}^\beta w(z / r_{n_j}) = \lim_{j\to\infty} a_{n_j}^\beta \rho^{k_{n_j} \beta} w(z / (a_{n_j} \rho^{k_{n_j}} )) = a^\beta m(z / a)
			\]
			as a consequence of uniform convergence on compacts and $a_{n_j} \to a$.
		\end{proof}

The main result of this section is the following theorem about the scaling behavior of kernels.

		\begin{theorem}\label{thm:3}
			Let $H\in\bbH$ and $m_H$ the associated Herglotz function. Let $\beta\in(-1,1)$, $\rho>1$ and $\xi\in\bbR$. 
			Assume that 
			\begin{align}\label{eq:25}
				\lim\limits_{y\to 0}|{(iy)}^{-\beta}m_H(\xi+iy)-\omega(iy)|=0
			\end{align}
			for some $\rho$-periodic $\omega$ which is not identically $0$. Let
			\[
			L_H=\inf\{L>0\mid \cK_{11}(L,\xi,\xi)\cK_{22}(L,\xi,\xi)>0\}.
			\]
			We have $L_H<\infty$ and define for $L>L_H$
			\[
			r(L):=\sqrt{\cK_{11}(L,\xi,\xi)\cK_{22}(L,\xi,\xi)}.
			\] 
			Then $r: (L_H,\infty) \to (0,\infty)$ is strictly increasing, continuous, $\lim_{L\to\infty} r(L) = \infty$, and uniformly for $(z,w)$ in compact subsets of $\bbC\times\bbC$. 	
			\begin{align}\label{eq:18}
				\lim\limits_{L\to\infty}\left(\frac{1}{r(L)}\cUnew(r(L)^\beta)\cK_H\left(L,\xi+\frac{z}{r(L)},\xi+\frac{w}{r(L)}\right)\cUnew(r(L)^\beta)^*-\cL_{r(L)}(z,w)\right)=0
			\end{align}
where $\cL_r(z,w)$ is the limit cycle associated to the Herglotz function $m_{\beta,\omega}(z)=z^\beta \omega(z)$ by Lemma \ref{lem:8}.
\end{theorem}
		
\begin{remark}
As explained in Section \ref{sec:Jacobi}, for every $\mu$ corresponding to a determinate moment problem there exists a canonical system $H$, so that $\cK_H=\cK_\mu$, where $\cK_\mu$ is defined in \eqref{intro:eq3}. Thus, Theorem \ref{thm:4}  is a special case of Theorem \ref{thm:3}.
\end{remark}
	
	We divide the proof into two lemmas.
		
		\begin{lemma}\label{lem:1}
			In the notation and with the assumptions of Theorem \ref{thm:3}  we have $L_H<\infty$ and for $L>L_H$, $r(L)$ is strictly increasing, continuous, and $\lim_{L\to\infty}r(L)=\infty$.
		\end{lemma}
		
		\begin{proof}
		By a shift we may assume $\xi=0$. 
			We have 
			\begin{equation}\label{eqn:5aug}
			\begin{pmatrix}
				\cK_{11}(L,0,0)& \cK_{12}(L,0,0)\\\cK_{21}(L,0,0)&\cK_{22}(L,0,0)
			\end{pmatrix}
			=\int_0^LH(s)ds.
			\end{equation}
Thus, $\cK_{11}(L,0,0)=0$ for all $L\in [0,\infty)$ if and only if $H$ is a reparametrization of $H_2=(\begin{smallmatrix}
				0&0\\0&1
			\end{smallmatrix})$. A direct computation shows that $m_{H_2}\equiv 0$, which contradicts our assumption that $\omega$ does not vanish identically. Similarly, $\cK_{22}(L,0,0)=0$ for all $L\in [0,\infty)$ if and only if $H$ is a reparametrization of $H_1=(\begin{smallmatrix}
				1&0\\0&0
			\end{smallmatrix})$, with Weyl function $m_{H_1}\equiv \infty$ again contradicting \eqref{eq:25}. 
			
			Since $H\geq 0$ and $H\neq 0$ $a.e.$ on $[0,\infty)$, we get that  as functions of $L$, $\cK_{11}(L,0,0), \cK_{22}(L,0,0)$ are increasing and their sum is strictly increasing. Since for $L>L_0$, both are postive, this implies that $r(L)$ is strictly increasing. Continuity of $r(L)$ follows from \eqref{eqn:5aug}. Since $\tr H\notin L^1([0,\infty))$, we conclude that $r(L)\to\infty$ as $L\to\infty$. 
		\end{proof}
		
		Let $k_r=\lfloor \log_\rho r\rfloor$ and $a_r=\rho^{\{log_\rho r\}}\in[1,\rho)$.
		The key result will be the following lemma, which operates along suitable subsequences:
		
		\begin{lemma}\label{lem:5}
			In the notation and with the assumptions of Theorem \ref{thm:3} with $\xi=0$, let $a(L)=a_{r(L)}$ and $L_n$ be so that $a(L_n)\to a_\infty\in[1,\rho]$. Let $\cL_{a_\infty}$ be defined as in Lemma \ref{lem:8} associated to the Herglotz function $m_{\beta,\omega}$. Then it holds that
			\begin{align}\label{eq:104}
				\lim\limits_{n\to\infty}\frac{1}{r(L_n)}\cUnew(r(L_n)^\beta)\cK_H\left(L_n,\frac{z}{r(L_n)},\frac{w}{r(L_n)}\right)\cUnew(r(L_n)^\beta)^*=\cL_{a_\infty}(z,w)
			\end{align}
			uniformly for $(z,w)$ in compact subsets of $\bbC\times\bbC$. 
		\end{lemma}
		\begin{proof}
			We divide the proof into two steps.
			
			\textbf{Step 1:} Let $g(r)=r^\beta$ and $\cF_r$ the corresponding transformation and set $m_r(z)=r^{\beta}m_H(z/r)=m_{\cF_rH}(z)$. Moreover, let $\hat H_r\in\hat\bbH$ be the trace normalized reparametrization of  $\cF_rH$ and $\hat \cK_r=\cK_{\hat H_r}$. Choose a subsequence so that $m_{r_n}$ converges to some $m_\infty\in \cN_0$ and assume that $m_\infty$ is not identically $0$ or $\infty$. Define $X=\{r_n\mid n\in\bbN\}\cup\{\infty\}$. By construction $(m_{r})_{r\in X}$ is a continuous family. By Theorem \ref{thm:1} the same holds for $(\hat H_r)_{r\in X}$ and $(\hat \cK_r)_{r\in X}$ with respect to the topologies defined in Section \ref{sec:cont}.
			Define on $X\times [0,\infty)$
			\[
			\sigma(r,t)=(\cK_{\hat H_r})_{11}(t,0,0)(\cK_{\hat H_r})_{22}(t,0,0).
			\]
			Then $\sigma$ is jointly continuous in $(r,t)$. Define $t_-(r)=\inf\{t>0\mid \sigma(r,t)>0\}$. Since in \eqref{eq:25}, $\omega$ is not identically $0$, and we assumed that our $m_\infty$ is not identically $0$ or $\infty$, we conclude that $m_r$ is not identically $0$ or $\infty$ for all $r\in X$.  Then as in the proof of Lemma \ref{lem:1}, we get that $t_-(r)<\infty$ and for fixed $r\in X$, $\sigma(r,\cdot)$ is an increasing bijection from $(t_-(r), \infty)$ onto $(0,\infty)$. In particular, there is a unique $t_1(r)$ so that $\sigma(r,t_1(r))=1$. We claim that 
			\begin{align}\label{eq:102}
				\sup_{r\in X}t_1(r)=T<\infty.
			\end{align}
			Indeed, since for fixed $r\in X$, $\lim_{t\to\infty}\sigma(r,t)=\infty$, monotonicity and Dini's theorem imply that
			\[
			\lim_{t\to\infty}\inf_{r\in X}\sigma(r,t)=\infty, 
			\]
			which implies \eqref{eq:102}. Compactness of $[0,T]$, and continuity $\sigma$ and the uniqueness claim about $t_1(r)$ imply that 
			\begin{align}\label{eq:103}
				\lim_{n\to\infty}t_1(r_n)= t_1(\infty).
			\end{align}
			
			\textbf{Step 2:} Let $r(L)=\sqrt{\cK_{11}(L,0,0)\cK_{22}(L,0,0)}$. By assumption we have $a(L_n)\to a_\infty$ and now Lemma \ref{lem:3} and Lemma \ref{lem:4} imply that 
			$m_{r(L_n)}$ converge to $a_\infty^\beta m_{\beta,\omega}(z/a_\infty)$.  Let $\hat H_{r(L_n)}, \hat \cK_{\hat H_{r(L_n)}}$ be as above. We have 
			\begin{align} \nonumber
				\left\{ \cK_{\hat H_{r(L_n)}}(t,z,w)\mid t\geq 0 \right\}=\hspace{6.6cm}\\ \left\{\frac{1}{r(L_n)}\cUnew(r(L_n)^\beta)\cK_H\left(t,\frac{z}{r(L_n)},\frac{w}{r(L_n)}\right)\cUnew(r(L_n)^\beta)^*\mid t\geq 0 \right\}.
			\end{align}
			Hence, there exists $t_n$ so that 
			\[
			\cK_{\hat H_{r(L_n)}}(t_n,z,w)=\frac{1}{r(L_n)}\cUnew(r(L_n)^\beta)\cK_H\left(L_n,\frac{z}{r(L_n)},\frac{w}{r(L_n)}\right)\cUnew(r(L_n)^\beta)^*.
			\]
			Evaluating at $z=w=0$, shows that $t_n=t_1(r(L_n))$. We conclude \eqref{eq:104} from continuity of the family $(\cK_{\hat H_{r(L_n)}})_{r(L_n)\in X}$ and \eqref{eq:103}.
		\end{proof}
		
		The proof of Theorem \ref{thm:3} is now a consequence of Lemma \ref{lem:5} and compactness arguments.
		
		\begin{proof}[Proof of Theorem \ref{thm:3}]
We start by reducing it to the case $\xi=0$. For $\xi\in\bbR$, consider the transfer matrix 
			\[
			W(t,z)=\fA(t,z+\xi)\fA(t,\xi)^{-1}.
			\]
			and the Hamiltonian
			\[
			H_\xi(t)=\fA(t,\xi)H(t)\fA(t,\xi)^*,.
			\]
			Using that $\fA(t,\xi)J\fA(t,\xi)^*=J$, a direct computation shows that 
			\[
			\partial_t	W(t,z)J=	zW(t,z)H_\xi(t).
			\]
			and 
			\[
			m_{H_{\xi}}(z)=m_H(z+\xi).
			\]
			Note also that $\tr H\in L^1([0,\infty))$ if and only if $\tr H_\xi\in L^1([0,\infty))$. Due to these transformations it suffices to consider $\xi=0$. 
			
			Since $r(L)=\rho^{k}a(L)$, for some $k$, periodicity of $\cL_s$ implies that $\cL_{r(L)}=\cL_{a(L)}$. Assume that \eqref{eq:18} does not hold. Then there exists a compact $C \subset \bbC$, $\e>0$ and $L_n\to\infty$ so that 
			\begin{align}\label{eq:106}
				\sup_{z,w\in C}\left|\frac{1}{r(L_n)}\cUnew(r(L_n)^\beta)\cK_H\left(L_n,\frac{z}{r(L_n)},\frac{w}{r(L_n)}\right)\cUnew(r(L_n)^\beta)^*-\cL_{a(L_n)}(z,w)\right|>\e.
			\end{align}
			Now choose a subsequence so that $a(L_{n_k})$ converges to some $a_\infty \in [1,\rho]$. Then, by Lemma \ref{lem:5} and continuity of the family $(\cK_a)_{a\in[1,\rho]}$ both expressions in \eqref{eq:106} converge to 
			$\cL_{a_\infty}$, which is a contradiction. 
		\end{proof}

\subsection{Fine structure of zeros from periodically modulated limits of Weyl functions}

The Freud--Levin theorem \cite{Freud69,LevinLubinsky08} was originally stated as the result that clock behavior follows from convergence of rescaled CD kernels to the sine kernel. The proof, however, allows for significant generalizations; see \cite[Section 10]{EichLukWor} for a general statement in case of convergence of rescaled kernels to an essentially arbitrary limit kernel.

For our current purposes, we need to look beyond convergence and generalize the Freud--Levin theorem to give uniform bounds under precompactness assumptions. To formulate this, consider a sequence of Hermite--Biehler functions $E_n$ (entire functions such that $\lvert E_n(z)\rvert > \lvert E_n(\ol z)\rvert$ for $z\in \bbC_+$) with no real zeros. We use the notation
\[
E^\sharp(z) = \ol{ E( \ol z) }
\]
and consider the associated reproducing kernels
\[
K_n(z,w) = i \frac{E_n(z) E_n^\sharp(\ol w) - E_n^\sharp(z) E_n(\ol w)}{z-\ol w}
\]
The function $E_n + E_n^\sharp$ has only real simple zeros, which we label $\xi_j^{(n)}$ be counting from $\xi$, so that
\[
\dots < \xi_0^{(n)} < \xi \le \xi_1^{(n)} < \xi_2^{(n)} < \dots
\] 
In a typical application to orthogonal polynomials, we use $E_n =  p_n + i p_{n-1}$, in which case $\xi_j^{(n)}$ are precisely zeros of $p_n$.

For a kernel $K_E$, if the function $K_E(\cdot,0)$ has infinitely many zeros on $(0,\infty)$, then by interlacing, so does $K_E(\cdot,\lambda)$ for any $\lambda \in \bbR$, and we denote by
\[
\kappa_1(\lambda,K_E) < \kappa_2(\lambda,K_E) < \dots
\]
all the zeros of $K_E(\cdot,\lambda)$ in $(\lambda,\infty)$.
\begin{theorem}\label{theoremFreudLevinPrecompact}
Assume that for some $\xi \in \bbR$ and some sequence $\tau_n \to \infty$, the sequence of rescaled kernels
\[
\frac{K\left(n,\xi+\frac{z}{\tau_n},\xi+\frac{w}{\tau_n}\right)}{K\left(n,\xi,\xi\right)}
\]
is precompact in the topology of uniform convergence on compacts. Denote by $\cK$ the set of its limit points.  Assume that for each $K_E \in \cK$, the function $K_E(\cdot,0)$ has infinitely many zeros on $(0,\infty)$. 
\begin{enumerate}[(a)]
\item 
\begin{equation}\label{eqnFirstEVupperbound}
\limsup_{n\to\infty} \tau_n (\xi_1^{(n)} - \xi) \le  \max_{K_E \in \cK}  \kappa_1(0,K_E).
\end{equation}
\item For every $1 \le k < l$, there exist constants $C_{k,l}^\pm \in (0,\infty)$ such that
\begin{equation} \label{eqnLiminfLimsup}
C_{k,l}^- \le \liminf_{n\to\infty} \tau_n (\xi_l^{(n)} - \xi_k^{(n)}) \le \limsup_{n\to\infty} \tau_n (\xi_l^{(n)} - \xi_k^{(n)}) \le C_{k,l}^+.
\end{equation}
The constants are determined by the family of limit kernels as follows:
\begin{align}
C_{k,l}^- & = \min_{K_E \in \cK} \min_{\lambda \in [0, \kappa_1(0,K_E)]} \Bigl( \kappa_{l-1}(\lambda,K_E) -  \kappa_{k-1}(\lambda,K_E) \Bigr) \label{eqnCklminus} \\
C_{k,l}^+ & = \max_{K_E \in \cK} \max_{\lambda \in [0, \kappa_1(0,K_E)]} \Bigl( \kappa_{l-1}(\lambda,K_E) -  \kappa_{k-1}(\lambda,K_E) \Bigr) \label{eqnCklplus}
\end{align}
\item If, in addition, for each $K_E \in \cK$, the function $K_E(\cdot,0)$ has infinitely many zeros on $(-\infty,0)$, then (b) holds for all $k < l$ in $\bbZ$. 
\end{enumerate}
\end{theorem}

\begin{proof}
(a) Along a subsequence $n_j$ such that for some $K_E \in \cK$,
\begin{equation}\label{eqnIOGSHE}
\lim_{j\to\infty} \frac{K\left(n_j,\xi+\frac{z}{\tau_{n_j}},\xi+\frac{w}{\tau_{n_j}}\right)}{K\left(n_j,\xi,\xi\right)} = K_E(z,w),
\end{equation}
by the generalization of the Freud--Levin theorem provided in \cite[Theorem 10.1]{EichLukWor},
\[
\limsup_{j\to\infty} \tau_{n_j} (\xi_1^{(n_j)} - \xi) \le  \kappa_1(0,K_E).
\]
 Thus, by precompactness of the rescaled kernels, \eqref{eqnFirstEVupperbound} holds.

 (b) First note that the constants defined in \eqref{eqnCklminus}, \eqref{eqnCklplus} are well-defined by compactness of $\cK$ and the continuity of zeros (by Hurwitz's theorem) and lie in $(0,\infty)$.  Take a subsequence $n_j$ such that \eqref{eqnIOGSHE} holds and that for some $\lambda \in [0, K_E(\cdot,0)]$,
 \[
 \lim_{j\to\infty}  \tau_{n_j} (\xi_1^{(n_j)} - \xi) = \lambda.
 \]
Then by \cite[Theorem 10.1]{EichLukWor}, 
\[
\lim_{j\to\infty} \tau_{n_j}  (\xi_l^{(n_j)} - \xi_k^{(n_j)}) = \kappa_{l-1}(\lambda,K_E) -  \kappa_{k-1}(\lambda,K_E).
\]
Using compactness completes the proof.

(c) is proved analogously to (b) if the limit kernels have infinitely many negative zeros.
\end{proof}

\begin{proof}[Proof of Corollary~\ref{corollaryZerosGeneral}]
(a), (b): by Lemma~\ref{lemmaInfinitelyManyZeros}, the limit kernels have infinitely many zeros on each half-line, except in the case of an edge universality class. Thus, the claims follow from Theorem~\ref{theoremFreudLevinPrecompact}.

(c): We set $\xi = 0$ and denote by $\nu$ the even measure on $\bbR$ whose pushforward by $x\mapsto x^2$ is the measure $\mu$. The corresponding Weyl function is computed to be 
\[
g(z) = \int \frac 1{x-z}\,d\nu(x) = z m(z^2).
\]
Using Lemma~\ref{lem:3}, we rephrase \eqref{intro:eq14} as
\[
\lim_{n \to \infty} m( z \rho^{-n} ) \rho^{n\beta} = z^\beta \omega(z)
\]
and see that this implies
\[
\lim_{n \to \infty} g( z  \rho^{-n/2} ) \sqrt{\rho}^{n(1+2\beta)} = z^{1+2\beta} \omega(z^2)
\]
In words, $g$ obeys the same type of scaling behavior, with the constant $\beta$ replaced by $\tilde \beta = 1+2\beta$ and the multiplicatively periodic function $\omega$ by $\tilde\omega(z) = \omega(z^2)$.

Since $\nu$ is even and its pushforward is $\mu$,  orthogonal polynomials of even degree for the measure $\nu$ are precisely $p_{2n}(z, \nu) = p_n(z^2, \mu)$, so their zeros are $\pm \sqrt{ \xi_j^{(n)}}$, with $1\le j \le n$  (compare \cite[Lemma 10.3]{EichLukWor}). Moreover, $p_{2n-1}(\cdot,\nu)$ are odd polynomials, so $K_{2n}(0,0;\nu) = K_n(0,0;\mu)$. Thus, by (b) applied to the measure $\nu$, we have
\[
\sqrt{\xi_l^{(n)}} - \left( - \sqrt{\xi_l^{(n)}}  \right) \asymp K(2n,0,0;\nu)^{-1/(1+\tilde \beta)}
\]
for $l=1,2,\dots,n$. Squaring and using $2 / (1+\tilde \beta) = 1 / (1+ \beta)$ concludes the proof.
\end{proof}

\subsection{The middle third Cantor measure and the a.c.\ example}

\begin{proof}[Proof of Theorem~\ref{thmCantorMeasure}]
Let $\mu$ denote the middle third Cantor measure. Let us first assume that $\xi$ is a periodic point for the dynamics. Equivalently, there exists $q \in \bbN$ and $c_1,\dots, c_q \in \{0,1\}$ such that the linear map
\[
g(\lambda) = \frac{\lambda}{3^q} + \sum_{j=1}^q \frac{2c_j}{3^j}
\]
has fixed point $\xi$. Under this map, in some neighborhood $U$ of $\xi$, the pushforward measure $g_*(\mu)$ is equal to $2^q \mu$.  Shifting $\xi$ to $0$ and using this self-similarity to extend the measure to the real line, we obtain a unique measure $\nu$ with the properties
\[
\nu( (a,b]) = \mu((\xi+a,\xi+b]), \qquad  \forall (\xi+a,\xi+b] \subset U
\]
and $2^{-q} \nu((3^q a, 3^q b]) = \nu((a,b])$ for all $(a,b] \subset \bbR$. This rescaling property implies
\begin{align*}
\int \frac 1{1+\lvert x\rvert}\,d\nu(x) & \le \nu([-1,1]) + \sum_{k=0}^\infty \frac{\nu(\{x \mid 3^{kq} < \lvert x \rvert \le 3^{kq+q}\})}{3^{kq}}  \\
& = 1 +  \sum_{k=0}^\infty \frac{2^{kq}}{3^{kq}}  \nu(\{x \mid 1 < \lvert x \rvert \le 3^{q}\}) < \infty
\end{align*}
so we can define Herglotz functions
\begin{equation}\label{eqnmf}
m(z) = \int \frac 1{x-z} \,d\mu(x), \qquad f(z) = \int \frac 1{x-z} \,d\nu(x).
\end{equation}
Using the rescaling property
\[
m\left( \xi + \frac z{3^{nq}} \right)  = \int \frac 1{x-z/3^{nq}} \,d\mu(x) = \left(\frac 32\right)^{nq} \int_{-3^{nq}}^{3^{nq}} \frac 1{y-z} \,d\nu(y)
\]
and by dominated convergence we conclude
\[
\lim_{n\to\infty}  \left(\frac 23\right)^{nq} m\left( \xi + \frac z{3^{nq}} \right)  = f(z)
\]
uniformly on compact subsets of $\bbC_+$. By Lemma~\ref{lem:3} and Theorem~\ref{thm:4}, we conclude limit cycle behavior at $\xi$.

Next, we note that
\[
m\left( \frac 23 + \frac w3\right) - \frac 32 m(w) = \int_0^{1/3} \frac 1{x - \left( \frac 23 + \frac w3\right)} \,d\mu(x) = o(1), \quad w \to \xi \in [0,1].
\]
Thus,
\[
m\left( \frac {2+\xi}3 + \frac z3\right) - \frac 32 m(\xi + z) = o(1), \qquad z \to 0.
\]
Analogously,
\[
m\left( \frac {\xi}3 + \frac z3\right) - \frac 32 m(\xi + z) = o(1),  \qquad z \to 0.
\]
Due to this, if
\[
\left( \frac 23 \right)^{nq} m\left( \xi + \frac z{3^{nq}} \right) \to f(z), \qquad n \to\infty
\]
then for $c\in \{0,1\}$,
\[
\left( \frac 23 \right)^{nq} m\left( \frac{2c+\xi}3 + \frac z{3^{nq}} \right) \to \frac 32 f(3z), \qquad n \to\infty.
\]
By applying this finitely many times, we reach any eventually periodic point. These steps don't change the function $f$, so they don't change the universality class.
\end{proof}

\begin{proof}[Proof of Example~\ref{xmplACLimitCycle}]
Similarly to the previous proof, we now define
\[
d\nu(x) = \sum_{n\in \bbZ}  \left( \frac 32 \right)^{n+1} \chi_{(2/3^{n+1}, 1/3^n)}(x) \,dx,
\]
and note that $\nu((0,3^n]) = 2^n$ so $\int \frac 1{1+x}\,d\nu(x) < \infty$. Thus, with the notation \eqref{eqnmf}, we have
\[
\left(\frac 23 \right)^n m(z/3^n) = \left(\frac 23 \right)^n \int_0^1 \frac 1{x-z/3^n} \,d\mu(x) = \int_0^{3^n} \frac 1{x-z}\,d\nu(x) \to f(z)
\]
as $n\to\infty$.
\end{proof}

\subsection{Examples in terms of comb mappings}\label{sec:Comb1}
		We provide an explicit construction in terms of certain comb mappings, which leads to a Herglotz function $m$ satisfying \eqref{mselfsimilar}.
		
		First we need to construct a comb domain with a self-similarity property. To construct such a domain, we can choose in an arbitrary way parameters $\lambda > 1$, and for each choice of sign $\pm$, a sequence of $\eta_\pm^k \in [1, \lambda)$ with no repetitions, and a sequence of $h_\pm^k > 0$, $1 \le k \leq n_\pm$. Here each $n_\pm$ may be finite or infinite; if it is infinite we impose the condition
		\begin{equation}\label{eq:12}
			\lim_{k\to\infty} h_\pm^k = 0.
		\end{equation}
		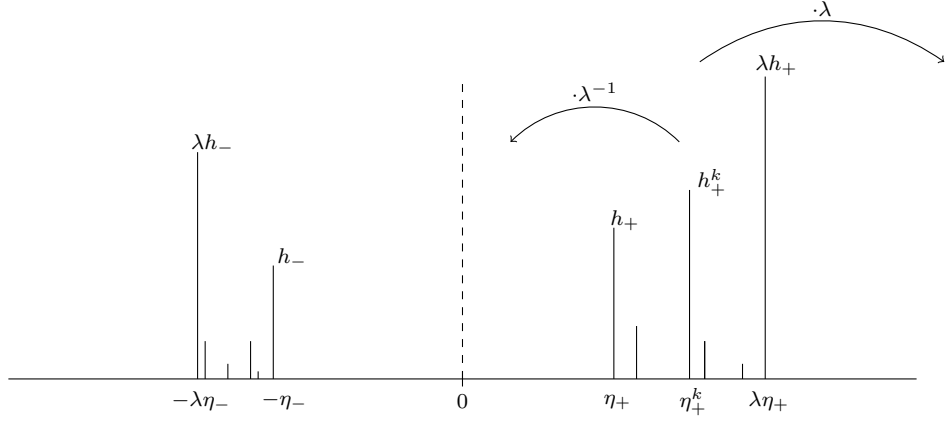
\begin{figure}[ht]
			\centering
			\begin{tikzpicture}
				\draw[-] (-6,0) -- (6,0) ;
				\draw[dashed] (0,0) -- (0,4) ;
				
				\draw[-] (0,-0.1) -- (0,0.1) ;
				\node at (0, -.3)   {\scriptsize$0$};

				\draw[-] (2,0) -- (2,2) ;
				\draw[-] (4,0) -- (4,4) ;
				\draw[-] (3,0) -- (3,2.5) ;
				\draw[-] (3.2,0) -- (3.2,0.5) ;
				\draw[-] (3.2,0) -- (3.2,0.5) ;
				\draw[-] (3.7,0) -- (3.7,0.2) ;
				\draw[-] (2.3,0) -- (2.3,0.7) ;

				\draw[-] (-2.5,0) -- (-2.5,1.5) ;
				\draw[-] (-3.5,0) -- (-3.5,3) ;
				\draw[-] (-2.7,0) -- (-2.7,0.1) ;
				\draw[-] (-2.8,0) -- (-2.8,0.5) ;
				\draw[-] (-3.4,0) -- (-3.4,0.5) ;
				\draw[-] (-3.1,0) -- (-3.1,0.2) ;

				\node at (2.05, -.3)   {\scriptsize$\eta_+$};
				\node at (4.05, -.3)   {\scriptsize{$\lambda\eta_+$}};
				\node at (3.05, -.3)   {\scriptsize{$\eta_+^k$}};
				\node at (3.3,2.6)   {\scriptsize{$h_+^k$}};
				\node at (2.15, 2.1)   {\scriptsize{$h_+$}};
				\node at (4.15, 4.15)   {\scriptsize{$\lambda h_+$}};
				
				\node at (-2.35, -.3)   {\scriptsize$-\eta_-$};
				\node at (-3.45, -.3)   {\scriptsize{$-\lambda\eta_-$}};

				\node at (-2.25,1.6)   {\scriptsize$h_-$};
				\node at (-3.3,3.1)   {\scriptsize$\lambda h_-$};

				\node  (A) at (3,3) {};
				\node (B) at (0.5,3)  {};
				
				\draw [->] (A) to [bend right=45] (B);
				\node at (1.75,3.8)   {\scriptsize$\cdot\lambda^{-1}$};
				
				\node  (C) at (3,4.1) {};
				\node (D) at (6.5,4.1)  {};
				
				\draw [->] (C) to [bend left=35] (D);
				\node at (4.75,4.9)   {\scriptsize$\cdot\lambda$};
			\end{tikzpicture}
			\caption{The comb-domain $\Pi$ related to $\sE$}
			\label{fig:1}
		\end{figure}
		
		Under these assumptions, the set
		\begin{align}\label{eq:6}
			\Pi=\bbC_+\setminus \bigcup_{n\in \bbZ} \bigg( \bigcup_{k=1}^{n_-} ( -\lambda^n \eta_-^k,  \lambda^n (-\eta_-^k + i h_-^k)] \cup  \bigcup_{k=1}^{n_+} ( \lambda^n \eta_+^k,  \lambda^n (\eta_+^k + i h_+^k)]  \bigg)
		\end{align}
		is a comb domain with the self-similar property
		\begin{align}\label{eq:14}
			\Pi=\lambda \Pi.
		\end{align}
		An illustration of this construction is given in Figure \ref{fig:1}. The domain $\Pi$ is simply connected and due to the assumption \eqref{eq:12} the boundary is locally connected. By Carath\'eodory's theorem (sometimes called the continuity theorem \cite[Chapter 2]{PommerenkeBook}) this ensures that conformal maps from $\bbC_+$ to $\Pi$ can be continuously extended to the boundary. Since $0$ and $\infty$ are not cut points, i.e., $\partial\Pi \setminus \{0\}$ and $\partial\Pi \setminus \{\infty\}$ are connected, these values are taken precisely once \cite[Chapter 2]{PommerenkeBook}. Thus, there is a conformal map $\theta:\bbC_+\to\Pi$ such that $\theta(0)=0$ and $\theta(\infty)=\infty$, and it is unique up to rescaling of $\bbC_+$ (i.e. all other such maps are of the form $  \theta(\rho z)$ for some $\rho >0$). Since $\theta$ can be continuously extended to the boundary, we can define
		\begin{align}\label{eq:5}
			\sE=\theta^{-1}(\bbR)=\bbR\setminus\bigcup_{k\in\bbZ}(\sa_k,\sb_k),
		\end{align}
		where each gap $(\sa_k,\sb_k)$ is the preimage of some needle of the comb. 
		
		\begin{lemma}\label{lem:2}
			There exists $\rho> \lambda$ such that 
			\begin{align}\label{eq:15}
				\lambda\theta(z)=\theta(\rho z),\quad z\in\bbC_+.
			\end{align}
		\end{lemma}
		\begin{proof}
			Due to \eqref{eq:14}, the map $ \lambda \theta(z)$ is a conformal map from $\bbC_+$ to $\lambda \Pi = \Pi$. Thus, it is of the form $\theta(\rho z)$ for some $\rho > 0$. Since $\lambda>1$, iterating \eqref{eq:15} shows that $\theta(\rho^n i) = \lambda^n \theta(i) \to \infty$ as $n\to\infty$, so $\rho>1$. 
			
			Using the scaling $\theta(\rho^n z) = \lambda^n \theta(z)$ for $z \in \ol{\bbC_+}$ with $1 \le \lvert z \rvert \le \rho$ shows that $\lvert \theta(z) \rvert \asymp \lvert z \rvert^{1+\beta}$ as $z \to \infty$, where $1+\beta = \ln \lambda / \ln \rho$. In particular, proving $\rho > \lambda$ is equivalent to proving $\lvert \theta(z) \rvert = o(\lvert z \rvert)$ as $z\to\infty$.
			
			To prove that $\rho  > \lambda$, it suffices to consider the simplest configuration with only one slit extended by a geometric progression (the case $n_- + n_+ = 1$). Indeed, for a general pattern, let $\hat \Pi$ be the domain with a single generating slit at $\eta^1_+$ of the height $h^1_+$ and the same multiplier $\lambda$ (if $n_+=0$, use $\eta^1_-$ and $h^1_-$ instead). Denote by $\hat\theta : \bbC_+ \to \hat\Pi$ the corresponding conformal map. Then $\Pi\subset \hat \Pi$, so the function
			\[
			w(z)=\hat\theta^{-1}(\theta(z))
			\]
			maps $\bbC_+$ into itself. A Herglotz function can grow at most as $z$. Therefore, if  $\lvert \hat\theta(z)\rvert \asymp \lvert z \rvert^{1 + \hat\beta}$ with $\hat\beta < 0$, then 
			\[
			\lvert \theta(z) \rvert = \lvert \hat\theta(w(z)) \rvert \asymp \lvert z \rvert^{1+\beta}, \qquad z \to\infty
			\]
			with $\beta \le \hat\beta < 0$.
			
			So now let $\theta$ be generated by a domain with a single generating slit. Since $\Im\theta$ is a positive harmonic function in $\bbC_+$, it has a Poisson representation
			\[
			\Im \theta(z) = a \Im z + \int \frac{ \Im z}{ \lvert x - z \rvert^2} \,d\sigma(x)
			\]
			for some $a\geq0$ and positive measure $\sigma$ with $\int \frac{d\sigma(x)}{1+x^2}<\infty$. Therefore
			\[
			\lambda ^{-n}\Im \theta(i)=\Im \theta(i\rho^{-n})= a\rho^{-n}+\int \frac{\rho^{-n}}{x^2+\rho^{-2n}}d\sigma(x)
			\]
			so
			\[
			\left(\frac{\lambda} {\rho}\right)^n=\frac{\Im \theta(i)}{a+\int \frac{d\sigma(x)}{x^2+\rho^{-2n}}}\le  \frac{\Im \theta(i)}{a+\int_0^1 \frac{d\sigma(x)}{x^2+\rho^{-2n}}}.
			\]
			By monotonicity we can pass to the limit
			\[
			\lim_{n\to\infty} \left(\frac{\lambda} {\rho}\right)^n\le  \frac{\Im \theta(i)}{a+\int_0^1 \frac{d\sigma(x)}{x^2}}.
			\]
			
			We estimate the last integral: Since $\theta$ can be continuously extended to $\bbR$, Stieltjes inversion \cite[Proposition 7.43]{LukicFirstCourse} shows that $\sigma$ is absolutely continuous with density 
			\begin{align}\label{eq:8b}
				d\sigma(x)=\frac{1}{\pi}\Im \theta(x)dx. 
			\end{align}
			Let $\eta_0,h_0$ be the data of $\Pi$ as in \eqref{eq:6} and $(\sa_0,\sb_0)=\theta^{-1}((\eta_0,\eta_0+ih_0])$ and assume without loss of generality that $(\sa_0,\sb_0)\subset [0,1]$. Due to \eqref{eq:15} we have $\sa_k=\rho^k \sa_0, \sb_k=\rho^k \sb_0$, $k\in\bbZ$. From \eqref{eq:8b} we conclude that $\sigma$ is supported in the gaps $\sa_k,\sb_k$.
			
			The function $\Im \theta$ is a Martin function for $\bbC\setminus \sE$ \cite{ErY12}. That is, it is a positive harmonic function in $\bbC\setminus\sE$ which extends to a subharmonic function to $\bbC$ vanishing on $\sE$. Using the Hadamard representation for subharmonic functions \cite[Section 4.2]{HayKenSubh}, yields that there exists a measure $\nu$ supported on $\sE$ with 
			\begin{align}\label{eq:17}\nonumber
				\Im \theta(z)=a\Re z+b+\int_{|x|\leq 1}\log|x-z|d\nu(x)  \\  +\int_{|x|>1}\left(\log\left|1-\frac{z}{x}\right|+\frac{\Re z}{x}\right)d\nu(x).
			\end{align}
			By this integral representation,  $\Im\theta$ is concave in gaps and thus we can estimate
			\[
			\int_{\sa_k}^{\sb_k}\frac{d\sigma(x)}{x^2}\geq\frac{1}{\sb_k^2}\int_{\sa_k}^{\sb_k}d\sigma(x)\geq \frac{(\sb_k-\sa_k)h_k}{2\sb_k^2}=\frac{(\sb_0-\sa_0)\rho^kh_0\lambda^k}{2\sb_0^2\rho^{2k}}=C\frac{\lambda^k}{\rho^k},
			\]
			where $C:=\frac{(\sb_0-\sa_0)h_0}{2\sb_0^2}$. Noting that for the interval $(0,1)$ we have to take negative $k$, we get
			\[
			\int_0^1 \frac{d\sigma(x)}{x^2}\geq C\sum_{k=0}^\infty\left(\frac{\rho}{\lambda}\right)^k.
			\]
			That is, 
			\[
			\lim_{n\to\infty} \left(\frac{\lambda} {\rho}\right)^n\le  \frac{\Im \theta(i)}{a+C\sum_{k=0}^\infty\left(\frac{\rho}{\lambda}\right)^k}.
			\]
Since the right-hand side is finite, we conclude that $\lambda\leq\rho$. Hence, the sum diverges and the right-hand side is $0$. We conclude that the limit on the left-hand side is $0$, i.e., $\lambda<\rho$.
		\end{proof}
		
		Let us set 
		\begin{align}\label{eq:19}
			m(z)=i\theta'(z).
		\end{align}
		
		\begin{corollary}\label{cor:1}
			Let $m$ be defined by \eqref{eq:19}. Then 
			\[
			m(z)=z^\beta \omega(z),
			\]
			where $\omega$ obeys \eqref{rhoperiodic} and
			\[
			\beta=\frac{\log\lambda}{\log \rho}-1 < 0.
			\]
		\end{corollary}
		\begin{proof}
			The definition of $\beta$ and \eqref{eq:15} yield 
			\[
			\rho^\beta m(z)=m(\rho z).
			\]
			Thus, Lemma \ref{lem:6} applies. 
		\end{proof}

		It remains to show that $m\in\cN_0$. Recall the definition of $\sE$ in \eqref{eq:5}. Since $\theta$ extends continuously to $\bbR$, we can define a measure on $\sE$ by 
		\begin{align}\label{eq:harmonicMeasure}
		\mu_\sE((a,b]):=\frac{1}{\pi}(\Re \theta(b)-\Re\theta(a)).
		\end{align}
		\begin{proposition}
			It holds that
			\begin{equation}\label{31jul1}
				\int\frac{d\mu_{\sE}(x)}{1+|x|}<\infty
			\end{equation}
			and  
			\begin{align}\label{eq:16}
				m(z)=m_{\sE}(z)= \int_\sE \frac{1}{x-z}\, d\mu_\sE(x).
			\end{align}
			In particular $m \in\cN_0$.
		\end{proposition}
		\begin{proof}
			We start again from the Hadamard representation \eqref{eq:17} of the subharmonic function $\theta$.  Adding the harmonic conjugate for $z\in\bbC_+$, differentiating and passing from normalization at $0$ to normalization at $i$ yields 
			\begin{equation}\label{31jul2}
				i\theta'(z)= a +  \int_{\sE} \left(\frac{1}{x-z}-\frac{x}{1+x^2}\right)\, d\nu(x).
			\end{equation}
			Now Stieltjes inversion yields that $\nu=\mu_\sE$. 
			It remains to show that \eqref{31jul1} holds and $i\theta'(z)$ can in fact be represented as in \eqref{eq:16}. For this, note that $\Im i\theta'(iy)\leq |i\theta'(iy)|\leq Cy^\beta$, for $\beta<0$ by Corollary \ref{cor:1}. Now both claims follow from the classical connection between growth of Herglotz functions and tails of the measure, cf. e.g. \cite[Proposition 7.33]{LukicFirstCourse}.
		\end{proof}

\section{Expanding polynomials with real Julia sets} \label{secExpandingJuliaSet}

\subsection{Limit cycles for balanced measures on Julia sets}\label{sec:JuliaSet}
		As an application of our main theorem, we show that Theorem \ref{thm:3} can be applied to orthogonal polynomials  with respect to the balanced measure on the Julia set of an arbitrary expanding  polynomial $T$ with a real Julia set  $\sE_0$, and rescaled at an arbitrary periodic point $\xi$ of $T$. Let  $\sa = \min \sE_0$ and $\sb = \max \sE_0$. Such a polynomial has the following characteristic property: the critical points 
		$$
		\cC_T=\{c: T'(c)=0\}
		$$
		are real and simple, and the critical values $T(c)$ are such that
		$$
		T(c)\in\bbR\setminus[\sa,\sb].
		$$
		That is, all critical points belong to the basin of attraction of infinity, see \cite[Theorem 2.19, p. 584]{ErLyu},
		$$
		T^{\circ n}(c)\to\infty,\quad\forall c\in\cC_T,
		$$
		where $T^{\circ (n+1)}(z)=T(T^{\circ n}(z))$.

	In this subsection we prove Theorem \ref{intro:thm2}. At first, let us assume that $\xi$ is a fixed or periodic point. It suffices, by replacing $T$ by some iterate $T^{\circ m}$, to consider the case when
	\[
	T(\xi) = \xi, \qquad T'(\xi) > 0.
	\]
	It is known that \cite[Theorem 4.2]{DGV2008} then $T'(\xi) \ge d$. By conjugating with a translation, we assume without loss of generality that $\xi = 0$, i.e., $T(0)=0$. Let $T'(0) = \rho > 1$. 
		
		The claim of Theorem \ref{intro:thm2}  follows mainly from two classical functional equations in iteration theory.
		First, the following limit exists,  see e.g. \cite[Sect. 3]{Grabner15},
		\begin{equation}\label{eq1-1-25}
			F(z)=\lim_{n\to\infty}F_n(z),\quad F_n(z)=T^{\circ n}(z/\rho^n).
		\end{equation}
		The limit function $F(z)$ satisfies the Poincar\'e equation \eqref{eq1-2-25}. Note that $F_n(0)=0,\ F_n'(0)=1$, respectively $F$ is normalized by the following conditions $F(0)=0,\ F'(0)=1$.
		
Similar to the construction in Section \ref{sec:Comb1}, there exists a comb mapping $\theta_0$ such that the Julia set $\sE_0$ as defined in \eqref{JuliaSet} is given by $\theta_0^{-1}([0,\pi])$. We will present the comb explicitly for a quadratic polynomial; see Figure \ref{fig:2} below. In this section, we work with the Green's mapping $g_{\sE_0}(z)=e^{i\theta_0(z)}$; cf. \cite{Pom77}.
			Note that in the iteration theory this is the well known B\"ottcher function, see e.g.
			\cite[Sect. 4]{Grabner15}.
			Initially only defined on $\bbC_+$, $g_{\sE_0}$ can be extended to the complex Green function with respect to $\infty$ to the domain $\bar \bbC\setminus \sE_0$.  It is a multivalued (character automorphic) function in the domain. We use the normalization $\theta_0(0)=0$. Since $e^{i\theta_0(T(z))}$ on the universal covering is again a Blaschke product with zeros corresponding to $\infty$ of multiplicity $d$, we have the B\"ottcher equation
		\begin{equation}\label{eq4-1-25}
			e^{i\theta_0(T(z))}=e^{di\theta_0(z)}.
		\end{equation}
		
		By \eqref{eq1-1-25}, the following limit exists
		\begin{equation}\label{eq1-4-25}
			\theta(z)=\lim_{n\to\infty}\theta_0(T^{\circ n}(z/\rho^n))=\theta_0(F(z)).
		\end{equation}
		Moreover, by \eqref{eq4-1-25}
		\begin{equation}\label{eqThetaLimitThetaN}
			\theta(z)=\lim_{n\to\infty}d^n\theta_0(z/\rho^n)
		\end{equation}
		and therefore
		$$
		\theta(\rho z)=d \theta(z).
		$$
		As a consequence we get the following important proposition.

			\begin{proposition}\label{prop:1} 
				The limit $\theta$ in \eqref{eqThetaLimitThetaN} is a complex Martin function for the set $\sE=F^{-1}(\sE_0)$. If we denote
				\begin{align}\label{eq:2-25}
					m_{\sE_0}(z)=i\theta_0'(z)=\int\frac{d\mu_{\sE_0}(x)}{x-z}
				\end{align}
				and
				\begin{align}\label{eq:3-25}
					m_{\sE}(z)=i\theta'(z)=
					\int\frac{d\mu_{\sE}(x)}{x-z},
				\end{align}
				cf. \cite{EichLuk,Lev89part2},
				with the Poincar\'e  function $F$ the following identity holds
				\begin{equation}\label{eq7-1-10}
					m_{\sE_0}(F(z)) F'(z)=m_\sE(z)
				\end{equation}
				and
				\begin{equation}\label{mrlim79}
					m_\sE(z) = \lim_{n\to\infty} \rho^{n \beta} m_{\sE_0}\left( z/\rho^n \right), \qquad \beta = \frac{\log d}{\log \rho} - 1 < 0.
				\end{equation}
			\end{proposition}
			
			\begin{proof}
				By \eqref{eq1-4-25} and \eqref{eqThetaLimitThetaN}, $\theta$ is nonconstant, positive subharmonic on $\bbC$, and harmonic on $\bbC \setminus \sE$. It is injective on $\bbC_+$ and defines a conformal map to a comb domain.
				
				By differentiating \eqref{eq1-4-25} we conclude \eqref{eq7-1-10}. Moreover, differentiating \eqref{eqThetaLimitThetaN} we obtain
				\[
				m_\sE(z) = \lim_{n\to\infty} \left( \frac d \rho \right)^n m_{\sE_0}\left( z/\rho^n \right)
				\]
				which gives \eqref{mrlim79}.
			\end{proof}

\begin{proof}[Proof of Theorem~\ref{intro:thm2}]
For a periodic point $\xi = 0$, from \eqref{mrlim79} it follows that $m_\sE(z\rho)=\rho^\beta m_\sE(z)$ and thus by Lemma \ref{lem:6} we have that 
				\begin{equation}\label{eq:26}
				m_\sE(z)=z^\beta\omega(z),\quad \omega(z\rho)=\omega(z).
				\end{equation}
				Hence, by Lemma~\ref{lem:3} the assumptions of Theorem~\ref{thm:3} are satisfied which implies \eqref{eq:18}.

If $\xi$ is eventually periodic, as before, by using a power of $T$ and an affine map, we pass to the case where $0$ is a fixed point, $T'(0) = \rho > 1$,
\[
T(\xi) =0, \quad \xi \neq 0.
\]
The equilibrium measure $\mu_{\sE_0}$ is an eigenvector for the adjoint $\cL_0^*$ of the Ruelle operator
\[
(\cL_0 f)(x) = \sum_{T(y)=x} f(y).
\]
\[
\int \frac 1{x-z} d\mu_{\sE_0}(x) = \frac 1d \int \frac{T'(z)}{x-T(z)} d\mu_{\sE_0}(x) 
\]
That is,
\[
m_{\sE_0} (z) = \frac {T'(z)}d m_{\sE_0}(T(z)), \quad z \in \bbC \setminus\sE_0.
\]
We denote
\[
f_n(z) = \left( \frac d\rho \right)^n m_{\sE_0} ( z/ \rho^n)
\]
and recall that $f_n \to m_\sE$ uniformly on compacts in $\bbC_\pm$. Moreover, differentiability implies that
\[
\rho^n T(\xi + z/\rho^n) \to T'(\xi) z, \qquad n \to\infty
\]
uniformly on compacts in $\bbC_+$. Composing these two statements, we conclude
\[
f_n(\rho^n T(\xi + z/\rho^n)) \to  m_\sE( T'(\xi) z )
\]
that is,
\[
\left( \frac d\rho \right)^n m_{\sE_0}(\xi + z/\rho^n) \to \frac{T'(\xi)}d m_{\sE}(T'(\xi)  z).
\]
Now the proof follows as before.
\end{proof}

		\section{Special case: orthogonal polynomials associated with the balanced measure  on the Julia set  of a quadratic polynomial }\label{sectionChristoffel}
In this section we consider the special case
		\begin{equation*}
			T(z):=z(z+\rho),\quad -\rho>2.
		\end{equation*}
		We point out that the multiplicator $\rho$ at the fixed point $\xi=0$ is negative. Let $\xi_+=1-\rho$ be the non-zero fixed point if $T$ and $\xi_-=-1$ be chosen so that $T^{-1}(\xi_+)=\{\xi_\pm\}$.  
		
The Julia set $\sE_0$ and the harmonic measure can be described in terms of an explicit comb mapping. For $h_0>0$ we construct the comb domain
		\[
		\Pi=\{x+iy\mid 0<x<\pi, y>0\}\setminus\bigcup_{k\geq 1}\bigcup_{\ell=1}^{2^{k-1}}\left\{\frac{2\ell-1}{2^k}\pi+iy\mid 0<y\leq\frac{h_0}{2^{k-1}}\right\}.
		\]
		Since $h_0/2^k\to 0$, the boundary of $\Pi$ is locally connected, conformal maps from $\bbC_+$ to $\Pi$ have continuous extensions to the closure of $\bbC_+$. Moreover, $0, \pi, \infty$ are not cut points, so there is a unique conformal map $\theta_0$ with the normalization $\theta_0(\xi_-)=0, \theta_0(\xi_+)=\pi, \theta_0(\infty)=\infty$. 
		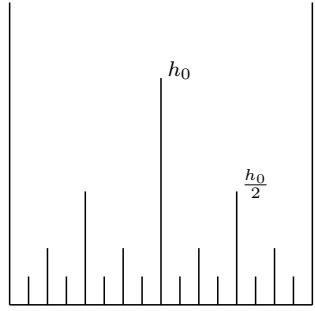
\begin{figure}[ht]
				\begin{tikzpicture}
				\draw[line width=0.2mm,-] (0,0) -- (4,0) ;
				\draw[line width=0.2mm,-] (0,0) -- (0,4) ;
				\draw[line width=0.2mm,-] (4,0) -- (4,4) ;

				\draw[line width=0.2mm,-] (2,0) -- (2,3) ;
				\draw[line width=0.2mm,-] (1,0) -- (1,1.5) ;
				\draw[line width=0.2mm,-] (3,0) -- (3,1.5) ;
				
				\draw[line width=0.2mm,-] (0.5,0) -- (0.5,0.75) ;
				\draw[line width=0.2mm,-] (1.5,0) -- (1.5,0.75) ;
				\draw[line width=0.2mm,-] (2.5,0) -- (2.5,0.75) ;
				\draw[line width=0.2mm,-] (3.5,0) -- (3.5,0.75) ;
				
				\draw[line width=0.2mm,-] (0.25,0) -- (0.25,0.375) ;
				\draw[line width=0.2mm,-] (0.75,0) -- (0.75,0.375) ;
				\draw[line width=0.2mm,-] (1.25,0) -- (1.25,0.375) ;
				\draw[line width=0.2mm,-] (1.75,0) -- (1.75,0.375) ;
				\draw[line width=0.2mm,-] (2.25,0) -- (2.25,0.375) ;
				\draw[line width=0.2mm,-] (2.75,0) -- (2.75,0.375) ;
				\draw[line width=0.2mm,-] (3.25,0) -- (3.25,0.375) ;
				\draw[line width=0.2mm,-] (3.75,0) -- (3.75,0.375) ;

				\node at (2.25, 3.1)   {\scriptsize$h_0$};
				\node at (3.25,1.6)   {\scriptsize $\frac{h_0}{2}$};		
			\end{tikzpicture}
				\caption{The comb-domain $\Pi$ related to $\sE_0$}
			\label{fig:2}
		\end{figure}  
		There exists a unique choice of $h_0$, so that $\theta_0^{-1}([0,\pi])=\sE_0$, see \cite{NPVY}. In this case the harmonic measure $\mu_{\sE_0}$  of the Julia set $\sE_0$ is given by the Lebesgue measure on the base of the comb as in \eqref{eq:harmonicMeasure}.
		
		For linear fractional transforms we will use a brief notation
		$$
		A\star w=\frac{a_{11} w+a_{12}}{a_{21}w+a_{22}},\quad A=\begin{pmatrix}
		a_{11}&a_{12}\\a_{21}&a_{22}
		\end{pmatrix},\ \det A\not=0.
		$$
In this section we derive asymptotics for $\cK_{11}(n,0,0)$ and $\cK_{22}(n,0,0)$. Moreover, we prove Theorem \ref{intro:thm3}, that is we show that the limit chain can be parametrized by the M-type. Since e.g. by the discussion \cite[Section 33]{Sodin2000} $\cK(z,w)$ can be recovered from the diagonal, we work in this section with the diagonal
	$$
	\cK(z):=\cK(z,z)=\frac{\fA(z)J\fA(z)^*-J}{z-\bar z}
	$$	
	 of the reproducing kernels $\cK(z,w)$.
	In what follows
	\[
\mathscr C(m_{\sE_0})=\{\fA_n(\cdot)\mid n\in\bbZ_+\}
\]
is the standard discrete chain of transfer matrices generated by orthogonal polynomials associated to the balanced measure $\mu_{\sE_0}$, see \eqref{intro:eq8}. Moreover let
\[
\mathscr C(m_\sE)=\{\fD(t,\cdot)\mid t\in [0,\infty)\}
\]
be the chain that meets de Branges normalizations
\begin{equation}\label{eq:dbr}
\fD(0,z)=I,\quad\tr\left(\partial_z\partial_t\fD(t,z)J|_{z=0}\right)=1.
\end{equation}

Based on the recurrence relations \eqref{intro:eq1}, we introduce the half-line Jacobi matrix $\cJ$ corresponding to $\mu$, 
		\begin{align*}
		\cJ e_n=&a_{n}e_{n-1}+b_n e_n+a_{n+1}e_{n+1},\quad n\ge 1,\\
		\cJ e_0=&b_0e_0+a_1 e_1,
		\end{align*}
		where $(e_n)$ is the standard basis in $\ell^2(\bbZ_+)$.  In our case, by symmetry, $b_k=-\rho/2$ for all $k$. The following formulas are standard and can for instance be found in \cite[Section 10]{LukicFirstCourse}. The resolvent (Weyl) function of $\cJ$ coincides with the Herglotz function $m_{\sE_0}$,
		$$
		m_{\cJ}(z)=\langle(\cJ-z)^{-1}e_0,e_0\rangle=\int\frac{d\mu_{\sE_0}}{x-z}=m_{\sE_0}(z).
		$$
		We will use block representations
		\begin{equation}\label{eq:27may-1}
		\cJ=\begin{pmatrix}
		\cJ_k^- & 0\\
		0&\cJ_k^+
		\end{pmatrix}+a_k ( e_{k-1}\langle \cdot, e_k\rangle+e_{k}\langle \cdot, e_{k-1}\rangle).
		\end{equation}
		Recall that the diagonal entry $R_k(z)$ of the resolvent  can be represented in terms of one-sided resolvent functions $m^\pm_{\cJ,k}$
		\begin{equation}\label{eq:27may-2}
-\frac 1{R_{k}(z)}=-\frac 1{m_{\cJ,k}^+(z)}+m_{\cJ,k}^-(z),
\end{equation}
where
$$
R_k(z)=\langle(\cJ-z)^{-1}e_k,e_k\rangle
$$
and $m^\pm_{\cJ,k}$ are defined by
$$
m_{\cJ,k}^-(z)=a_k^2\langle(\cJ_k^--z)^{-1}e_{k-1},e_{k-1}\rangle,
\ \ 
m_{\cJ,k}^+(z)=\langle(\cJ_k^+-z)^{-1}e_{k},e_{k}\rangle
$$
and possess the following representations
\begin{equation}\label{eq:27may-3}
		m^-_{\cJ,k}(z)=-\frac{a_kp_{k-1}(z)}{p_k(z)},
		\quad
		m_{\cJ,k}^+(z)=	\fA_k(z)^{-1}\star m_{\cJ}(z).
		\end{equation}
		
		\subsection{An explicit expression for the limit reproducing kernels  on a dense set}
		In this subsection,
		using renormalization relations for orthogonal polynomials with respect to $\mu_{\sE_0}$, we derive  explicit formulas for the limit reproducing kernels
		corresponding to the $M$-type of the form $k/2^s$.
				 
		The following relations are well known, we provide a proof for the reader's convenience.		
		\begin{lemma}\label{lem:orthoPoly}
			Let $p_k,q_k$ denote the orthonormal polynomials of first and second kind associated to $\mu_{\sE_0}$. Then for $n,k\geq 0$ it holds that
			\begin{align}\label{eq:3}
				p_{2^nk}(z)&=p_k(T^{\circ n}(z)),\\
				q_{2^nk}(z)&=\frac{1}{2^n}q_k(T^{\circ n}(z))(T^{\circ n})'(z).\label{eq:4}
			\end{align}
		\end{lemma}

		\begin{proof}
			The Ruelle operator $\cL_0:C(\sE_0)\to C(\sE_0)$ is defined by 
			\[
			(\cL_0f)(x)=\sum_{T(y)=x}f(y).
			\]
			The proof is based on the property that $\mu_{\sE_0}$ is an eigenmeasure of its adjoint $\cL_0^*$ to the eigenvalue $2$. That is, for every $f\in C(\sE_0)$,
			\begin{equation}\label{eq:1}
				\int (\cL_0f)(x)d\mu_{\sE_0}(x)=2\int f(x)d\mu_{\sE_0}(x).
			\end{equation}
			We will  use the following identity for $f(x)=\frac{1}{z-x}$ and $z\in\bbC$
			\begin{equation}\label{eq:2}
				(\cL_0 f)(x)=\frac{T'(z)}{T(z)-x}.
			\end{equation}

			First we show \eqref{eq:3}, \eqref{eq:4} for $k\geq 0$ and $n=1$. \eqref{eq:3} for $n=1$ follows from \eqref{eq:1}, see \cite{BarnGerHarr82}.
			 Using the first identity, \eqref{eq:1} and \eqref{eq:2}, we compute
			\begin{align*}
				q_{2k}(z)&=\int\frac{p_{k}(T(z))-p_k(T(x))}{z-x}d\mu_{\sE_0}(x)\\
				&=\frac{1}{2}\int \cL_0\bigg(\frac{p_{k}(T(z))-p_k(T(y))}{z-y}\bigg)(x)d\mu_{\sE_0}(x)\\
				&=\frac{1}{2}\int(p_{k}(T(z))-p_k(x))\cL_0\bigg(\frac{1}{z-y}\bigg)(x)d\mu_{\sE_0}(x)\\
				&=\frac{T'(z)}{2}\int\frac{p_{k}(T(z))-p_k(x)}{T(z)-x}d\mu_{\sE_0}(x)
				=\frac{T'(z)}{2}q_k(T(z))
			\end{align*}
			That is \eqref{eq:4} holds for $n=1$. Induction over $n$ shows that 
			\[
			p_{2^n k}(z)=p_{2 2^{n-1}k}(z)=p_{2^{n-1}k}(T(z))=p_k(T^{\circ(n-1)}(T(z)))=p_k(T^{\circ n}(z)).
			\]
			Likewise
			\begin{align*}
				q_{2^n k}(z)&=q_{2 2^{n-1}k}(z)=\frac{1}{2}q_{2^{n-1}k}(T(z))T'(z)\\
				&=\frac{1}{2}\frac{1}{2^{n-1}}q_k(T^{\circ (n-1)}(T(z)))
				(T^{\circ (n-1)})'(z)T'(z)\\
				&=\frac{1}{2^n}q_k(T^{\circ n}(z))(T^{\circ n})'(z).
			\end{align*}
		\end{proof}
		Recall that
		$
		F_n(z)=T^{\circ n}(z/\rho^n).
		$
		As an immediate consequence of Lemma \ref{lem:orthoPoly} we get
		\begin{lemma}
			For $n,k\geq 0$ it holds that
			\begin{align}\label{eq2-5-11}
				p_{k2^n}(z/\rho^n)&=p_k(F_n (z))\\ \label{eq2-6-11}
				(-1)^n q_{k2^n}(z/\rho^n)&=\left(\frac{|\rho|}{2}\right)^nq_k(F_n(z))F_n'(z).
			\end{align}
		\end{lemma}
		\begin{proof}
			Observing that 
			\[
			F_n'(z)=\frac{1}{\rho^n}(T^{\circ n})'(z/\rho^n),
			\]
			both identities follow by substituting $z/\rho^n$ in \eqref{eq:3} and \eqref{eq:4}.
		\end{proof}

		Based on \eqref{eq2-5-11},   \eqref{eq2-6-11} we extend renormalization relations on  transfer matrices $\fA_n(z)$.
		\begin{lemma}\label{l2-2-11}
			Let
			\begin{align}\label{eq:53}
				\cU_n=\begin{pmatrix}
				(-1)^n(|\rho|/2)^{-n/2}&0\\
				0&(|\rho|/2)^{n/2}
			\end{pmatrix}\quad\text{and}\quad
			\cF_n=\begin{pmatrix}
				1 & 0 \\
				0 & {F'_n(z)}
			\end{pmatrix}
		\end{align}
			Then the family $\{\fA_n\}$ obeys the following renormalization
			\begin{equation}\label{23-08-05-03'S}
				\cU_n
				\fA_{k2^n}(z/\rho^n)
				\cU_n^{-1}
				=
				\cF_n(z)^{-1}\fA_k(F_n(z))\cF_n(z)
				\begin{pmatrix}
					1 & 0 \\  U_{k,n}(z) & 1
				\end{pmatrix}
			\end{equation}
			where
			\begin{align}\label{eq:110}
				U_{k,n}(z)=
				-\sum_{c\in \cC_n}\frac{m^-_{\cJ,k}(F_n(c))}{(c-z)F_n''(c)}
			\end{align}
			and
			$\cC_n$ is the set of critical points of $F_n(z)$
			$$
			\cC_n=\{c: F'_n(c)=0\}.
			$$
		\end{lemma}
		
		\begin{remark}\label{rem:1}
			Note that the renormalized family $\{\cU_n
			\fA_{k2^n}(z/\rho^n)
			\cU_n^{-1}\}_{n\in\bbZ_0}$ consists also of $J$-expanding  matrix-functions for both even and odd indices.
		\end{remark}

		
		\begin{proof}
			Since
			$$
			\cU_n
			\fA_{k2^n}(z/\rho^n)
			\cU_n^{-1}=
			\begin{pmatrix}
				-a_{_{k2^n}} q_{{k2^n}-1}& -(\rho/2)^{-n}q_{k2^n}\\
				(\rho/2)^n a_{_{k2^n}} p_{{k2^n}-1}& p_{k2^n}
			\end{pmatrix}(z/\rho^n)
			$$
			in view of \eqref{eq2-5-11} and \eqref{eq2-6-11}
			\begin{equation*}\label{23-08-05-01S}
				=
				\begin{pmatrix}
					-a_{_{k2^n}} q_{{k2^n}-1}(z/\rho^n)& -q_{k}(F_n(z))F'_n(z)\\
					(\rho/2)^n a_{_{k2^n}} p_{{k2^n}-1}(z/\rho^n)& p_{k}(F_n(z))
				\end{pmatrix}.
			\end{equation*}
			We consider the following ratio
			\begin{equation*}\label{23-08-05-04S}
				(\rho/2)^n m_{\cJ,k 2^n}^-(z/\rho^n)
				=-\dfrac{(\rho/2)^n a_{_{k2^n}} p_{{k2^n}-1}(z/\rho^n)}{p_{k}(F_n(z))}.
			\end{equation*}
			Let $x_j$ be zeros of $p_k(z)$, $1\le j\le k$. Then zeros of the denominator are solutions of
			$$
			F_n(y)=x_j.
			$$
			Since at the same time they are zeros of $p_{k2^n}(z/\rho^n)$, they are simple. Therefore,
			$$
			(\rho/2)^n m_{\cJ,k 2^n}^-(z/\rho^n)=
			\sum_{j=1}^k\sum_{F_n(y)=x_j}\frac{(\rho/2)^n a_{_{k2^n}} p_{{k2^n}-1}(y/\rho^n)}{p'_k(F_n(y)) F_n'(y)}
			\dfrac{1}{y-z}
			$$
			Since $\det \fA_{k 2^n}(z)=1$ and $p_{k}(F_n(z))$ vanishes at $y$'s, we get
			$$
			(\rho/2)^n a_{_{k2^n}} p_{{k2^n}-1}(y/\rho^n)\cdot q_{k}(F_n(y))F'_n(y)=1.
			$$
			Hence,
			$$
			(\rho/2)^n m_{\cJ,k 2^n}^-(z/\rho^n)
			=\sum_{j=1}^k\sum_{F_n(y)=x_j}\frac{1}{q_k(F_n(y))p'_k(F_n(y)) F_n'(y)^2}
			\dfrac{1}{y-z}
			$$
			$$
			=
			\sum_{j=1}^k\frac{1}{p_k'(x_j)q_k(x_j)}\sum_{F_n(y)=x_j}\frac{1}{F_n'(y)^2(y-z)}.
			$$
			Since $\det\fA_k(z)=1$, we have
			$$
			a_k p_{k-1}(x_j)\cdot q_k(x_j)=1.
			$$
			Therefore,
			\begin{equation}\label{eq2-8-11}
				(\rho/2)^n m_{\cJ,k 2^n}^-(z/\rho^n)
				=\sum_{j=1}^k\frac{a_{k} p_{k-1}(x_j)}{p_k'(x_j)}\sum_{F_n(y)=x_j}\frac{1}{F'_n(y)^2(y-z)}.
			\end{equation}
			Now we note that
			\begin{equation}\label{eq2-9-11}
				\frac{1}{F_n'(z)(F_n(z)-x_j)}=\sum_{F(y)=x_j}\frac{1}{F_n'(y)^2(z-y)}+\sum_{c\in \cC_n}\frac{1}{F_n''(c)(F_n(c)-x_j)(z-c)}.
			\end{equation}
			Using \eqref{eq2-9-11} we can further transform \eqref{eq2-8-11} as follows
			$$
			(\rho/2)^n m_{\cJ,k 2^n}^-(z/\rho^n)
			=-\sum_{j=1}^k\frac{a_{k} p_{k-1}(x_j)}{p_k'(x_j)}\frac{1}{F_n'(z)(F_n(z)-x_j)}
			$$
			$$
			+
			\sum_{j=1}^k\frac{a_{k} p_{k-1}(x_j)}{p_k'(x_j)}\sum_{c\in \cC_n}\frac{1}{F_n''(c)(F_n(c)-x_j)(z-c)}
			$$
			$$
			=\frac{m^-_{\cJ,k}(F_n(z))}{F'_n(z)}
			+
			\sum_{c\in \cC_n}
			\frac{1}{F_n''(c)(z-c)}
			\sum_{j=1}^k\frac{a_{k} p_{k-1}(x_j)}{p_k'(x_j)}\frac{1}{F_n(c)-x_j}
			$$
			$$
			=\frac{m^-_{\cJ,k}(F_n(z))}{F'_n(z)}
			-
			\sum_{c\in \cC_n}
			\frac{m_{\cJ,k}^-(F_n(c))}{F_n''(c)(z-c)}
			=\frac{m^-_{\cJ,k}(F_n(z))}{F'_n(z)}
			-
			U_{k,n}(z).
			$$
			Hence
			$$
			(\rho/2)^n a_{_{k2^n}} p_{{k2^n}-1}(z/\rho^n)
			=\frac{a_kp_{k-1}(F_n(z))}{F'_n(z)}
			+
			p_k(F_n(z))U_{k,n}(z).
			$$
			Thus we proved that three entries of the matrices on the LHS and RHS of
			\eqref{23-08-05-03'S} coincide. Moreover they are not vanishing identically. Since the determinants of both matrices are one, 
			the identity
			\eqref{23-08-05-03'S} is proved.
		\end{proof}

		In order to show that the limit of $U_{k,n}$ as $n\to\infty$ exists, we need a preliminary lemma. 
		
		\begin{lemma}\label{lem:7}The following holds:
			\begin{enumerate}[(i)]
				\item ${\displaystyle \forall n\in\bbN, \forall c\in\cC_n: F_n(c)\ge \rho^2/4-\rho>1-\rho\lor F_n(c)\le -\rho^2/4<-1}$;
				\item ${\displaystyle \exists M>0, \forall n\in\bbN, \forall c\in\cC_n: 0<-\frac{m_{\cJ,k}^-(F_n(c))}{F_n(c)}<M}$;
				\item ${\displaystyle \exists M>0, c\in\cC: 0<-\frac{m_{\cJ,k}^-(F(c))}{F(c)}<M}$;
				\item ${\displaystyle -U_{k,n}(z)\in\cN_0}$;
			\end{enumerate}
		\end{lemma}
		\begin{proof}
			(i):	First we show that for all $n\in\bbN$ and all $c\in\cC_n$
			\begin{align}\label{eq:108}
			F_n(c)\in \{T^{\circ k}(-\rho/2)\mid k\in\bbN\}.
			\end{align}
			Since $F_n$ and $T^{\circ n}$ only differ by rescaling of the argument, the set of critical values coincide. Let  $\tilde \cC_n$ denote the set of critical points of $T^{\circ n}$ and set $\tilde \cC_0=\emptyset$. We claim that for $n\geq 1$, $\tilde \cC_{n+1}=\tilde\cC_1\cup T^{-1}(\tilde\cC_n)$. For $n=0$ this is clear. For $n\geq 1$ we have
			\[
			T^{\circ(n+1)}(z)'=T^{\circ n}(T(z))'={T^{\circ n}}'(T(z))T'(z).
			\]
			Hence, $T^{\circ(n+1)}(c)'=0$ if and only if $c\in \tilde \cC_1$ or $T(c)\in \tilde\cC_{n}$, which shows $\tilde \cC_{n+1}=\tilde\cC_1\cup T^{-1}(\tilde\cC_n)$. Now we can show \eqref{eq:108} by induction.  Since $\tilde\cC_1=\{-\rho/2\}$, the statement is true for $n=1$. Now take $c\in\cC_{n+1}$. If $c=-\rho/2$, $T^{\circ(n+1)}(c)$ clearly belongs to the set in \eqref{eq:108}. If $c\in T^{-1}(\tilde\cC_{n})$, then $T(c)\in \tilde\cC_{n}$ and thus, by induction hypothesis, $T^{\circ n}(T(c))=T^{\circ(n+1)}(c)$ belongs also to this set. This finishes the proof of \eqref{eq:108}. It follows from \eqref{eq:108} again by induction that 
			$$
			|F_n(c)+\rho/2|\ge \rho^2/4-\rho/2,\quad \forall c\in\cC_n
			$$ 
			which clearly implies (i).
			
			(ii): Since $m_{\cJ,k}^-$ has only poles in $[-1,1-\rho]$, $\lim\limits_{x\to\pm\infty}m_{\cJ,k}^-(x)=0$ and is monotonically increasing and positive on $(-\infty,-1)$ and monotonically increasing and negative on $(1-\rho,\infty)$, by (i) we find $M>0$ so  that uniformly in $n$ and $c\in\cC_n$
			\[
			\frac{1}{\sqrt{M}}<|F_n(c)|,\quad |m_{\cJ,k}^-(F_n(c))|<\sqrt{M}
			\]
			and hence 
			\[
			0<-\frac{m_{\cJ,k}^-(F_n(c))}{F_n(c)}<M.
			\]
			
			(iii): Since $F_n\to F$, by Hurwitz theorem and local uniform convergence, we have that for every $c\in\cC$ there is a unique $c^{(n)}\in\cC_n$ such that $F_n(c^{(n)})\to F(c)$. Since $F_n(c^{(n)})=T^{\circ n_k}(-\rho/2)$ for some $n_k\in\bbN$ and the set in \eqref{eq:108} is discrete we conclude that $F(c)=T^{\circ k}(-\rho/2)$ for some $k\in\bbN$. Hence, the estimates can be proved as in (ii).
			
			(iv): By \eqref{eq:110} we need to show that for all $c\in\cC_n$, 
			\[
			\frac{m_{\cJ,k}^-(F_n(c))}{F''_n(c)}>0.
			\]
			If $F_n(c)>0$, then by (i) $F_n(c)>1-\rho$ and hence by the properties of $m_{\cJ,k}^-$, $m_{\cJ,k}^-(F_n(c))<0$. On the other hand, in this case we have $F''_n(c)<0$ and the claim follows. The proof for $F_n(c)<0$ is analogously. 
		\end{proof}

		In order to find the limit of $U_{k,n}$, it will be convenient to use an auxiliary Herglotz function. 
		Let us define the rational Herglotz function
		\[
		w_n(z)=\frac{F_n(z)}{F_n'(z)}.
		\]
		Define
		\[
		\sigma_n=\sum_{c\in \cC_n}\sigma^n_{c}\delta_{c},\quad \sigma_{c}^n=-\frac{F_n(c)}{F_n''(c)}>0.
		\]
		Note that positivity of $\sigma_{c}^n$ follows as in the proof of Lemma \ref{lem:7}(iv), which together with the behavior at infinity given below implies that the rational function $w_n$ is indeed a Herglotz function.
		If $\cC_n=\{c\in\bbC\mid F'_n(c)=0\}$ denotes the set of critical points of $F_n$, then  since $w_n(0)=0$ we get 	
		\[
		w_n(z)=\frac{z}{2^n}+\sum_{c\in \cC_n}\left(\frac{1}{c-z}-\frac{1}{c}\right)\sigma_{c}^n,\qquad \sum_{c\in \cC_n}\frac{\sigma^n_c}{1+c^2}<\infty.
		\]
		Likewise, let us define the meromorphic Herglotz function 
		\[
		w(z)=\frac{F(z)}{F'(z)}
		\]
		and
		\[
		\sigma=\sum_{c\in \cC}\sigma_{c}\delta_{c},\quad \sigma_{c}=-\frac{F(c)}{F''(c)}>0,
		\]
		where $\cC$ denotes the set of critical points of $F$. 
		Since by \eqref{eq7-1-10} and \eqref{eq:26} $F(z)/F'(z)\sim z^{-\beta}$,
		\[
		\lim_{y\to\infty} \frac{\Im w(iy)}{y}=0,
		\]
		and hence there is no point mass at infinity in the integral representation of $w$. Using again that $w(0)=0$ we obtain
		\begin{align}\label{eq:105}
		w(z)=\sum_{c\in\cC}\left(\frac{1}{c-z}-\frac{1}{c}\right)\sigma_c, \quad \sum_{c\in\cC}\frac{\sigma_c}{1+c^2}<\infty. 
		\end{align}

		\begin{lemma}\label{l:22dec1}
			Uniformly on compact subsets of $\bbC_+$, the following limit exists
			\begin{equation}\label{14-nov-1}
				\hat U_k(z):=\lim_{n\to\infty}(U_{k,n}(z)-U_{k,n}(0))=
				-\sum_{c\in \cC}\frac{m_{\cJ,k}^{-}(F(c))}{F''(c)}\left(\frac 1{c-z}-\frac 1 c\right)
			\end{equation}
		and $-\hat U_k\in\cN_0$.
		\end{lemma}
		\begin{proof}
		
		Recall that 
		\[
		U_{k,n}(z)-U_{k,n}(0)=-\sum_{c\in \cC_n}\left(\frac{1}{c-z}-\frac{1}{c}\right)\frac{m^-_{\cJ,k}(F_n(c))}{F_n''(c)}=\sum_{c\in\cC_n}\frac{z(1+c^2)}{(c-z)c}\nu^n_c,
		\]
		where
		\[
		\nu^n_c=-\frac{m^-_{\cJ,k}(F_n(c))}{F_n''(c)(1+c^2)}
		\]
		Likewise we write
		\[
		-\sum_{c\in \cC}\frac{m_{\cJ,k}^{-}(F(c))}{F''(c)}\left(\frac 1{c-z}-\frac 1 c\right)=\sum_{c\in\cC_n}\frac{z(1+c^2)}{(c-z)c}\nu_c,\quad \nu_c=-\frac{m^-_{\cJ,k}(F(c))}{F''(c)(1+c^2)}.
		\]
		Recall that by Lemma \ref{lem:7} $\nu_c,\nu_{c_n}<0$. If we denote $\nu^n=-\sum_{c\in\cC_n}\nu^n_c\delta_c$ and $\nu=-\sum_{c\in\cC}\nu_c\delta_c$, the claim follows from $\nu^n\to\nu$ in the weak$^*$ topology of $C(\overline{\bbR})^*$.

		 Since $F_n$ converges, uniformly on compact subsets of $\bbC_+$, $w_n\to w$. If $d\sigma_n=\sum_{c\in\cC_n}\sigma^n_c\delta_{c}$ and likewise $d\sigma=\sum_{c\in\cC}\sigma_c\delta_{c}$ the general theory of Herglotz functions, see e.g. \cite[Proposition 7.28]{LukicFirstCourse}, implies
		\begin{align}\label{eq:107}
			\delta_\infty2^{-n}+\frac{d\sigma_n(c)}{1+c^2}\to \frac{d\sigma(c)}{1+c^2}
		\end{align}
		 in the weak$^*$ topology of $C(\overline{\bbR})^*$. By \eqref{eq:105}, for every $\e>0$, there exists $N>0$ $-N,N\notin\cC$ and 
		 \begin{align}\label{eq:111}
		 \sum_{c\in\cC\setminus[-N,N]}\frac{\sigma_c}{1+c^2}<\e.
		 \end{align}
		 From \eqref{eq:107} and the fact that $2^{-n}\to 0$, we obtain that there is $n_0$ such that for $n\geq n_0$, 
		 \begin{align}\label{eq:112}
		 \sum_{c\in\cC_n\setminus[-N,N]}\frac{\sigma^n_c}{1+c^2}<2\e.
		 \end{align}
	
		Note that 
		\[
		\nu^n_c=\frac{m^-_{\cJ,k}(F_n(c))}{F_n(c)}\frac{1}{1+c^2}\sigma^n_c,\quad \nu_c=\frac{m^-_{\cJ,k}(F(c))}{F(c)}\frac{1}{1+c^2}\sigma_c.
		\]
		By Lemma \ref{lem:7} all multipliers are uniformly bounded. Hence, by \eqref{eq:111}, \eqref{eq:112} for every $f\in C(\overline{\bbR})$, we find $N$ and $n_0$ such that for $n\geq n_0$
		\[
		\int_{\overline{\bbR}\setminus[-N,N]}|f(x)|d\nu^n(x)+\int_{\overline{\bbR}\setminus[-N,N]}|f(x)|d\nu(x)<\epsilon. 
		\]
		
		On the other hand on $[-N,N]$ there are only finitely many critical points of $F$ and again by Hurwitz's theorem for every $c\in\cC\cap[-N,N]$ there exists an eventually unique sequence $c^{(n)}\in\cC_n$ such that $c^{(n)}\to c$. Since $F_n\to F$ locally uniformly we conclude that the multipliers of $\nu^n$ also converge and we find that $\nu^n|_{[-N,N]}\to \nu|_{[-N,N]}$ in the weak$^*$ topology of $C([-N,N])^*$. This finishes the proof of  \eqref{14-nov-1}. 
		
		Since $-U_{k,n}\in\cN_0$ by Lemma \ref{lem:7}, it follows that $-\hat U_k\in\cN_0$. 
			\end{proof}

		Lemma \ref{l:22dec1} allows us to send $n\to\infty$ in  \eqref{23-08-05-03'S}. From this, we obtain asymptotics of the diagonal reproducing kernels associated with orthogonal polynomials
		$$
		\cK_n(z)=\frac{\fA_n(z)J\fA_n(z)^*-J}{z-\bar z},
		$$
		where $\fA_n(z)$ is as in \eqref{eq:50}.

		\begin{theorem}\label{thnov22-1}
			Let
			$$
			 \fB_k(z)=\cF(z)^{-1}\fA_k(F(z))\cF(z)
			 \begin{pmatrix}
				1 & 0 \\  \hat U_{k}(z) & 1
			\end{pmatrix}, \quad \cF=\begin{pmatrix}
				1 & 0 \\
				0 & {F'(z)}
			\end{pmatrix},
			$$
			Then
			\begin{align}\label{eq:51}
					\lim_{k\to\infty}\fB_k\star 0={F'(z)}{m_{\cJ}(F(z))}=m_\sE(z).
			\end{align}
			We define the associated reproducing kernels
			\begin{align}
				\hat \cK_k(z)=\frac{\fB_k(z)J \fB_k(z)^*-J}{z-\bar z}.
			\end{align}
			Then 
			\begin{equation}\label{eq16_nov_1}
				\lim_{n\to\infty}\frac 1{|\rho|^n}\cU_n\cK_{k2^n}(z/\rho^n)\cU_n^*=\hat \cK_k(z).
			\end{equation}
		\end{theorem}
		\begin{proof}
			Since 
			\[
			\lim\limits_{k\to\infty}\fA_k(z)\star 0=m_{\cJ}(z),
			\]
			we get the first equality in \eqref{eq:51}. The second equality follows from \eqref{eq7-1-10}. Finally, $F_n\to F$ and Lemma \ref{l:22dec1}, allow us to take the limit in \eqref{23-08-05-03'S} to obtain \eqref{eq16_nov_1}. 
		\end{proof}
		For a fixed  $s\in\bbN_0$ we have 
		\begin{equation}\label{eq:27may-6}
		\frac 1{|\rho|^s}\cU_s\hat \cK_k(z/\rho^s)\cU_s^*=\lim_{n+s\to\infty}
		\frac 1{|\rho|^{n+s}}\cU_{n+s}\cK_{(k2^{-s})2^{n+s}}(z/\rho^{n+s})\cU_{n+s}^*.
	\end{equation}
		This justifies to formally extend the definition of $\hat\cK_k$ by
		\begin{align}\label{eq:54}
		\hat\cK_{k/2^s}(z):=\frac 1{|\rho|^s}\cU_s\hat \cK_k(z/\rho^s)\cU_s^*.
	\end{align}
		It follows from \eqref{eq:27may-6} that $k_1/2^{s_1}=k_2/2^{s_2}$ imply that $\hat\cK_{k_1/2^{s_1}}(z)=\hat\cK_{k_2/2^{s_2}}(z)$ and thus for such values the kernel is well defined. In particular, this shows that for $k,s\in\bbN_0$, $\hat\cK_{k2^s/2^s}(z)=\hat\cK_{k}(z)$. Moreover,  by \eqref{eq:54} this extension is $J$-monotonic. Namely, we have: 
		$$
		\hat\cK_{k_1/2^s}(z)\le \hat\cK_{k_2/2^s}(z), \quad k_1\le k_2.
		$$
		In the following we will use the trace parametrization of the chain $\mathscr C(m_\sE)$ introduced in Remark \ref{rem:intro}
				\begin{corollary}\label{cor:2} Let
		\[
\mathscr C(m_\sE)=\{\fD(t,\cdot)\mid t\in [0,\infty)\}
\]
be the unique de Branges' chain generated by the Herglotz function $m_\sE$ and 
$$
		 \cK_{\fD}(t,z)=\frac{\fD(t,z)J\fD(t,z)^*-J}{z-\bar z}
		$$
		be the corresponding chain of kernels in $\bbC\bbD(2)$. Then there is a unique $t_{k,s}>0$ such that
		\begin{equation}\label{eq:27may-4}
	\hat \cK_{k/2^s}(z)=	 \cK_\fD(t_{k,s},z).
		\end{equation}
		Moreover $\frac{k}{2^s}$ corresponds to $M$-type of  $\fD(t_{k,s},z)$.
		\end{corollary}
		\begin{proof}
		Note that by Remark \ref{rem:1}, $\fB_{k/2^s}(z)$ is the limit of $J$-inner matrix functions and therefore also $J$-inner. The $\bbC\bbD(2)$-kernel $\hat\cK_{k/2^s}$ can be represented as the $J$-form of the corresponding $J$-expanding matrix-function
	\begin{equation}\label{eq:bks}
		\fB_{k/2^s}(z):=\cU_s\fB_k(z/\rho^s)\cU_s^{-1}.
		\end{equation}
		Due to the de Branges uniqueness theorem (see Subsection \ref{sec:cont} and Remark \ref{rem:intro}) there exists a unique normalized monotonic family $(\fD(t,z))_{t\geq 0}$ 
		generated by $m_{\sE}(z)$. Let $L=\lceil k/2^s\rceil$. 
		 Since by \eqref{eq:51}, $\fB_{L}(z)\fB_{L}(0)^{-1}$ belongs to this chain, monotonicity implies that this chain should contain the matrix $\fB_{k/2^s}(z)\fB_{k/2^s}(0)^{-1}$. That is, there exists a unique $t_{k,s}$
		such that
		$$
		\fD(t_{k,s},z)=\fB_{k/2^s}(z)\fB_{k/2^s}(0)^{-1},
		$$
		which is \eqref{eq:27may-4}.
		
		Integrating \eqref{eq:2-25} in $z$ and using $\theta_0(0)=0$ gives
		\[
		\theta_0(z) =  i \int \log \frac{x - z}{x} \,d\mu_{\sE_0}(x)
		\]
		Using $\theta(z) = \theta_0(F(z))$ and taking imaginary parts, we get
		\begin{align}\label{eq:Mtype}
		M(z) = \log|F(z)| - \int\log|x|d\mu_{\sE_0}(x) + \int\log\left |1-\frac x{F(z)}\right|d\mu_{\sE_0}(x).
		\end{align}
		That is, $F(z)$ is of $M$-type $1$. Respectively $p_k(F(z))$ is of $M$-type $k$. By the rescaling property $\theta(\rho^2 z)=4\theta(z)$ we obtain 
		that the matrix function $\fB_{k/2^s}(z)$ is of $M$-type $ k/2^s$. 
		\end{proof}
		This corollary shows the meaning of the kernel $\hat\cK_{k/2^s}$. Using de Branges' parametrization we have
		\[
		\mathscr C(m_{\sE})=\{\fD(t,\cdot)\mid t\in [0,\infty)\}.
		\]
		However, $\hat\cK_{k/2^s}$ provides another partial parametrization of the chain corresponding to the M-type of the matrix elements. In the following, we will extend this chain by continuity, to obtain a bijection between  $(\hat\cK_{\ell})_{\ell\geq 0}$ and $(\cK_{\fD(t,\cdot)})_{t\geq 0}$, showing that the natural parametrization of 	$\mathscr C(m_{\sE})$ in terms of the M-type is possible. 
		\subsection{An extension by continuity in $M$-type scale}
		
		\begin{proposition}\label{prop:28}
			Let $\ell\in(0,1)$ and
			$$
			\ell=\sum_{s\ge 1}\frac{\e_s}{2^s},\quad \e_s\in\{0,1\},
			$$
			be its dyadic expansion. Define the sequence
			$$
			k_n=\e_n+{\e_{n-1}}{2}+\dots+{\e_1}{2^{n-1}},
			$$
			i.e., $k_s/2^s\to \ell$ as $s\to\infty$.
			Then
			$$
			\lim_{n\to\infty}\hat\cK_{k_n/2^n}(z)=\cK_\fD(t_-(\ell),z)
			\le
			\liminf_{s\to\infty}
			\frac 1{|\rho|^{s}}\cU_{s}\cK_{k_s}(z/\rho^{s})\cU_{s}^*
			$$
			$$
			\le \limsup_{s\to\infty}
			\frac 1{|\rho|^{s}}\cU_{s}\cK_{k_s}(z/\rho^{s})\cU_{s}^*
			\le
			\lim_{n\to\infty}\hat\cK_{(k_n+1)/2^n}(z)=\cK_\fD(t_+(\ell),z),
			$$
			where $t_{\pm}(\ell)$ are the biggest/smallest values of $t$ such that the $M$-type of $\fD(t,z)$ is equal to the given $\ell$.
		\end{proposition}
		
		\begin{proof}
			Since $k_{s_0}2^n\le k_{s_0+n}$, we have
			$$
			\hat \cK_{k_{s_0}/2^{s_0}}(z)=
			\lim_{n\to\infty}
			\frac 1{|\rho|^{s_0+n}}\cU_{s_0+n}\cK_{k_{s_0} 2^n}(z/\rho^{{s_0}+n})\cU_{s_0+n}^*
			$$
			$$
			\le
			\liminf_{n\to\infty}
			\frac 1{|\rho|^{s_0+n}}\cU_{s_0+n}\cK_{k_{s_0+n}}(z/\rho^{s_0+n})\cU_{s_0+n}^*.
			$$
		Let us now show that 
			\[
			\lim_{n\to\infty}\hat\cK_{k_n/2^n}(z)=\cK_\fD(t_-(\ell),z).
			\]
			Let $t_{k_{s},s}$ be as in Corollary \ref{cor:2}. By Corollary \ref{cor:Appendix} $t_{k_{s},s}$ is monotonic increasing as a function of $s$ and thus so is $\cK_\fD(t_{k_{s},s},z)$. Moreover, $\cK_\fD(t_{k_{s},s},z)\leq \cK_\fD(t_-(\ell),z)$. Since $\cK_\fD(t_{k_{s},s},z)$ is increasing, the limit 
			$$\lim_{s\to\infty}\cK_\fD(t_{k_{s},s},z)=:\cK_\fD(t_0,z)
			$$
			 exists and belongs to the chain. Moreover, $\cK_\fD(t_0,z)\leq \cK_\fD(t_-(\ell),z)$. If $\cK_\fD(t_0,z)\neq \cK_\fD(t_-(\ell),z)$, then by definition of $t_-(\ell)$ the M-type of $\cK_\fD(t_0,z)$ is less then $\ell$. This contradicts that $k_s/2^s\to \ell$.
		
			On the other hand
			$$
			(k_{s_0}+1)2^n\ge k_{s_0} 2^n+2^{n-1}+\dots+1\ge k_{s_0+n}.
			$$
			Therefore
			$$
			\hat \cK_{(k_{s_0}+1)/2^{s_0}}(z)=
			\lim_{n\to\infty}
			\frac 1{|\rho|^{s_0+n}}\cU_{s_0+n}\cK_{(k_{s_0}+1) 2^n}(z/\rho^{{s_0}+n})\cU_{s_0+n}^*
			$$
			$$
			\ge
			\limsup_{n\to\infty}
			\frac 1{|\rho|^{s_0+n}}\cU_{s_0+n}\cK_{k_{s_0+n}}(z/\rho^{s_0+n})\cU_{s_0+n}^*,
			$$
			and the estimation from above is also proved. The proof of 	$\lim\hat\cK_{(k_n+1)/2^n}(z)=\cK_\fD(t_+(\ell),z)$ is the same as above.
		\end{proof}
		
		It remains to show that the upper and lower limits coincide. The proof will be given in two steps. First, in Lemma 	\ref{l:20dec-1}, we find an expression for a possible jump matrix \eqref{eq27nov-13}. Then, based on the Main Lemma \ref{l:20dec-2}, we show that this jump matrix-function is in fact  trivial, i.e., the left and right limits coincide indeed.
		
		\begin{lemma}\label{l:20dec-1} Let $|\rho|>\sqrt{13}-1$.
			If $ \cK_\fD(t_-(\ell),z)\not = \cK_\fD(t_+(\ell),z)$, then
			\begin{equation}\label{eq27nov-13}
				\fD(t_+(\ell),z)=
				\fD(t_-(\ell),z)
				(I-\cP
				\cD_\vk
				\cP^*J z)
			\end{equation}
			with a certain $\cP$ such that $\cP J \cP^*=0$ and $\cD_\vk\ge 0$. 
		\end{lemma} 
		Note that the second matrix in the product is $J$-expanding. For the proof, we will need to following lemma:

		\begin{lemma}\label{l25s-3-2}
			The following family of matrices 
			\begin{equation}\label{eq25s-2-2}
				\left\{\frac 1{|\rho|^{n/2}}
				\cU_n\fA_{k}(0)
				\right\}_{k\le 2^n}
			\end{equation}
			is precompact.			
		\end{lemma}

		\begin{proof}
			We use the well known CD-identity \eqref{intro:eq10} for orthogonal polynomials 
			$$
			\frac 1{|\rho|^n}\cU_n\cK_{2^n}(z/\rho^n)\cU_n^*
			$$
			$$
			=\frac 1{|\rho|^n}\cU_n\begin{pmatrix}
				\sum_{j=0}^{2^n-1}|q_j(z/\rho^n)|^2&\sum_{j=0}^{2^n-1}\overline{p_j(z/\rho^n)}q_j(z/\rho^n) \\
				\sum_{j=0}^{2^n-1}\overline{q_j(z/\rho^n)}p_j(z/\rho^n) \
				&
				\sum_{j=0}^{2^n-1}|p_j(z/\rho^n)|^2
			\end{pmatrix}
			\cU_n^*;
			$$
			Due to  \eqref{eq16_nov_1} the LHS has a limit $\hat\cK_1(z)$.
			Therefore the sum is uniformly bounded for an arbitrary  fixed $z\in \bbC$. In the origin we get that
			$$
			\frac 1{|\rho|^n}\cU_n
			\sum_{j=0}^{2^n-1}
			\begin{pmatrix}
				q_j(0)^2&{p_j(0)}q_j(0) \\
				{q_j(0)}p_j(0) 
				&
				p_j(0)^2
			\end{pmatrix}
			\cU_n^*\le C<\infty.
			$$
			Respectively each term in the sum is bounded
			$$
			\frac 1{|\rho|^n}\begin{pmatrix}
				(2/|\rho|)^{n/2}  q_{k}(0)\\
				(|\rho|/2)^{n/2} p_{k}(0)
			\end{pmatrix}\begin{pmatrix}
				(2/|\rho|)^{n/2}  q_{k}(0)&
				(|\rho|/2)^{n/2} q_{k}(0)
			\end{pmatrix}\le C.
			$$
			This implies precompactness of the family \eqref{eq25s-2-2}.
			
			In particular, from the (1,1) and (2,2) entry, we conclude that there exists $C>0$ so that for all $n\in\bbN$, $k<2^n$
			\[
			\frac{1}{|\rho|^{n/2}}\left(\frac{2}{|\rho|}\right)^{n/2}|q_{k}(0)|<C,\quad \frac{1}{|\rho|^{n/2}}\left(\frac{|\rho|}{2}\right)^{n/2}|p_{k}(0)|<C
			\]
			Since in addition $(a_k)$ are uniformly bounded  we get precompactness of the family \eqref{eq25s-2-2}.
		\end{proof}

			\begin{proof}[Proof of Lemma \ref{l:20dec-1}]  For $|\rho|>\sqrt{13}-1$  the recurrence coefficient $a_k$, $k\in\bbZ_+$, can be extended by continuity on the set of dyadic integers $\bbZ_2$ \cite{BellBeMou82}.
			For $\vk=\sum_{j=0}^\infty\e_j 2^j\in\bbZ_2$, $\e_j\in\{0,1\}$, the following limit exists
			$$
			a_{\vk}=\lim_{n\to\infty} a_{k_n},\quad k_n=\sum_{j=0}^n\e_j 2^j.
			$$
			Let $\cJ_\vk$ be the two-sided Jacobi matrix, formed by sequences
			$$
			a^\vk_n=a_{\vk+n},\quad b^\vk_n=-\rho/2, \quad\vk\in \bbZ_2,\ n\in\bbZ.
			$$
			Then  related $m_{\cJ_\vk}^\pm$ and $\hat U_\vk(z)$ are well defined and coincide with the limit of the corresponding
			$m_{\cJ,k_n}^\pm$ and $\hat U_{k_n}(z)$. We have $-\hat U_{\vk}\in\cN_0$ and
			\begin{align}\label{eq:56}
				\hat U_\vk(z)=-\sum_{c\in \cC}\frac{m_{\cJ_\vk}^{-}(F(c))}{F''(c)}\left(\frac 1{c-z}-\frac 1 c\right).
			\end{align}
			In particular, $\hat U_\vk(0)=0$.
			
			Since $t_{-}(\ell)\leq t_{+}(\ell)$, $\fD(t_-(\ell),z)$ is a divisor of $\fD(t_+(\ell),z)$. That is $\fD(t_-(\ell),z)^{-1}\fD(t_+(\ell),z)=:\fD_\Delta(z)$ is a $J$-expanding matrix function. We show, that $\fD_\Delta$ must necessarily be of the form \eqref{eq27nov-13}.  In the following, it will be more convenient to compute $\fD^{-1}_\Delta$. That is
			$$
			\fD_\Delta^{-1}(z):=\lim_{n\to\infty}
			\fB_{(k_n+1)/2^n}(0)
			\fB^{-1}_{(k_n+1)/2^n}(z)\fB_{k_n/2^n}(z)
			\fB^{-1}_{k_n/2^n}(0),
			$$
			where $\fB_{(k_n+1)/2^n}(z)$ and $\fB_{k_n/2^n}(z)$ are defined by \eqref{eq:bks}.
			Note that
			\begin{align}\label{eq:55}
				\cU_nJ\cU_n^*=(-1)^nJ.
			\end{align}
			To simplify notations in what follows we assume that $n$ is odd. We will frequently use that for this choice
			\[
			(-1)^n=-1\quad \text{and} \quad |\rho|^n=-\rho^n.
			\]
			First we note that 
			$$
			\fB_{(k_n+1)/2^n}(0)
			=\cU_n\fA_{k_n+1}(0)\cU_n^{-1}.
			$$
			Using \eqref{eq:55} and $\fA_{k_n+1}(0)J\fA_{k_n+1}(0)^{*}=J$ we get
			$$
			\fB_{k_n/2^n}(0)^{-1}
			=\cU_n\fA_{k_n}(0)^{-1}\cU_n^{-1}=
			\cU_n\begin{pmatrix}
				0&-1/a_{k_n+1}\\
				a_{k_n+1}& \rho/(2a_{k_n+1})
			\end{pmatrix}\fA_{k_n+1}(0)^{-1}\cU_n^{-1}
			$$
			$$
			=-
			\cU_n\begin{pmatrix}
				0&-1/a_{k_n+1}\\
				a_{k_n+1}& \rho/(2a_{k_n+1})
			\end{pmatrix}J\fA_{k_n+1}(0)^* J\cU_n^{-1}
			$$
			$$
			=-\cU_n
			\begin{pmatrix}
				1&0\\
				-\rho/2& 1
			\end{pmatrix}
			\begin{pmatrix}
				1/a_{k_n+1}&0\\
				0& a_{k_n+1}
			\end{pmatrix}\fA_{k_n+1}(0)^*\cU_n^*J
			$$
			Therefore, we compute
			$$
			\cU_n^{-1}\fB^{-1}_{(k_n+1)/2^n}(z)\fB_{k_n/2^n}(z)
			\cU_n
			\begin{pmatrix}
				1&0\\
				-\rho/2& 1
			\end{pmatrix}
			$$
			$$
			=\begin{pmatrix}
				1&0\\
				-\hat U_{k_n+1}(\frac z{\rho^n})&1
			\end{pmatrix}
			\begin{pmatrix}
				\frac{F(\frac z{\rho^n})+\frac\rho 2}{a_{k_n+1}}&\frac{F'(z/\rho^n)}{a_{k_n+1}}\\
				-\frac{a_{k_n+1}}{F'(z/\rho^n)}&
				0
			\end{pmatrix}
			\begin{pmatrix}
				1&0\\
				-\rho/2& 1
			\end{pmatrix}
			\begin{pmatrix}
				1
				&0\\
				\hat U_{k_n}(\frac z{\rho^n})&1\end{pmatrix}
			$$
			$$
			=\begin{pmatrix}
				1&0\\
				-\hat U_{k_n+1}(\frac z{\rho^n})&1
			\end{pmatrix}
			\begin{pmatrix}
				\frac{F(\frac z{\rho^n})+\frac\rho 2}{a_{k_n+1}}
				-\frac \rho 2
				\frac{F'(z/\rho^n)}{a_{k_n+1}}
				&\frac{F'(z/\rho^n)-1}{a_{k_n+1}}\\
				-\frac{a_{k_n+1}}{F'(z/\rho^n)}+a_{k_n+1}&
				0
			\end{pmatrix}
			\begin{pmatrix}
				1
				&0\\
				+\hat U_{k_n}(\frac z{\rho^n})&1\end{pmatrix}
			$$
			$$
			+\begin{pmatrix}
				1&0\\
				-\hat U_{k_n+1}(\frac z{\rho^n})&1
			\end{pmatrix}
			\begin{pmatrix}
				0&1/a_{k_n+1}\\
				-a_{k_n+1}&0
			\end{pmatrix}
			\begin{pmatrix}
				1
				&0\\
				\hat U_{k_n}(\frac z{\rho^n})&1\end{pmatrix}
			$$
			We can represent the last product as
			$$
			\begin{pmatrix}
				0&1/a_{k_n+1}\\
				-a_{k_n+1}&0
			\end{pmatrix}+\frac 1{a_{k_n+1}}
			\begin{pmatrix}
				-\hat U_{k_n}(z/\rho^n)&0\\
				-\hat U_{k_n}(z/\rho^n)\hat U_{k_n+1}(z/\rho^n)
				&\hat U_{k_n+1}(z/\rho^n)
			\end{pmatrix}
			$$
			
			For the first part we note that 
				\begin{align*}
					-&\cU_n\fA_{k_n+1}(0)\begin{pmatrix}
						0&1/a_{k_n+1}\\
						-a_{k_n+1}&0
					\end{pmatrix}
					\begin{pmatrix}
						1/a_{k_n+1}&0\\
						0& a_{k_n+1}
					\end{pmatrix}\fA_{k_n+1}(0)^*\cU_n^*J\\
					&=\cU_n\fA_{k_n+1}(0)J\fA_{k_n+1}(0)^*\cU_n^*J=\cU_nJ\cU_n^*J=-JJ=I.
			\end{align*}
			Let us now consider the remaining terms.
			Let $I_\vk\subset \bbZ$ be a subsequence such that $k_n\to\vk\in\bbZ_2$. By \eqref{eq:56}, $\hat U_{k_n}\to \hat U_\vk $ uniformly in a vicinicty of $0$ and we get that the following limits exist
			$$
			\lim_{k_n\to\vk}{\rho^n} \hat U_{k_n}(z/\rho^n)=\hat U_\vk'(0) z. 
			$$
			and
			$$
			\lim_{k_n\to\infty}\rho^n
			\begin{pmatrix}
				\frac{F(\frac z{\rho^n})+\frac\rho 2}{a_{k_n+1}}
				-\frac \rho 2
				\frac{F'(z/\rho^n)}{a_{k_n+1}}
				&\frac{F'(z/\rho^n)-1}{a_{k_n+1}}\\
				-\frac{a_{k_n+1}}{F'(z/\rho^n)}+a_{k_n+1}&
				0
			\end{pmatrix}=
			\begin{pmatrix}
				\frac{1-\rho/2 F''(0)}{a_{\vk+1}}
				&\frac {F''(0)}{a_{\vk+1}} \\
				a_{\vk+1} F''(0)&
				0
			\end{pmatrix} z
			$$
			Due to precompactness of the family \eqref{eq25s-2-2} we can choose a subsequence $I'_\vk$ in $I_\vk$ so that
			\begin{equation}\label{eq27nov-12}
				\lim_{n_j\to \infty}\frac 1{|\rho|^{n_j/2}}
				\cU_{n_j}\fA_{k_{n_j}}(0)=\cP.
			\end{equation}
Since $\det \cU_{n_j} = 1$ and $\det \fA_{k_{n_j}} = 1$, $\det \cP$ is evidently zero and the matrix is real. Therefore $\cP J\cP^*=0$.
			
			Therefore, along the subsequence $I_\vk'$ we obtain 
			$$
			\lim_{n\to\infty}
			\fB_{(k_n+1)/2^n}(0)
			\fB^{-1}_{(k_n+1)/2^n}(z)\fB_{k_n/2^n}(z)
			\fB^{-1}_{k_n/2^n}(0)=I+\cP
			\cD_\vk
			\cP^*J z
			$$
			where
			\begin{align*}
				\cD_\vk=&\begin{pmatrix}
					\frac{1-\rho/2 F''(0)+\hat U'_\vk(0)}{a^2_{\vk+1}}
					&{F''(0)} \\
					F''(0)&
					-\hat U'_{\vk+1}(0)
				\end{pmatrix}
				\\
				=&
				\begin{pmatrix}
					\frac 1{a^2_{\vk+1}}\left(
					\frac{F+\rho/2}{F'}+\hat U_\vk
					\right)'(0)
					&-\left(\frac 1{F'}\right)'(0) \\
					-\left(\frac 1{F'}\right)'(0)&
					-\hat U'_{\vk+1}(0)
				\end{pmatrix}\ge 0.
			\end{align*}
		Noting that $(I+\cP\cD_\vk\cP^*J z)^{-1}=I-\cP\cD_\vk\cP^*J z$, the lemma is proved.
		\end{proof}

		On the collection of $J$-expanding matrix-functions we define an involution
		$$
		\overleftarrow{\fA}(z)=\fj
		\fA(\bar z)^*\fj,\quad \fj=
		\begin{pmatrix}
			-1&0\\0&1
		\end{pmatrix}.
		$$
		Note that 
		\begin{align}\label{eq:57}
			\overleftarrow{\fA_k}(z)\star 0=-\frac{a_kp_{k-1}(z)}{p_k(z)}=m_{\mathcal J,k}^-(z),
		\end{align}
		which explains the notation. 
		
		Our goal is to show that $\cP=0$, which will be a consequence of the Main Lemma \ref{l:20dec-2} below. It's proof requires three preliminary lemmas. 
		\begin{lemma}\label{lem:8dec1}
			We define
			$$
			m_\ell^-(z)=\overleftarrow{\fD}(t_-(\ell),z)\star 0,\quad
			m_\ell^+(z)=\fD(t_-(\ell),z)^{-1}\star m_{\sE}(z).
			$$
			Then
			\begin{align}
				m_\ell^+(z)
				\label{eq29nov-16}
				=\lim_{s\to\infty}\cU_s\fA_{k_s}(0)
				\begin{pmatrix}
					1 & 0 \\  -\hat U_{k}(z/\rho^s) & 1
				\end{pmatrix} \star m_{\cJ,k_s}^+(F(z/\rho^s))F'(z/\rho^s)
			\end{align}
			and
			\begin{align}
				m_\ell^-(z)
				=
				\label{eq1dec-16}
				\lim_{s\to\infty}\cU_s^{-1}\overleftarrow \fA_{k_s}(0)^{-1}
				\left(\frac{m^-_{\cJ,k_s}(F(z/\rho^s))}{F'(z/\rho^s)}-\hat U_{k_s}(z/\rho^s)\right).
			\end{align}
			\end{lemma}
			\begin{proof}
					We have
				$$
				\fD(t_-(\ell),z)^{-1}\star m_{\sE}(z)=\lim_{s\to\infty}\fB_{k_s/2^s}(0)\fB_{k_s/2^s}(z)^{-1}\star m_{\sE}(z)
				$$
				Recall
				$$
				\fB_{k_s/2^s}(0)
				=\cU_s\fA_{k_s}(0)\cU_s^{-1}
				$$
				and that by \eqref{mrlim79}
					\[
					\cU_s^{-1}\star m_\sE(z)=\left(\frac{\rho}{2}\right)^sm_{\sE}(z)=m_\sE(z/\rho^s).
					\]
				Therefore
				$$
				\fB_{k/2^s}(z)^{-1}\star m_{\sE}(z)=\cU_s\fB_{k_s}(z/\rho^s)^{-1}\cU_s^{-1}\star m_\sE(z)=
				\cU_s\fB_{k_s}(z/\rho^s)^{-1}\star  m_\sE(z/\rho^s)
				$$
				and
				$$
				\fD(t_-(\ell),z)^{-1}\star m_\sE(z)=\lim_{s\to\infty}
				\cU_s\fA_{k_s}(0)\fB_{k_s}(z/\rho^s)^{-1}\star m_\sE(z/\rho^s).
				$$
				In its turn
				$$
				\fB_{k_s}(z/\rho^s)^{-1}\star m_\sE(z/\rho^s)
				$$
				$$
				=
				\begin{pmatrix}
					1 & 0 \\  -\hat U_{k}(z/\rho^s) & 1
				\end{pmatrix}
				\cF(z/\rho^s)^{-1}\fA_{k}(F(z/\rho^s))^{-1}\cF(z/\rho^s)\star m_\sE(z/\rho^s)
				$$
				By \eqref{eq7-1-10}
				and \eqref{eq:27may-3}
				we get
				$$
				\fB_{k_s}(z/\rho^s)^{-1}\star m_\sE(z/\rho^s)=
				\begin{pmatrix}
					1 & 0 \\  -\hat U_{k}(z/\rho^s) & 1
				\end{pmatrix} \star  m_{\cJ,k}^+(F(z/\rho^s))F'(z/\rho^s)
				$$
				This proves  \eqref{eq29nov-16}.
				
				At the same time
				$$
				\overleftarrow{\fD}(t_-(\ell),z)=\lim_{s\to\infty}\overleftarrow \fD(t_{k_s,s},z)=\lim_{s\to\infty}\overleftarrow\fB_{k_s/2^s}(0)^{-1}\overleftarrow\fB_{k_s/2^s}(z).
				$$
				Using \eqref{eq:57} we get
				$$
				\overleftarrow\fB_{k_s/2^s}(z)\star 0=\cU_s^{-1}
				\begin{pmatrix}
					1 &   -\hat U_{k_s}(z/\rho^s)\\
					0 & 1
				\end{pmatrix}\overleftarrow{\hat \fB}_{k_s}(z/\rho^s)\star 0
				$$
				$$
				=\cU_s\left(\frac{m_{\cJ,k}^-(F(z/\rho^s))}{F'(z/\rho^s)}-\hat U_{k_s}(z/\rho^s)
				\right)
				$$
				Therefore
				$$
				\overleftarrow{\fD}(t_-(\ell),z)\star 0=\lim_{s\to\infty}\cU_s^{-1}\overleftarrow \fA_{k_s}(0)^{-1}
				\star \left(\frac{m_{\cJ,k}^-(F(z/\rho^s))}{F'(z/\rho^s)}-\hat U_{k_s}(z/\rho^s)\right),
				$$
				which finishes the proof of \eqref{eq1dec-16}. 
			\end{proof}
		We introduce some notation, which will be convenient in the following.
			Recall \eqref{eq27nov-12}
		\begin{align} 
		\nonumber
			\cP:=\begin{pmatrix}
				\Psi_1& \Psi_2\\
				\Phi_1&\Phi_2
			\end{pmatrix}=&
			\lim_{s\to \infty}\frac 1{|\rho|^{s/2}}
			\cU_{s}\fA_{k_s}(0)
			\\
			=&\lim_{s\to \infty}\frac 1{|\rho|^{s/2}}
			\begin{pmatrix}
				\psi_{1,k_s}& \psi_{2,k_s}\\
				\phi_{1,k_s}&\phi_{2,k_s}
			\end{pmatrix}
		\end{align}
		which holds for odd $s$, where 
		\[
		\begin{pmatrix}
			\psi_{1,k_s}& \psi_{2,k_s}\\
			\phi_{1,k_s}&\phi_{2,k_s}
		\end{pmatrix}:=
		\cU_{s}\fA_{k_s}(0)=\begin{pmatrix}
			(2/|\rho|)^{s/2}& 0\\
			0&(|\rho|/2)^{s/2}
		\end{pmatrix}
		\begin{pmatrix}
			a_{k_s} q_{k_s-1}(0)& q_{k_s}(0)\\
			a_{k_s} p_{k_s-1}(0)&p_{k_s}(0)
		\end{pmatrix}.
		\]
		Moreover, we will use 
		
			\begin{align}\label{eq27nov-12v}\nonumber
				\fj J \cP J\fj=-\begin{pmatrix}
					\Phi_2& \Phi_1\\
					\Psi_2&\Psi_1
				\end{pmatrix}
				=&-\lim_{s\to \infty}\frac 1{|\rho|^{s/2}}
				\begin{pmatrix}
					\phi_{2,k_s}& \phi_{1,k_s}\\
					\psi_{2,k_s}&\psi_{1,k_s}
				\end{pmatrix}\\
				=&\lim_{s\to \infty}\frac 1{|\rho|^{s/2}}\cU_s^{-1}\overleftarrow \fA_{k_s}(0)^{-1}.
			\end{align}
			The last identity follows from 
	$$
	\cU_s^{-1}\overleftarrow \fA_{k_s}(0)^{-1}=\begin{pmatrix}
		(-1)(|\rho|/2)^{s/2}&0\\
		0&(|\rho|/2)^{-s/2}
	\end{pmatrix}\begin{pmatrix}
		p_{k_s}(0)& a_{k_s}p_{k_s-1}(0)\\
		-q_{k_s}(0)& -a_{k_s}q_{k_s-1}(0).
	\end{pmatrix}
	$$
	Using that $\det (\cU_s\fA_{k_s}(0))=-1$ we get the following triangular decompositions
	\begin{align}\label{eq:58}
		\cU_s\fA_{k_s}(0)
		=
		\begin{pmatrix}
			1 & \frac{ \psi_{2,k_s}}{\phi_{2,k_s}} \\ 0 & 1
		\end{pmatrix} \Lambda^+_{k_s}\cT^+_{k_s}
	\end{align}
	where 
	\[
	\Lambda^+_{k_s}:=\begin{pmatrix}
		-1/ {\phi_{2,k_s}}& 0\\
		0&  {\phi_{2,k_s}}
	\end{pmatrix},
	\quad
	\cT^+_{k_s}:=\begin{pmatrix}
		1& 0\\
		\frac{ \phi_{1,k_s}}{\phi_{2,k_s}}&1
	\end{pmatrix}
	\]
	and
	\begin{align}\label{eq:59}
		\cU_s^{-1}\overleftarrow \fA_{k_s}(0)^{-1}
		=\begin{pmatrix}
			1 & 0 \\ \frac{\psi_{2,k_s}}{\phi_{2,k_s}} & 1
		\end{pmatrix} \Lambda^-_{k_s}\cT^-_{k_s},
	\end{align}
	where
	$$
	\Lambda^-_{k_s}=\begin{pmatrix}
		-\phi_{2,k_s}& 0\\
		0& 1/\phi_{2,k_s}
	\end{pmatrix},
	\quad
	\cT_{k_s}^-=
	\begin{pmatrix}
		1&\frac{ \phi_{1,k_s}}{\phi_{2,k_s}}  \\
		0& 1
	\end{pmatrix}.
	$$
		\begin{lemma}
			Assume that $\Phi_2\not=0$.
			Then the limit \eqref{eq29nov-16} can be simplified to
			\begin{align}\label{eq4dec-18}
				-{\Phi_2^2}\hat U'_{\vk}(0)z &-\frac 1{m^+_\ell(z)-m_0^+}
				\\ \nonumber &={\Phi_2^2}\lim_{s\to\infty}\rho^s\left(-
				\frac{\phi_{1,k_s}}{\phi_{2,k_s}}-\frac{1}{{m_{\cJ,k_s}^+(F(z/\rho^s)F'(z/\rho^s)}}
				\right)
			\end{align}
			where 
			$
			m^+_0={\Psi_2}/{\Phi_2}.
			$
		\end{lemma}
		\begin{proof}
			Along the chosen subsequence of $s$ we have
			$$
			\lim_{s\to \infty} \frac{\psi_{2,k_s}}{\phi_{2,k_s}}=\frac{\Psi_2}{\Phi_2}=m^+_0.
			$$
			Further note that the triangular matrices
			$$
			\begin{pmatrix}
				1 & 0 \\  -\hat U_{k}(z/\rho^s) & 1
			\end{pmatrix} \quad \text{and}\quad \cT_{k_s}^+
			$$
			commute and
			$$
			\Lambda_{k_s}^+\begin{pmatrix}
				1 & 0 \\  -\hat U_{k}(z/\rho^s) & 1
			\end{pmatrix} =
			\begin{pmatrix}
				1 & 0 \\   \phi_{2,k_s}^2\hat U_{k}(z/\rho^s) & 1
			\end{pmatrix}  \Lambda_{k_s}^+.
			$$
			Moreover,
			\begin{equation}\label{eq6dec-1}
				\lim_{s\to \infty}\left( \frac 1{|\rho|^{s}} \phi_{2,k_s}^2\right)({|\rho|^{s}}\hat U_{k_s}(z/\rho^s))=-\Phi_2^2 \hat U'_{\vk}(0) z.
			\end{equation}
			Therefore, we get from \eqref{eq29nov-16} 
			\begin{align*}
				\begin{pmatrix}
					1 & 0 \\  \Phi_2^2 \hat U'_{\vk}(0) z & 1
				\end{pmatrix}
				\begin{pmatrix}
					1 & -m_{0}^+ \\  0 & 1
				\end{pmatrix}
				&\star m_\ell^+(z)\\
				&=\lim_{s\to\infty}\Lambda^+_{k_s}
				\cT_{k_s}^+
				\star m_{\cJ,k}^+(F(z/\rho^s))F'(z/\rho^s).
			\end{align*}
			
			This is exactly \eqref{eq4dec-18}.
		\end{proof}
		
		\begin{lemma}
			Let $\Phi_2\not=0$. Then the limit \eqref{eq1dec-16} can be simplified to 
			\begin{align}\label{eq6dec-2}
				\frac{m^-_\ell(z)}{-m_0^+ m^-_\ell(z)+1}+\Phi_2^2 \hat U'_\vk(0)z
				=\Phi_2^2\lim_{s\to\infty}
				\rho^{s}
				\left(\frac{m_{\cJ,k}^-(F(z/\rho^s))}{F'(z/\rho^s)}+ \frac{ \phi_{1,k_s}}{\phi_{2,k_s}} \right).
			\end{align}
		\end{lemma}
		
		\begin{proof}
			This time we note that the matrices
			$$
			\begin{pmatrix}
				1 & -\hat U_{k}(z/\rho^s)\\0 & 1
			\end{pmatrix} \quad \text{and}\quad \cT_{k_s}^-
			$$
			commute and
			$$
			\Lambda_{k_s}^-\begin{pmatrix}
				1 &  -\hat U_{k}(z/\rho^s) \\ 0& 1
			\end{pmatrix} =
			\begin{pmatrix}
				1 &  \phi_{2,k_s}^2\hat U_{k}(z/\rho^s)\\ 0 & 1
			\end{pmatrix}  \Lambda_{k_s}^-.
			$$
			Using \eqref{eq6dec-1}, we obtain
			$$
			\begin{pmatrix}
				1 &  \Phi_2^2 \hat U'_\vk(0)z\\ 0 & 1
			\end{pmatrix} \begin{pmatrix}
				1 & 0 \\ -m_0^+& 1
			\end{pmatrix} \star m_\ell^-(z)=\Phi_2^2\lim_{s\to\infty}
			\rho^{s}
			\left(\frac{m_{\cJ,k}^-(F(z/\rho^s))}{F'(z/\rho^s)}+\frac{\phi_{1,k_s}}{\phi_{2,k_s}}\right).
			$$
			Thus \eqref{eq6dec-2} is proved.
		\end{proof}
		
		\begin{lemma}[Main Lemma]\label{l:20dec-2}
		For $\Phi_2\not=0$ along the suitable subsequence $I'_\vk$ the following limit holds
		\begin{align}\label{eq6dec-3}
			\frac{1}{\Phi_2^2}\lim_{s\to\infty}R_{k_s}(F(z/\rho^s))F'(z/\rho^s)/\rho^s
			=\left(\frac{m^-_\ell(z)}{1-m_0^+ m^-_\ell(z)}
			-\frac 1{m^+_\ell(z)-m_0^+}\right)^{-1}
		\end{align}
		where $R_{k_s}(z)$ is the  $k_s$-th diagonal entry of the resolvent matrix $(\cJ-z)^{-1}$.
	\end{lemma}
\begin{proof}
We recall that the resolvent diagonal entry can be represented in terms of one-sided resolvent functions $m_{\cJ,k}^{\pm}$ \eqref{eq:27may-2}.
By adding \eqref{eq4dec-18} and \eqref{eq6dec-2} we obtain \eqref{eq6dec-3}.
\end{proof}

\begin{theorem}\label{thm:417} Let $|\rho|>\sqrt{13}-1$.
Then matrix $\cP$ in \eqref{eq27nov-12}  is trivial. 	In other words 
$$
 \cK_\fD(t_-(\ell),z) = \cK_\fD(t_+(\ell),z)
 $$ for all $\ell>0$.
\end{theorem}
\begin{proof}

For an arbitrary $\vk\in\bbZ_2$, we define $\mu_\vk$ by
$$
R_\vk(z)=\langle(\cJ_\vk-z)^{-1}e_0,e_0\rangle=\int \frac{d\mu_\vk(x)}{x-z}.
$$
We consider the function $R_\vk(F(z/\rho^s))F'(z/\rho^s)/\rho^s$. By \eqref{eq6dec-3} it is a Herglotz function. Moreover, for $z\to\infty$, we have 
\[
R_\vk(F(z/\rho^s))F'(z/\rho^s)\sim \frac{-F'(z/\rho^s)}{F(z/\rho^s)}\sim z^\beta,
\]
for $\beta<0$ and we conclude again by \cite[Proposition 7.33]{LukicFirstCourse}  that
\[
R_\vk(F(z/\rho^s))F'(z/\rho^s)/\rho^s=\int\frac{d\nu_{\vk,s}(y)}{y-z}
\]
for some measure $\nu_{\vk,s}$. We claim that 
$$
\nu_{\vk,s}((y_1,y_2))=\mu_\vk((F(y_1/\rho^s),F(y_2/\rho^s))), \quad (y_1,y_2)\subset [\rho^s\sb^F_{j}, \rho^s\sa^F_{j+1}].
$$
Choose, $y_1,y_2$ as above and assume that $\nu_{\vk,s}$ has no mass points at those points. Then, by Stieltjes inversion we have 
\[
\nu_{\vk,s}((y_1,y_2))=\lim\limits_{\epsilon\to 0}\int_{y_1+i\epsilon}^{y_2+i\epsilon}\Im (R_\vk(F(z/\rho^s))F'(z/\rho^s)/\rho^s) dz.
\]
On $[\rho^s\sb^F_{j}, \rho^s\sa^F_{j+1}]$, $u=F(z/\rho^s)$ is monotonic and without loss of generality we consider an interval where it is increasing. Then, we get
\begin{align*}
\nu_{\vk,s}((y_1,y_2))&=\lim\limits_{\epsilon\to 0}\int_{F((y_1+i\epsilon)/\rho^s)}^{F((y_2+i\epsilon)/\rho^s)}\Im(R_\vk(u))du\\
&=\mu_\vk((F(y_1/\rho^s),F(y_2/\rho^s)))
\end{align*}
where we used the Stieljes inversion formula for the Herglotz function $R_\vk(u)$ in the last step.

It is shown in \cite{EichLukYu3}, that for $|\rho|>\sqrt{13}-1$  the spectrum of $\cJ_\vk$ is (singular) continuous for an arbitrary $\vk\in\bbZ_2$. Therefore for an arbitrary $\ve>0$ there exists $\delta>0$ such that
$$
\mu_\vk((-\delta,\delta))\le \ve.
$$
We choose a convergent subsequence $k_s\to\vk$ and therefore $\mu_{k_s}\to \mu_{\vk}$ weakly.
Then for an arbitrary fixed $A>0$ for all sufficiently big $s$ we have $F(-A/\rho^s,A/\rho^s)\subseteq (-\delta,\delta)$ and thus
$$
\nu_{k_s,s}(-A,A)=\mu_{k_s}(F(-A/\rho^s,A/\rho^s))\le 2\ve.
$$
Thus, $\nu_{k_s,s}$ converges to $0$ and hence the function on the LHS of \eqref{eq6dec-3} converges to $az+b$ uniformly on compact subsets of $\bbC_+$ for some $a\geq 0,b\in\bbR$. On the other hand, the RHS is not a linear function. The contradiction  deals with our assumption $\Phi_2\not=0$ (the value in the denominator in the LHS). We conclude that $\Phi_2=0$.

Recall that by \eqref{eq27nov-12}
\[
\Phi_2=\lim_{s\to\infty}\frac{p_{k_s}(0)}{2^{s/2}}.
\]
Thus, applying the same arguments for the subsequence of $\{k_s-1\}_{k_s\in I_\vk}$ yields
\[
\lim_{s\to\infty}\frac{p_{k_s-1}(0)}{2^{s/2}}=0
\] 
and therefore since $a_{k_s}$ are uniformly bounded
\[
\Phi_1=\lim_{s\to\infty}\frac{a_{k_s}p_{k_s-1}(0)}{2^{s/2}}=0.
\]

To show that $\Psi_2=0$,
we multiply both parts of \eqref{eq29nov-16} and \eqref{eq1dec-16} by $J$.
We get
\begin{align*}\nonumber
	-\frac 1{ m_\ell^+(z)}
	\label{}
	=\lim_{s\to\infty}J\cU_s\fA_{k_s}(0)
	\begin{pmatrix}
		1 & 0 \\  -\hat U_{k}(z/\rho^s) & 1
	\end{pmatrix} \star m_{\cJ,k_s}^+(F(z/\rho^s)F'(z/\rho^s)
\end{align*}
and
\begin{align*}
	- \frac 1{m_\ell^-(z)}
	=
	\label{}
	\lim_{s\to\infty}J\cU_s^{-1}\overleftarrow \fA_{k_s}(0)^{-1}\star
	\left(\frac{m^-_{\cJ,k_s}(F(z/\rho^s))}{F'(z/\rho^s)}-\hat U_{k_s}(z/\rho^s)\right).
\end{align*}
Assume that $\Psi_2\not=0$. Note that multiplication by $J$ leads to a permutation of the upper and lower row of the matrices $\cU_s\fA_{k_s}(0)$ and 
$\cU_s\overleftarrow  \fA_{k_s}(0)$. So, repeating the arguments, we get the same identity \eqref{eq6dec-3} with the replacements
$\Phi_2\mapsto -\Psi_2$ and $m_\ell^\pm\mapsto -1/m_\ell^\pm$, i.e.,
\begin{align}\label{eq19dec-3}\nonumber
	\frac{1}{\Psi_2^2}\lim_{s\to\infty}R_{k_s}(F(z/\rho^s))F'(z/\rho^s)/\rho^s\\
	=\left(\frac{1}{m_0^-- m^-_\ell(z)}
	+\frac {m_\ell^+(z)}{m^+_\ell(z)-m_0^-}\right)^{-1},\quad m^-_0=-\frac{\Phi_2}{\Psi_2}.
\end{align}
This time we get a contradiction with the assumption $\Psi_2\not=0$. The same arguments that shows that $\Phi_1=0$ proves that $\Psi_1=0$.
\end{proof}

\begin{proof}[Proof of Theorem \ref{intro:thm3}] Note that $\ell(t)$ is monotonic. Since for every $k/2^s$ there is $t_{k,s}$ such that $M$-type of
$\fD(t_{k,s},z)$ equals $k/2^s$, $\ell(t)$ is continuous. Theorem \ref{thm:417} implies that $t_-(\ell)=t_+(\ell)$. That is, $t_1<t_2$ implies $\ell(t_1)<\ell(t_2)$.
\end{proof}

\subsection{Proof of Theorem \ref{intro:thm1}}

\begin{proof}[Proof of Theorem \ref{intro:thm1}, part (a)]
Evaluating \eqref{eq16_nov_1} for $z=0$ we get 
\[
\lim_{k\to\infty}\frac{r(2^k,0)}{|\rho|^k}= c(1).
\]
Since $\kappa=\log|\rho|/\log 2$ we have for $n=2^k$ we have that
\begin{align}\label{eq:60}
|\rho|^k=(2^k)^{\kappa}.
\end{align}
Choose $k(n)$ so that 
\[
2^{k(n)-1}\leq n<2^{k(n)}.
\]
Hence, by monotonicity of $r(\cdot,0)$ we get that
\[
r(2^{k(n)-1},0)\leq r(n,0)<r(2^{k(n)},0)
\]
On the other hand, by \eqref{eq:60} we have 
\[
\frac{1}{|\rho|^{k(n)}}\leq \frac1{n^{\kappa}}\leq \frac{1}{|\rho|^{k(n)-1}}
\]
Thus,
\[
\frac{1}{|\rho|}\frac{r(2^{k(n)-1},0)}{|\rho|^{k(n)-1}}=\frac{r(2^{k(n)-1},0)}{|\rho|^{k(n)}}\leq \frac{r(n,0)}{n^{\kappa}}\leq\frac{r(2^{k(n)},0)}{|\rho|^{k(n)-1}}=\frac{r(2^{k(n)},0)}{|\rho|^{k(n)}}|\rho|
\]
Hence, 
\[
\frac{c(1)}{|\rho|}=\lim_{n\to\infty}\frac{r(2^{k(n)-1},0)}{|\rho|^{k(n)}}\leq \liminf_{n\to\infty}\frac{r(n,0)}{n^{\kappa}}.
\]
and
\[
\limsup_{n\to\infty}\frac{r(n,0)}{n^{\kappa}}\leq \lim_{n\to\infty}\frac{r(2^{k(n)},0)}{|\rho|^{k(n)-1}}=c(1)|\rho|.
\]
The proof for $K(n,0,0)$ is analogous.
\end{proof}

In what follows we asssume that $|\rho|>\sqrt{13}-1$.
Due to continuity, we can extend Proposition \ref{prop:28} to a more general setting.

\begin{lemma}
Let a subsequence  $(n_s)$ be such that $n_s/2^s\to\ell$. Then
\begin{align}\label{eqComp:28-1}
	\frac{1}{|\rho|^s}\cU_s\cK_{n_s}(z/\rho^s)\cU_s^*\to \cK_\fD(t(\ell),z).
\end{align}
\end{lemma}
\begin{proof}
It suffices to consider $\ell$ which are not of the form $\frac{k}{2^n}$ for some $k,n$. Then, for $k_s$ as defined in Proposition \ref{prop:28} we have $k_s/2^s<\ell$. Thus, for a fixed $s$ and a sufficiently big $j$ we have
$$
\frac{k_s}{2^s}<\frac{n_{j+s}}{2^{j+s}},
$$
that is, $k_s 2^j<n_{s+j}$. Therefore
$$
			\frac 1{|\rho|^{s+j}}\cU_{s+j}\cK_{k_{s} 2^j}(z/\rho^{s+j})\cU_{s+j}^*
			\le
			\frac 1{|\rho|^{s+j}}\cU_{s+j}\cK_{n_{s+j}}(z/\rho^{s+j})\cU_{s+j}^*.
			$$
			By Theorem \ref{thnov22-1} the LHS has a limit as $j\to\infty$. Thus, we get
$$
			\hat \cK_{k_s/2^s}(z)
			\le \liminf_{j\to\infty}
			\frac 1{|\rho|^{s+j}}\cU_{s+j}\cK_{n_{s+j}}(z/\rho^{s+j})\cU_{s+j}^*.
			$$
Since $s$ is arbitrary we get
$$
			\cK_\fD(t(\ell),z)
			\le \liminf_{j\to\infty}
			\frac 1{|\rho|^{s+j}}\cU_{s+j}\cK_{n_{s+j}}(z/\rho^{s+j})\cU_{s+j}^*.
			$$
Now, for a fixed $s$ and sufficiently big $j$ we have
$(k_s+1)/{2^s}\ge{n_{j+s}}/{2^{j+s}}$. 
 Therefore
$$
			\frac 1{|\rho|^{s+j}}\cU_{s+j}\cK_{(k_{s}+1) 2^j}(z/\rho^{s+j})\cU_{s+j}^*
			\ge
			\frac 1{|\rho|^{s+j}}\cU_{s+j}\cK_{n_{s+j}}(z/\rho^{s+j})\cU_{s+j}^*.
			$$
		Hence, we obtain
$$
			\hat \cK_{(k_s+1)/2^s}(z)\ge \limsup_{j\to\infty}
			\frac 1{|\rho|^{s+j}}\cU_{s+j}\cK_{n_{s+j}}(z/\rho^{s+j})\cU_{s+j}^*,
			$$
			and therefore by continuity in $\ell$
			$$
			\cK_\fD(t(\ell),z)
			\ge \limsup_{j\to\infty}
			\frac 1{|\rho|^{s+j}}\cU_{s+j}\cK_{n_{s+j}}(z/\rho^{s+j})\cU_{s+j}^*.
			$$
			Therefore, we get \eqref{eqComp:28-1}.
\end{proof}

We require one more lemma. 
\begin{lemma}\label{lem:9}
	The function $m_\sE$ has the invariance property
	\[
	|\rho|^\beta m_{\sE}(z/|\rho|)=-m_{\sE}(-z).
	\]
\end{lemma}
\begin{proof}
	Since $|\rho|=-\rho$, it follows from \eqref{eq7-1-10} that
	\[
	m_\sE(z/|\rho|)=\int\frac{F'(-z/\rho)}{x-F(-z/\rho)}d\mu_{\sE_0}(x).
	\] 
	Using the invariant property $\cL_0^*\mu_{\sE_0}=2\mu_{\sE_0}$ of the measure, we get from \eqref{eq:2}
	\[
	m_\sE(z/|\rho|)=
	\frac 1 2 \int\frac{F'(-z/\rho)T'(F(-z/\rho))}{x-T(F(-z/\rho))}d\mu_{\sE_0}(x).
	\]
	Since
	\[
	F(z)=T(F(z/\rho)), \quad F'(z)= \frac 1\rho T'(F(z/\rho))F'(z/\rho),
	\]
	we obtain
	\[
	m_\sE(z/|\rho|)=
	\frac{\rho} 2 \int\frac{F'(-z)}{x-F(-z)}d\mu_{\sE_0}(x)=-|\rho|^{-\beta} m_{\sE}(-z). \qedhere
	\]
\end{proof}
\begin{proof}[Proof of Theorem \ref{intro:thm1}, part (b)]
In particular, \eqref{eqComp:28-1} contains the limit
$$
\frac 1{2^s}(\cK_{n_s})_{22}(0)\to (\cK_\fD)_{22}(t(\ell),0).
$$
Therefore, we have
$$
\lim_{s\to\infty}\frac 1{n_s}(\cK_{n_s})_{22}(0)=\lim_{s\to\infty}\frac{2^s}{n_s}\frac 1{2^s}(\cK_{n_s})_{22}(0)= \frac 1 \ell(\cK_\fD)_{22}(t(\ell),0).
$$

Since $\det\cU_s=\pm 1$, \eqref{eqComp:28-1} implies
$$
\frac 1{|\rho|^s}\sqrt{\det\cK_{n_s}(0)}\to \sqrt{\det\cK_\fD(t(\ell),0)}.
$$
Since $(2^s)^\kappa=|\rho|^s$, we obtain
$$
\lim_{s\to\infty}\frac 1{n^\kappa_s}\sqrt{\det\cK_{n_s}(0)}=\lim_{s\to\infty}
\left(\frac{2^s}{n_s}\right)^\kappa\frac 1{|\rho|^s}\sqrt{\det\cK_{n_s}(0)}=\frac 1{\ell^\kappa} \sqrt{\det\cK_\fD(t(\ell),0)}.
$$

Similarly the relation for the product of diagonal entries
$$
\frac 1{|\rho|^s}\sqrt{(\cK_{n_s})_{11}(0)(\cK_{n_s})_{22}(0)}\to\sqrt{(\cK_\fD)_{11}(t(\ell),0)(\cK_\fD)_{22}(t(\ell),0)}
$$
implies $r(n_s,0)\sim {c(\ell)}n_s^{\kappa}$ with $c(\ell)$ given by \eqref{eq:28may-1}. By compactness, \eqref{eqnrnoscillations} follows, and \eqref{eqnKnoscillations} is proved analogously. 

It follows from Lemma \ref{lem:9} and Lemma \ref{lem:18} that
\[
\frac{1}{|\rho|}\fj\cUnew(|\rho|^\beta)\cK_\fD(t(2\ell),z/\rho)\cUnew(|\rho|^\beta)^*\fj=\cK_\fD(t(\ell),z)
\]
for all $\ell>0$, which implies $b(2\ell) = b(\ell)$ and $c(2\ell) = c(\ell)$. 
\end{proof}

\appendix	
	
\section{Monotonicity of $M$-type}
For the reader's convenience we provide a proof of a well known lemma.
\begin{lemma}
Assume that
\begin{equation}\label{eq:30may-1}
\lim_{y\to\infty}\frac{\log y}{M(iy)}=0.
\end{equation}
Then
\begin{equation}\label{eq:30may-2}
\limsup_{y\to\infty}\frac{\log\|\fA(iy)\|}{M(iy)}=\limsup_{y\to\infty}\frac{\log\|\fA(iy)J\fA(iy)^*-J\|}{2M(iy)}.
\end{equation}
\end{lemma}
\begin{proof}
We use the following change of variables
$$
A(\z)=V^*\fA(z)V,\quad V=\frac 1{\sqrt{2}}
\begin{pmatrix}i&-i\\ 1&1
\end{pmatrix},
\quad
z=i\frac{1+\z}{1-\z}.
$$
In this case
$$
\|\fA(z)\|=\|A(\z)\|\quad\text{and}\quad
\|\fA(z)J\fA(z)^*-J\|=\|A(\z)\fj A(\z)^*-\fj\|.
$$
The matrix function
$$
A(\z)=\begin{pmatrix} a_{11}(\z)&a_{12}(\z)
\\ a_{21}(\z)&a_{22}(\z)
\end{pmatrix},
$$
is $\fj$-expanding in $\bbD$,
\begin{equation}\label{eq:30may-3}
\Gamma(\z)=
\begin{pmatrix} \g_{11}(\z)&\g_{12}(\z)
\\ \g_{21}(\z)&\g_{22}(\z)
\end{pmatrix}=
A(\z)\fj A(\z)^*-\fj\ge 0.
\end{equation}
 WLOG we assume that $A(\z)$ meets the Arov normalization $a_{21}(0)=0$ (the Arov gauge \cite{BessLukYu}).
 
 Diagonal entries of the inequality \eqref{eq:30may-3} provides the following estimations
\begin{align*}
 -|a_{11}(\z)|^2 +|a_{12}(\z)|^2+1&\ge 0\\
  -|a_{21}(\z)|^2 +|a_{22}(\z)|^2-1&\ge 0.
\end{align*}
In addition \eqref{eq:30may-3} and $\det A(\z)\not=0$ implies $A(\z)^*\fj A(\z)-\fj\ge 0$ and therefore
$$
  -|a_{12}(\z)|^2 +|a_{22}(\z)|^2-1\ge 0.
$$
Thus
\begin{align*}
 |a_{11}(\z)|^2\le |a_{12}(\z)|^2+1&\le |a_{22}(\z)|^2\\
 |a_{21}(\z)|^2\le |a_{22}(\z)|^2-1&\le |a_{22}(\z)|^2\\
  |a_{12}(\z)|^2\le |a_{22}(\z)|^2-1&\le |a_{22}(\z)|^2.
\end{align*}
We have
\begin{align}\nonumber
\|A(\z)\|^2=&\|A(\z)^*A(\z)\|\le \tr(A(\z)^*A(\z))\\
\label{eq:30may-5}
=&
|a_{11}(\z)|^2 +|a_{12}(\z)|^2+
|a_{21}(\z)|^2 +|a_{22}(\z)|^2\le 4 |a_{22}(\z)|^2.
\end{align}

On the other hand
$$
\g_{22}(\z)+1=|a_{22}(\z)|^2-|a_{21}(\z)|^2=|a_{22}(\z)|^2(1-|s(\z)|^2),\ \ s(\z)=\frac{a_{21}(\z)}{a_{22}(\z)}.
$$
The function $s(\z)$ belongs to the Schur class. Since $s(0)=0$, by Schwarz lemma 
$|s(\z)|^2\le|\z|^2$.
Thus 
$$
\log|a_{22}(\z)|^2+\log(1-|\z|^2)\le\log(\g_{22}(\z)+1).
$$
On the imaginary axis $z=iy$ we have $\z=|\z|$, consequently 
$$
\frac{1}{1-|\z|^2}\le\frac{(1+|\z|)^2}{1-|\z|^2}=y.
$$
Therefore
$$
\log|a_{22}(\z)|^2\le\log(\g_{22}(\z)+1)+\log y.
$$
In combination with
\eqref{eq:30may-5} we obtain
\begin{align}\nonumber
\limsup_{y\to\infty}\frac{\log\|\fA(iy)\|}{M(iy)}=
\limsup_{y\to\infty}\frac{\log\|A(\z)\|}{M(iy)}
\le \limsup_{y\to\infty}\frac{\log 4|a_{22}(\z)|^2}{2M(iy)}\\
\label{eq:30may-6}
\le \limsup_{y\to\infty}\frac{\log 4y+\log (\|\Gamma(\z)\|+1)}{2M(iy)}=
 \limsup_{y\to\infty}\frac{\log \|\Gamma(\z)\|}{2M(iy)}.
\end{align}

Since
$$
\|\Gamma(\z)\|\le 1+\|A(\z)\|^2
$$
the opposite estimation
$$
 \limsup_{y\to\infty}\frac{\log \|\Gamma(\z)\|}{2M(iy)}\le \limsup_{y\to\infty}\frac{\log\|A(\z)\|}{M(iy)}
$$
is evident.
\end{proof}

\begin{corollary}\label{cor:Appendix}
For a monotonic chain 
$\mathscr C=\{\fA(t,\cdot)\mid t\in [0,\infty)\}$
 the $M$-type $\ell(t)$, whenever it is well defined, is monotonic.
\end{corollary}
\begin{proof}
Monotonicity of the chain means that
$$
\frac{\fA(z,t_1)J\fA(z,t_1)^*-J}{z-\overline{z}}\le \frac{\fA(z,t_2)J\fA(z,t_2)^*-J}{z-\overline{z}},\quad t_1<t_2.
$$
Therefore
$$
\|\fA(z,t_1)J\fA(z,t_1)^*-J\| \le 
\|\fA(z,t_2)J\fA(z,t_2)^*-J\|,\quad t_1<t_2.
$$
Then \eqref{eq:30may-2} implies $\ell(t_1)\le\ell(t_2)$.
\end{proof}

\begin{remark}
The norm of an individual $\fj$-expanding matrix $A$ can not be estimated in terms of its $j$-form $\Gamma$, since $\Gamma$ does not depend on a possible $\fj$-unitary multiplier of $A$ from  the right. Note that the norm of a $\fj$-unitary matrix can be arbitrary large. However, as soon as we have a matrix function $A(\z)$ it is enough to normalize it at a single point, say $\z=0$, to have an estimation $\|A(\z)\|$ by $\|\Gamma(\z)\|$ in an arbitrary point $\z\in\bbD$. To the purpose of such normalization,  V.P. Potapov introduce a concept of $\fj$-modulus. The Arov normalization is much simpler and provides similar or better estimations. It appeared naturally in the theory of unitary extensions of isometries. For details see the appendices in \cite{DamEichYud21}.
\end{remark}

		\bibliographystyle{amsplain}

\providecommand{\MR}[1]{}
\providecommand{\bysame}{\leavevmode\hbox to3em{\hrulefill}\thinspace}
\providecommand{\MR}{\relax\ifhmode\unskip\space\fi MR }
\providecommand{\MRhref}[2]{%
  \href{http://www.ams.org/mathscinet-getitem?mr=#1}{#2}
}
\providecommand{\href}[2]{#2}

	\end{document}